\documentclass[12pt,leqno]{amsart}

\usepackage{amsmath,amsthm,amscd,amssymb,eucal,mathrsfs,xspace}
\usepackage[PostScript=dvips,balance,midshaft,nohug]{diagrams}
\usepackage{enumitem,calc,color}
\usepackage[pagebackref]{hyperref}
\swapnumbers
\newarrow{Equals} =====
\newarrow{Implies} ===={=>}
\newarrow{Onto} ----{>>}
\newarrow{Into} C--->
\newarrow{IntoA} {hooka}--->
\newarrow{IntoB} {hookb}--->
\newarrow{Dotsto} ....>
\newarrowhead{sp}{\ }{\ }{\ }{\ }
\newarrowtail{sp}{\ }{\ }{\ }{\ }
\newarrow{Line}----{-}
\newarrow{Eq}{sp}==={sp}
\newtheorem{thm}[subsection]{Theorem}
\newtheorem*{thm*}{Theorem}
\newtheorem{cor}[subsection]{Corollary}
\newtheorem{lem}[subsection]{Lemma}
\newtheorem*{lem*}{Lemma}
\newtheorem{prop}[subsection]{Proposition}
\newtheorem*{prop*}{Proposition}

\theoremstyle{definition}
\newtheorem{defn}[subsection]{Definition}
\numberwithin{equation}{subsection}

\DeclareMathOperator*{\tnull}{}
\newcommand{\tr}{\sideset{^{t{}}}{}\tnull\nolimits\hskip-3.5pt}

\renewcommand\mod[1]{\ (\mathop{\rm mod}#1)}
\newcommand{\quash}[1]{}

\newcommand{\Diag}{\text{Diag}}
\newcommand{\Sp}{{\rm{Sp}}}
\newcommand{\FF}{\mathbb F}
\newcommand{\CC}{\mathbb C}
\newcommand{\RR}{\mathbb R}\newcommand{\QQ}{\mathbb Q}

\newcommand{\ZZ}{\mathbb Z}
\newcommand{\pp}{\mathfrak p}

\newcommand{\gsp}{\mathop{\textstyle\rm GSp}}
\newcommand{\Gsp}{\mathop{\textstyle\rm GSp}}
\newcommand{\GL}{\textstyle\mathop{\rm GL}}\newcommand{\GSp}{\textstyle\mathop{\rm GSp}}
\newcommand{\sGL}{\scriptstyle\mathop{\rm GL}}

\newcommand{\Hom}{\mathop{\rm Hom}}\newcommand{\hHom}{{\rm Hom}}
\newcommand{\End}{\mathop{\rm End}}

\newcommand{\Gal}{\mathop{\rm Gal}}

\renewcommand{\Diag}{\mathop{\rm diag}}

\newcommand{\rank}{\mathop{\rm rank}}
\newcommand{\contains}{\supset}

\renewcommand{\AA}{\mathbb A}
\newcommand{\vol}{\mathop{\rm vol}}

\newcommand{\val}{\mathop{\rm val}}
\newcommand{\Tr}{\mathop{\rm Trace}}
\newcommand{\tTr}{{\rm Trace}}

\newcommand{\diag}{{\rm diag}}
\newcommand{\GO}{{\rm GO}}\newcommand{\sig}{{\rm sig}}

\newcommand{\B}{\omega}
\renewcommand{\pp}{\prime\prime}

\newcommand{\ToFrom}{{\ensuremath\raisebox{-1ex}{$\overset{\textstyle\hookrightarrow}{\twoheadleftarrow}$ }}}

\newcommand{\taub}{\bar\tau}
\renewcommand{\k}{{\Bbbk}}

\renewcommand{\pp}{\prime\prime}
\newcommand{\qreal}{$q$-inversive\xspace}

\newcounter{remarkscounter}

\begin{document}
\today
\title{Real structures on ordinary abelian varieties}
\author{Mark Goresky}
\author{Yung sheng Tai}
\maketitle
\setcounter{tocdepth}{1}
{\footnotesize\tableofcontents}

\makeatletter
\quash{
\renewcommand{\subsection}{\@startsection
{subsection}{2}{0mm}
{-\baselineskip} 
{0.2\baselineskip}%
{\bfseries\normalsize}}
}

\renewcommand{\section}{\@startsection
{section}{1}{0mm}
{-\baselineskip}
{.5\baselineskip}
{\bf\sffamily\centering}}
\makeatother

\section{Introduction}

\subsection{}\label{subsec-intro1}
The moduli space $X_{\CC}$ of $n$-dimensional principally polarized complex 
Abelian varieties (with a level structure) is the quotient
\[ X_{\CC} = \Gamma \backslash G(\RR) /K\]
where $G(\RR) = \Sp_{2n}(\RR)$, $K = U_n$ is a maximal compact subgroup, and
$\Gamma = \Sp_{2n}(\ZZ)$ (or a congruence subgroup thereof).  This space has the 
structure of a quasi-projective complex algebraic
variety and it can be defined over a certain number field $F$.  With appropriate 
choices, the reduction of $X$ modulo a prime ideal $\mathfrak p$ can be 
interpreted as the moduli space $X_k$ of principally polarized Abelian 
varieties (with level structure) over the finite field 
$k = \mathcal O_F/\mathfrak p$ of characteristic $p>0$.  
The number of points in $X_k$ was computed by R. Kottwitz \cite{Kottwitz, Kottwitz2} (and proves a
reformulation of the conjecture of R. Langlands and M. Rapoport (\cite{LanglandsRapoport}) and follows
earlier work on this question by J. Milne, W. Waterhouse, Langlands, Rapoport and others).
The Kottwitz ``counting formula", which is recalled below, (see also the review article by Clozel, \cite{Clozel})
is a certain sum of (finite ad\`elic) orbital 
integrals over the symplectic group, designed to facilitate comparison with the trace formula.
\quash{
of) this formula and the trace formula is one of the key steps in the proof 
of Langlands' conjectures (for the symplectic group) relating automorphic 
forms and Galois representations.
}

\subsection{}
One might ask whether there is a similar circle of geometric results for the 
space $Y = \Gamma_1 \backslash
G_1(\RR) /K_1$ where $G_1(\RR) = \GL_{n}(\RR)$, $K_1 = O_n$, and $\Gamma_1 = \GL_{n}(\ZZ)$
(or a congruence subgroup thereof).  Enthusiasm for this question follows
the recent spectacular advances
(\cite{HarrisLan, HarrisTaylor, PatrikisTaylor, Barnet, Taylor, Scholze}) in associating automorphic forms for $\GL_n$ with Galois
representations.  Unfortunately the space $Y$
does not carry the structure of a complex algebraic variety.  Rather, it has a natural interpretation as the moduli space of $n$-dimensional
compact Riemannian tori (with level structure), and there is no obvious
characteristic $p$ analog of a compact Riemannian torus.

Complex conjugation acts on the variety $X$ and it has been observed
(\cite{Silhol1, Silhol2, Comessatti, Shimura3, MilneShih, Adler, Gross, GT}) 
that its fixed points (that is, the set of ``real" points in $X$) correspond to (principally polarized)
Abelian varieties that admit a real structure.  A given Abelian variety may admit many distinct
real structures. However, in \cite{GT} the authors showed,  
with the addition of appropriate level structures, that \begin{enumerate}
\item the real points of $X$ are in a natural one to one correspondence with real isomorphism classes of
principally polarized Abelian varieties with anti-holomorphic involution and
\item
this set of real points in $X$  is a finite disjoint union of copies of the space $Y$.
\end{enumerate}
This provides a sort of ``algebraic moduli" interpretation for the space $Y$.

In an effort to find a characteristic $p>0$ analog of this result, 
one might ask whether it is possible to make sense of the notion of an anti-holomorphic 
involution of an Abelian variety $A/k$ that is defined over a finite field $k$.  We do not
have an answer to this question.  In this paper however, we investigate a construction that
makes sense when  $A$ is an {\em ordinary} Abelian variety.  We then discover that there are
finitely many isomorphism classes of principally polarized ordinary Abelian varieties with
anti-holomorphic involution.  We are able to count the number of
isomorphism classes of such varieties and to obtain a formula, similar to that of
Kottwitz, which involves ad\`elic integrals over the general linear group
rather than the symplectic group.  It resembles the finite ad\`elic part
(away from $p$) of the relative trace formula.

\subsection{} \label{subsec-Deligne-modules}
We use the descriptions, due to P. Deligne \cite{Deligne} and E. Howe \cite{Howe}, of the
category of $n$-dimensional polarized ordinary Abelian varieties over the finite field
$k = \FF_q$ with $q=p^a$ elements.  This
category is equivalent to the category $\mathcal D(n,q)$ of  {\em polarized
Deligne modules}, whose objects are triples $(T,F,\B)$ where $T$ is a
free $\ZZ$-module of rank $2n$, where $F:T \to T$ is a semisimple endomorphism whose
eigenvalues are (ordinary) Weil $q$-numbers, where $\B:T \times T \to \ZZ$ is a symplectic 
form, and where $F$ and $\B$ are required
to have a certain further list of properties, all of which are reviewed in Section
\ref{subsec-Deligne2} and Appendix  \ref{subsec-polarizations}, 
but we mention in particular that there exists $V:T \to T$ such that $VF = FV = qI$.
(This equivalence of categories involves a choice of
embedding $\varepsilon:W(\overline{k}) \to \CC$. )

There is a natural involution on the category of polarized Deligne modules, which may be viewed as a characteristic
$p$ analog of complex conjugation on the moduli space $X$.  This involution takes
$(T,F,\B)$ to $(T,V,-\B)$.  It takes a mapping $f:T \to T'$ to the same mapping so 
it takes isomorphism classes to isomorphism classes.  It is compatible with the natural
functor $\mathcal D(n,q) \to \mathcal D(n,q^r)$ which takes $(T,F,\B)$ to
$(T,F^r, \B)$.  In \S \ref{subsec-proof-iso} we show:

\begin{lem}\label{lem-fixed-iso-class}
Let $(T,F,\B)$ be a polarized Deligne module.  Then $(T,V,-\B)$ is also a polarized Deligne
module.  Suppose $T\otimes \QQ$ is a simple $\QQ[F]$ module.  Then $(T,F,\B)$ and
$(T,V,-\B)$ are isomorphic if and only if $(T,F,\B)$ admits a {\em real structure}, 
that is, a group isomorphism $\tau:T \to T$ such that
\begin{enumerate}
\item $\tau^2 = I$
\item $\tau F \tau^{-1} = V$
\item $\B(\tau x, \tau y) = -\B(x, y)$.
\end{enumerate}
\end{lem}

\subsection{}  Exploring this simple definition is the main object of this paper. 
 Morphisms of Deligne modules with real structures 
are required to commute with the involutions\footnote{The resulting category of
``real" polarized Deligne modules may be described as the ``fixed category"
$\mathcal D(n,q)^G$ where $G \cong \ZZ/(2)$ is the group generated by
the above involution.}. In \S \ref{def-real-structure} we
incorporate level structures. In Theorem \ref{prop-finite-isomorphism} we prove 
 that there is a
finite number of isomorphism classes of principally polarized Deligne modules (of rank $2n$, over
the field $k=\FF_q$) with real structure and principal level $N$ structure (where $(N,q)=1$).  The main goal of this
paper (Theorem \ref{thm-statement2}) is to determine this number, expressing it in a 
form that is parallel to the ``counting
formula" of R. Kottwitz, but with the symplectic group replaced by the general linear group.  
For $n=1$ the answer is very simple, see \S \ref{prop-case-n-equals-one}.

\subsection{}
Roughly speaking, the formula of Kottwitz (\cite{Kottwitz, Kottwitz2}) has two parts:  
an orbital integral (which is broken into a product of contributions at $p$ and 
contributions away from $p$) over the symplectic group gives the number of
isomorphism classes (of principally polarized Abelian varieties with level structure) within
each isogeny class, and the second part, which is a sum over certain conjugacy classes in
the symplectic group, which ``labels" the isogeny classes.  The set of 
ordinary Abelian varieties is a union of isogeny classes.  Therefore the number of 
isomorphism classes of ordinary Abelian varieties (with principal polarization and 
level structure) is a sum of a subset of the terms in Kottwitz' formula.  However, rather than starting 
from Kottwitz' general formula and identifying the relevant terms,  we instead
follow his method of proof, which, in this case of ordinary Abelian varieties,
gives rise to a formula that is slightly simpler in two ways.  First, the ``Kottwitz invariant"
$\alpha(\gamma_0;\gamma,\delta)$ does not appear in our formula (so presumably this invariant 
equals 1 in the case of ordinary Abelian varieties).  Second, the orbital integral at $p$ in
our formula is an ``ordinary" orbital integral over $\GSp(\QQ_p)$, rather than the twisted
orbital integral over $\GSp(K(k))$ that appears in \cite{Kottwitz}.  (Here, $K(k)$ denotes the
fraction field of the ring of Witt vectors $W(k)$ over the finite field $k$.)
The resulting formula, which may be viewed as the ``ordinary" part of Kottwitz' formula,
 is described in Theorem \ref{thm-statement1}.

\subsection{}
In \cite{Kottwitz}, Kottwitz uses the fundamental lemma for base change in order to convert
the twisted orbital integral into an ordinary orbital integral.   For completeness, and also 
to serve as a model for how to proceed in the presence of a real structure, in
\S \ref{sec-dieudonne-lattices} we do the reverse:  we convert the ordinary orbital integral
into a twisted orbital integral by identifying certain $\ZZ_p$-lattices in $\QQ_p^{2n}$ (which
are counted by the ordinary orbital integral) with
corresponding $W(k)$-lattices in $K(k)^{2n}$, which are counted by the twisted orbital integral.

\subsection{}  
By following the same method of proof as that in Theorem \ref{thm-statement1}
 we arrive at the ``counting formula" of Theorem \ref{thm-statement2} for the number of (``real")
isomorphism classes of principally polarized Deligne modules with real structure.  As in Theorem
\ref{thm-statement1}, the formula involves two parts.  The first part is  an ordinary orbital integral
(which is a product of contributions at $p$ and contributions away from $p$) over the general
linear group\footnote{more precisely, it is an integral over the group $\GL_n^* = \GL_1 \times \GL_n$}
which ``counts" the number of real Abelian varieties (with principal polarization and
principal level structure) within an isogeny class. The second part, which labels the
isogeny classes, is a sum over certain conjugacy classes in the general linear group.

\subsection{}
As in Theorem \ref{thm-statement1} the (ordinary) orbital integral at $p$ may be replaced 
by a twisted orbital integral, but a new idea is needed.
Let $(T,F,\B,\tau)$ be a polarized Deligne module with real structure.
The Tate module $T_{\ell}(A)$ (at a prime $\ell \ne p$) of the corresponding Abelian variety $A$ 
is naturally isomorphic to $T \otimes \ZZ_{\ell}$ so it inherits an involution
$\tau_{\ell}:T_{\ell}A \to T_{\ell}A$ from the real structure $\tau$.  
However, the Dieudonn\'e module $T_pA$ does not
inherit an involution that exchanges the actions of $F$ and $V$ until a further universal choice,
(analogous to the choice of embedding $\varepsilon:W(\overline{k}) \to \CC$) is made,
namely, that of a continuous $W(k)$-linear  involution $\bar\tau:W(\overline{k}) \to W(\overline{k})$
such that $\bar\tau\sigma^a = \sigma^{-a}\bar\tau$ where $\sigma$ denotes the absolute Frobenius.  In 
Proposition \ref{prop-tau-Witt} we show
that such involutions exist.  Using this choice we show in Proposition \ref{prop-conjugation-Dieudonne} 
that a real structure determines an involution of the Dieudonn\'e module $T_pA$.  Finally, these
pieces are assembled to construct a twisted orbital integral (over $\GL^*_n(K(k))$) that is equal to
the (ordinary) orbital integral at $p$ in Theorem \ref{thm-statement2}.

\subsection{} The statements and outline of proof of the main results appear in Section
\ref{sec-statement-of-results}, which may also be read as an overall guide to the paper.

\quash{
\subsection{}
The key technical observation in this paper is the fact that a choice of real structure $\tau$ on
a polarized Deligne module $(T,F,\B)$ corresponds to the involution of $\GSp_{2n}$
that was studied in \cite{GT}, namely, $g \mapsto \tau_0 g \tau_0^{-1}$ where
$\tau_0 = \left(\begin{smallmatrix} -I_n & 0 \\ 0 & I_n \end{smallmatrix} \right)$.  The
subgroup fixed by this involution is the group $\GL_n^* = \GL_1 \times \GL_n$.
}

\subsection{Acknowledgments}  We wish to thank R. Guralnick for useful conversations
about the symplectic group.  We are particularly grateful to the Defense Advanced
Research Projects Agency for their support under grant no. HR0011-09-1-0010
and to our program officer, Ben Mann, who continued to believe in this project in
its darkest hour.  The first author was also partially supported by the Institute
for Advanced Study through a grant from the Simonyi Foundation.


\section{The complex case}\label{sec-complex}
We briefly recall several aspects of the theory of moduli of real Abelian varieties, which 
serve as a partial motivation for the results in this paper.  We refer to Appendix
\ref{appendix-symplectic} for notations concerning the symplectic group, and to Appendix
\ref{sec-involutions} for discussion of the ``standard involution" $\tau_0\left(
\begin{smallmatrix} A & B \\ C & D \end{smallmatrix} \right) = \left( \begin{smallmatrix}
A & -B \\ -C & D \end{smallmatrix}\right)$.  For $x \in \Sp_{2n}$ let $\tilde x =
\tau_0 x \tau_0^{-1}$.
\subsection{}
Recall that a real structure on a complex Abelian variety $A$ is an anti-holomorphic involution of $A$.
It has been observed \cite{Silhol1, Silhol2, Comessatti, Shimura3, MilneShih, Adler, Gross, GT} 
that principally polarized 
Abelian varieties (of dimension $n$) with real structure correspond to ``real points" of the moduli space 
\[ X = \Sp_{2n}(\ZZ) \backslash \mathfrak h_n\] 
of all principally polarized Abelian varieties, where  
$\mathfrak h_n$ is the Siegel upper halfspace.  On this
variety, complex conjugation is induced from the mapping on $\mathfrak h_n$ that 
is given by $Z \mapsto \tilde{Z} = -\bar Z$ which is in turn induced from the ``standard involution"
$\tau_0$. 
A fixed point in $X$ therefore comes from a point
$Z \in \mathfrak h_n$ such that $\tilde{Z} = \gamma Z$ for some $\gamma \in \Sp_{2n}(\ZZ)$.  By the
Comessatti Lemma (\cite{Silhol1, Silhol2, Comessatti, GT}) this implies that $Z = \frac{1}{2} S + iY$ where
$S \in M_{n \times n}(\ZZ)$ is a symmetric integral matrix, which may be taken to consist of zeroes 
and ones, and $Y\in C_n=\GL_n(\RR)/O(n)$ is an element of the cone
of positive definite symmetric real matrices.  Let $\left\{S_1,S_2,\cdots,S_r\right\}$ be a collection
of representatives of symmetric integral matrices consisting of zeroes and ones, 
modulo $\GL_n(\ZZ)$-equivalence\footnote{that is, $S \sim A S \tr{A}$}.  The element
\[ \gamma_j = \left(\begin{matrix} I & \frac{1}{2}S_j \\ 0 & I \end{matrix} \right)\in \Sp_{2n}(\QQ)\]
takes the cone $iC_n$ into the cone $\frac{1}{2}S_j + iC_n$.  It follows that:
\begin{prop}  The set of real points
$Y$ in $X$ is the union $(1 \le j \le r)$ of translates by $\gamma_i$ of the arithmetic quotients
\[ Y_j = \Gamma_j \backslash C_n\]
where $\Gamma_j = \GL_n(\RR) \cap \left( \gamma_j \Sp_{2n}(\ZZ) \gamma_j^{-1}\right)$
and where $\GL_n \to \Sp_{2n}$ is the embedding 
\[A \mapsto \left( \begin{smallmatrix}
A & 0 \\ 0 & \tr{A}^{-1}\end{smallmatrix} \right).\qedhere\]\end{prop}

\subsection{}  However, a given principally polarized Abelian variety
$A$ may admit several non-isomorphic real structures (\cite{Silhol1}).  
Thus, the coarse moduli space of principally polarized
Abelian varieties with real structure does not coincide with $Y$ but rather, it maps to $Y$ by a finite mapping.
This multiplicity may be removed by replacing $X$ with the moduli space of  
principally polarized Abelian varieties with a sufficiently high level structure.  
More generally let $K^{\infty} \subset \Sp_{2n}(\AA_f)$ be a compact open subgroup of the finite
ad\`elic points of $\Sp_{2n}$ that is preserved by the involution $\tau_0$ and is
sufficiently small that $K^{\infty} \cap \Sp_{2n}(\QQ)$ is torsion free.  (We use $\Sp$
rather than $\GSp$ for expository purposes because the argument for $\GSp$ is
similar but slightly messier.)
Let $K_{\infty} = U(n)$. For $N \ge 2$ let
\begin{align*} 
\widehat{K}^0_N &= \ker\left( \Sp_{2n}(\widehat{\ZZ}) \to \Sp_{2n}(\ZZ/N\ZZ)\right)\\
\widehat{K}'_N &= \ker\left(\GL_n(\widehat{\ZZ}) \to \GL_n(\ZZ/N\ZZ)\right)
\end{align*}
denote the principal congruence subgroups of the symplectic and general linear groups.

As in \cite{Rohlfs}, the fixed points of the 
involution $\tau_0$ on  double coset space
\[X = \Sp_{2n}(\QQ) \backslash \Sp_{2n}(\AA)/K_{\infty}K^{\infty}\] 
are classified by classes in the nonabelian cohomology $H^1(\langle \tau_0 \rangle,
K^{\infty})$ (cf. Appendix \ref{appendix-cohomology}).
  
\begin{prop} \cite{Rohlfs, GT} The involution $\tau_0:\Sp_{2n} \to \Sp_{2n}$ passes to an anti-holomorphic involution 
$\tau:X \to X$ whose fixed point $X^{\tau}$ is isomorphic to the finite disjoint union,
\[ X^{\tau}\ \cong \coprod_{\alpha \in H^1(\langle \tau_0 \rangle,K^{\infty})} Y_{\alpha}\]
over cohomology classes $\alpha$, where (see below)
\[ Y_{\alpha} =  \GL_n(\QQ) \backslash \GL_n(\AA)/K_{\alpha}O(n)\] 
is an arithmetic quotient of $\GL_n(\RR)$ and $K_{\alpha}$ is a certain compact
open subgroup of $\GL_n(\AA_f)$.  If $4|N$ and if $K^{\infty} = \widehat{K}^0_N$ is the principal
congruence subgroup of $\Sp_{2n}(\widehat{\ZZ})$ of level $N$  then 
$K_{\alpha} = \widehat{K}'_N$ is independent of the cohomology class $\alpha$,
and $X^{\tau}$ may be
identified with the parameter space (or coarse moduli space) of principally polarized
Abelian varieties with real structure and level $N$ structure.
\end{prop}

\begin{proof}
The fixed point components correspond to nonabelian cohomology 
classes in the following (well known) way.
Suppose $x \in \Sp_{2n}(\AA)$ gives rise to a fixed point $\langle x\rangle \in X$.  
Then there exists
$\gamma \in \Sp_{2n}(\QQ)$, $k\in K^{\infty}$ and $m\in K_{\infty}$ such that
$\tilde x = \gamma x km$, hence
\[ x = \tilde\gamma  \gamma x k \tilde k m \tilde m.\]
The assumption that $K^{\infty}$ is sufficiently small implies that $\tilde\gamma \gamma = I$
and $k \tilde k m \tilde m = I$.  This means that $\gamma$, $k$ and $m$ are 1-cocycles 
in the $\tau_0$-cohomology of $\Sp_{2n}(\QQ)$, $K^{\infty}$ and $U(n)$ respectively, 
the first and third of which are trivial (see Appendix \ref{appendix-cohomology} 
and also \cite{GT} Prop. 4.6).  So there exists $b \in \Sp_{2n}(\QQ)$ and $u \in U(n)$ so that
$\gamma = (\tilde b)^{-1}b$ and $m = u {\tilde u}^{-1}$.  Set $y = bxu$; it
represents the same point $\langle y\rangle = \langle x\rangle \in X$ and 
\[ \tilde y = \tilde b \tilde x \tilde u = bxku = yk.\]
In this way we have associated a cohomology class $[k] \in 
H^1(\langle \tau_0 \rangle, K^{\infty})$
to the fixed point $\langle x\rangle \in X$ which is easily seen to be 
independent of the choice of representative $x$.

Suppose $\langle z\rangle \in X$ is a point whose associated cohomology class 
is the same as that of $\langle x\rangle=\langle y \rangle$.  As above there 
exists a representative $z \in \Sp_{2n}(\AA)$ so that $\tilde z = z k'$ for some 
$k' \in K^{\infty}$.  If the cohomology classes $[k] = [k'] $ coincide there exists
$u \in K^{\infty}$ so that $k' = u^{-1} k \tilde u$.  This implies that the element
$ g = z u^{-1} y^{-1} $ is fixed under the involution so it lies in $\GL_n(\AA)$,
which is to say that $\langle z \rangle = \langle zu^{-1} \rangle = \langle gy \rangle
\in X$. Thus the point $\langle z \rangle$ lies in the $\GL_n(\AA)$ ``orbit'' of
$\langle y \rangle = \langle x \rangle$.  The set of such points can therefore be
described as follows.

Let $\alpha = [k] \in H^1(\langle \tau_0\rangle, K^{\infty})$ denote the cohomology
class determined by the cocycle $k$.  It vanishes in
$H^1(\langle \tau_0 \rangle, \Sp_{2n}(\AA_f)$ so we may express $k = h^{-1} \tilde h$
for some $h \in \Sp_{2n}(\AA_f)$. Let $z = yh^{-1}$ and let
\[K_{\alpha} = K_h = \left( h^{-1} K^{\infty} h\right) \cap \GL_n(\AA_f).\]
Then $\tilde z = z$ so $z \in \GL_n(\AA)$.  
This shows that (right) translation by $h^{-1}$ defines a mapping
\[ \GL_n(\QQ) \backslash \GL_n(\AA) / O(n).K_{\alpha} \to X\]
whose image consists of the set $Y_{\alpha}\subset X^{\tau}$
fixed points that are associated to the cohomology class $\alpha=[k]$.  

If $N$ is even and if $K^{\infty} = \widehat{K}^0_N\subset\Sp_{2n}(\widehat{\ZZ})$ 
is the principal congruence subgroup of level
$N$  then the cohomology class $[k] \in H^1(\langle \tau_0\rangle, K^{\infty})$
actually vanishes in $H^1(\langle\tau_0\rangle,\Sp_{2n}(\widehat{\ZZ}))$ (see Proposition
\ref{prop-trivial-cohomology}).  If $N\ge 3$ then $\widehat{K}^0_N$ is sufficiently small.
In this case we can therefore take the element
$h$ to lie in $\Sp_{2n}(\widehat{\ZZ})$.  Since $K^{\infty}$ is a normal subgroup it follows that 
\[ K_h = (h^{-1}K^{\infty}h) \cap \GL_n(\AA_f) = K^{\infty} \cap \GL_n(\widehat{\ZZ})\]
is the principal level $N$ subgroup $\widehat{K}'_N\subset\GL_n(\widehat{\ZZ})$ 
and is independent of the
cohomology class $\alpha = [k]$.  Consequently the set $X^{\tau}$ consists of 
$|H^1(\langle\tau_0\rangle, K^{\infty})|$
isomorphic copies of $\GL_n(\QQ) \backslash \GL_n(\AA)/\widehat{K}'_N$.  When $4|N$ 
the set $X^{\tau}$ of real points in $X$  is identified
in \cite{GT} with a parameter space for isomorphism classes of Abelian varieties with level
$N$ structure and anti-holomorphic involution.
\end{proof}
\quash{
\subsection{} In \cite{GT} it is shown, for any $m \ge 1$ that the natural mapping
\[ H^1(\langle\tau_0\rangle,\widehat{K}^0_{4m}) \to 
H^1(\langle\tau_0\rangle, \Sp_{2n}(\widehat{\ZZ}))\]
is trivial, where $\widehat{K}^0_{4m}$ is the principal level $N=4m$ subgroup of 
$\Sp_{2n}(\widehat{\ZZ})$.  (See also Proposition \ref{prop-trivial-cohomology}).  
In this case, for any cohomology class
$\alpha \in H^1(\langle \tau_0 \rangle, \widehat{K}^0_{4m})$ one may therefore take the above
element $h$ to lie in $\Sp_{2n}(\widehat{\ZZ})$, which contains $\widehat{K}^0_{4m}$ as a normal
subgroup.  Therefore $K_{h}$ is the principal congruence subgroup $K'_{4m}$ of level
 $N=4m$ in $\GL_n(\widehat{\ZZ})$, and the full fixed point set $X^{\tau}$ consists of
$|H^1(\langle\tau_0\rangle,\widehat{K}^0_{N})|$ isomorphic copies of
\[ Y_{\alpha} = \GL_n(\QQ) \backslash \GL_n(\AA)/O(n)K'_{4m}.\]
Moreover, as discussed in \cite{GT} the full fixed point set $X^{\tau}$ may 
in this case be identified with the parameter space (or coarse moduli space) 
of principally polarized Abelian varieties with real structure and  level $N$ structure.
}
\subsection{} The Siegel space $\mathfrak h_n$ admits other interesting anti-holomorphic
involutions.  In \cite{GT1} such an involution is described on $\mathfrak h_2$ whose fixed point
set is hyperbolic 3-space (cf. \cite{Nygaard}).  After appropriate choice of level structure, it passes to an
involution of the moduli space $X$ whose fixed point set is a union of arithmetic hyperbolic
3-manifolds which may be interpreted as constituting a coarse moduli space for Abelian 
varieties with ``anti-holomorphic multiplication" by an order in an imaginary quadratic number field.

\section{Deligne modules}\label{sec-ordinary}
\subsection{Ordinary Abelian varieties} \label{subsec-ordinary}
Throughout this section we fix a finite field $k = \FF_q$ of characteristic $p$.
Let $A/k$ be a dimension $n$ Abelian variety.  Recall that $A$ is {\em ordinary} if any of the following
equivalent conditions is satisfied.
\begin{enumerate}
\item If $\cdot p:A(\bar k) \to A(\bar k)$ denotes the multiplication by $p$
then its kernel has exactly $p^g$ points.
\item the local-local component of the $p$-divisible group
$A(p^{\infty}) = \underset{\longleftarrow}{\lim}A[p^r]$ is trivial.
\item The middle coefficient of the characteristic polynomial $h_A$ of the Frobenius
endomorphism of $A$ is not divisible by $p$.
\item Exactly half of the roots of $h_A$ in $\overline{\QQ}_p$ are $p$-adic units.
\end{enumerate}

\subsection{}\label{subsec-Deligne2}  Recall the basic definitions of Deligne \cite{Deligne}.
A {\em Deligne module}  of rank $2n$ over the field $k = \FF_q$ of $q$ elements is a pair
$(T,F)$  where $T$ is a free $\ZZ$-module of dimension $2n$
and $F:T \to T$ is an endomorphism such that the following conditions are satisfied:
\begin{enumerate}
\item The mapping $F$ is semisimple and all of its eigenvalues in $\CC$ have magnitude
$\sqrt{q}.$
\item Exactly half of the eigenvalues of $F$ in $\overline{\QQ}_p$ are $p$-adic units 
and half of the eigenvalues are divisible by $q.$  (So $\pm\sqrt{q}$ is not an eigenvalue.)
\item The middle coefficient of the characteristic polynomial of $F$ is coprime to $p.$
\item There exists an endomorphism $V:T \to T$ such that $FV=VF=q.$
\end{enumerate}
A morphism $(T_A,F_A) \to (T_B,F_B)$ of Deligne modules is a group homomorphism
$\phi:T_A \to T_B$ such that $F_B \phi = \phi F_A.$ 

\subsection{}
Let $W(k)$ be the ring of (infinite) Witt vectors over $k.$
Following \cite{Deligne},  fix an embedding
\begin{equation}\label{eqn-epsilon}
\varepsilon: W(\bar k) \to \CC.\end{equation}
By a theorem of Serre and Tate,  \cite{Drinfeld, Katz, Messing, Srinivas} the ordinary Abelian
variety $A$ has a canonical lift $\bar A$ over $W(k)$ which, using (\ref{eqn-epsilon}) 
gives rise to a complex variety $A_{\CC}$ over $\CC.$  Let $F \in \Gal(\bar k/k)$ denote the Frobenius.
The geometric action of $F$ on $A$ lifts to an automorphism $F_A$ on
\[ T = T_A = H_1(A_{\CC},\ZZ).\]

\begin{thm} \cite{Deligne} \label{thm-Deligne}
This association $A \to (T_A,F_A)$, determined by the embedding (\ref{eqn-epsilon}), 
induces an equivalence of categories between
the category of $n$-dimensional ordinary Abelian varieties over $k=\FF_q$ and the
category of Deligne modules of rank $2n$.
\end{thm}

\subsection{} \label{subsec-Howe}
The main results of Howe \cite{Howe} are recalled in Appendix
\ref{appendix-polarizations} and summarized here.  A polarization of a
Deligne module $(T,F)$ is a pair $(\B,\Phi)$ where $\B: T\times T
\to \ZZ$ is a symplectic form (alternating and nondegenerate over
$\QQ$) such that 
\begin{enumerate}
\item[(0)] $\omega(x,y) = - \omega(y,x)$ for all $x,y\in T$,
\item $\omega:(T\otimes \QQ) \times (T \otimes \QQ) \to \QQ$ is nondegenerate,
\item $\omega(Fx,y) = \omega(x,Vy)$ for all $x,y \in T$,
\end{enumerate}
and where $\Phi$ is a CM type on $\QQ[F]$ such that the following {\em positivity
condition} holds:
\begin{enumerate} 
\item[(3)] the form $R(x,y) = \omega(x,\iota y)$ is symmetric and positive definite,
where $\iota$ is some (and hence, any) totally $\Phi$-positive imaginary element of $\QQ[F]$ 
(cf.~Appendix \ref{subsec-polarizations} or \ref{subsec-viable}).
\end{enumerate}
In this case we say that $\B$ is a $\Phi$-positive polarization.

If $(T,F)$ is a simple Deligne module and if $\omega$ is a symplectic form 
satisfying (0),(1),(2) above then there exists a (unique) CM type $\Phi$ on 
$\QQ[F]$ so that $\omega$ is $\Phi$-positive (cf.~Lemma \ref{lem-unique-positive}).
 A polarization of $(T,F)$ is a symplectic form $\omega$ satisfying the above conditions for some CM type $\Phi$, in which case there is only one such.
In \cite{Deligne} it is shown that the embedding (\ref{eqn-epsilon}) determines a CM type
$\Phi_{\varepsilon}$ on every such CM algebra $\QQ[F]$, see \S \ref{subsec-phi-varepsilon}.  
In (\cite{Howe}), Howe proves the following.
\begin{thm}
The equivalence of categories in Theorem \ref{thm-Deligne} (that is determined by
the embedding (\ref{eqn-epsilon})) extends to an equivalence between
the category of  polarized $n$-dimensional Abelian varieties over $\FF_q$ with the category of
$\Phi_{\varepsilon}$-positively polarized Deligne modules (over $\FF_q$) of rank $2n$.
\end{thm}

\subsection{} The {\em standard symplectic form} on $\ZZ^{2n}$ is the symplectic form
$\B_0:\ZZ^{2n} \times \ZZ^{2n} \to \ZZ$ whose matrix is $\left( \begin{smallmatrix}
0 & I \\ -I & 0 \end{smallmatrix} \right)$.  The {\em standard involution} on
$\ZZ^{2n}$ is the involution $\tau_0(x,y) = (-x,y)$.

\begin{defn}\label{def-real-structure}
Fix a Deligne module $(T,F)$ over $k = \FF_q$ of dimension $2n$.  A {\em real structure} on  
$(T,F)$ is a $\ZZ$-linear homomorphism $\tau:T \to T$ such that $\tau^2 = I$ and such that 
$\tau F \tau^{-1} = V.$  A (real) morphism $\phi:(T,F,\tau) \to (T',F',\tau')$ of
Deligne modules with real structures is a group homomorphism $\phi:T \to T'$ so
that $\phi F = F'\phi$ and $\phi\tau = \tau'\phi$.  
A real structure $\tau$ is compatible with a 
polarization $\B:T \times T \to \ZZ$ if
  \begin{equation}\label{eqn-real-polarization}
 \B(\tau x, \tau y) = -\B(x,y)\end{equation}
for all $x, y \in T.$

Let $N \ge 1$.  We define a {\em (principal) level $N$ structure} on $(T,F)$ under the conditions that $p\nmid N$ and
that $F \equiv I \mod N$.  In this case, a level $N$ structure is an isomorphism $\beta:T/NT \to (\ZZ/N\ZZ)^{2n}$.
A level $N$ structure is compatible with a polarization $\B:T \times T \to \ZZ$
if $\beta_*(\B) = \bar\B_0$ is the reduction modulo $N$ of the standard symplectic form.  

If $(T,F,\tau)$ is a Deligne module with real structure then a level $N$ structure $\beta$
on $(T,F)$ is {\em compatible with} $\tau$ if $\beta_*(\tau) = \bar\tau_0$ is the reduction 
modulo $N$ of the standard involution (cf.~Appendix \ref{appendix-level}).  A necessary condition for the
existence of a level $N$ structure that is compatible with $\tau$ is that $p \equiv 1 \mod N$, which
also implies that $V \equiv I \mod N$, cf.~\S \ref{subsec-q-real}.
\end{defn}

Let $(T,F,\tau)$ be a Deligne module with real structure. For any CM type $\Phi$ on
$\QQ[F]$ there exists a polarization that is $\Phi$-positive and compatible with $\tau$, by
Lemma \ref{lem-unique-positive} and equation (\ref{eqn-omega-conjugation}).
The proof of the following finiteness theorem will appear in Appendix
\ref{sec-finiteness}, along with an explicit description of the count for $n=1$.

\begin{thm}  \label{prop-finite-isomorphism}  Assume $p \nmid N$.
There are finitely many isomorphism classes of principally ($\Phi_{\varepsilon}$-positively)
polarized Deligne modules of rank $2n$ over $\FF_q$ with real structure and with 
principal level $N$ structure. \end{thm}



\subsection{Isogeny}\label{sec-isogeny}
Let $S$ be a commutative ring with 1.  Let $\mathcal C$ be a $\ZZ$-linear
Abelian category.  Following \cite{DeligneLetter} and \cite{Kottwitz}, define the associated  
category {\em up to $S$-isogeny} to be the category with the same objects but with morphisms
\[ \hHom_S(A,B) = \hHom_{\mathcal C}(A,B)\otimes_{\ZZ} S.\]
An $S$-isogeny  of polarized Deligne modules $\phi:(T_1,F_1,\B_1) \to (T_2,F_2,\B_2)$ is defined to be
an $S$-isogeny $\phi:(T_1,F_1) \to
(T_2,F_2)$ for which there exists $c \in S^{\times}$ such that $\phi^*(\B_2) = c\B_1$, in which
case $c$ is called the {\em multiplier} of the isogeny $\phi$. 

\subsection{Remark}
A $\overline{\QQ}$-isogeny $h:(T_1,F_1,\omega_1) \to (T_2,F_2,\omega_2)$  between polarized Deligne modules induces an isomorphism $\QQ[F_1] \cong \QQ[F_2]$ which may fail to take 
the corresponding CM type $\Phi_1$ to $\Phi_2$.   This difficulty vanishes if $(T_i, F_i,
\omega_i)$ ($i=1,2$) are positively polarized with respect to the canonical CM 
type $\Phi_{\varepsilon}$ that is determined by the embedding (\ref{eqn-epsilon}). 
Moreover if $(T_i,F_i,\omega_i)$ are $\Phi_{\varepsilon}$-positively polarized then 
so is their direct sum.  If $S \subset \RR$ then an $S$-isogeny of 
$\Phi_{\varepsilon}$-positively polarized Deligne modules will have positive multiplier.

\begin{prop}\label{prop-real-existence}
The category of Deligne modules (resp.  polarized Deligne modules)
 with real structure, up to $\QQ$-isogeny is semisimple.
\end{prop}
\begin{proof} The proof is more or less standard.
For the first statement, it suffices to check complete reducibility.  
Let  $(T,F,\tau)$ be a Deligne module,  and let $(T_1,F,\tau)$ be a submodule.
Since $F$ is semisimple the ring $\QQ[F]$ is isomorphic to a product of distinct 
number fields.   It follows that $(T,F,\tau)$ decomposes canonically into a sum of modules over 
these constituent fields.  So we may assume that $\QQ[F]$ is a field.
Set $W = T\otimes\QQ$ and let $W_1 = T_1 \otimes \QQ$.   Choose any decomposition of $W$ into simple
$\QQ[F]$-submodules so that $W_1$ is a summand.  The resulting projection $\pi:W \to W_1$ is 
$\QQ[F]$-equivariant.  Let $e = \pi + \tau \pi \tau:W \to W_1$.  Then $e$ is surjective (since its 
restriction to $W_1$ coincides with multiplication by $2$) and $W'_1:= \ker(e)$
is preserved by $F$ and by $\tau$.  Thus, the decomposition $W = W_1 \oplus W'_1$ 
is preserved by $F$ and by $\tau$.  For any choice of lattice $T'_1 \subset W'_1$ preserved
by $F$ and $\tau$ the module $(T_1\oplus T'_1,F,\tau)$ is $\QQ$-isogenous to $(T,F,\tau)$.

The same method works in the presence of a polarization.  
Let $(T,F, \B,\tau)$ be a Deligne
module with real structure and $\Phi$-positive polarization with respect to a choice $\Phi$ of
CM type on $\QQ[F]$.  Let $W = T\otimes \QQ$ and suppose that $W_1\subset W$ is a subspace
preserved by $F$ and by $\tau$.  Set $F_1 = F|W_1$.  It follows that
\begin{enumerate}
\item the restriction of $\B$ to $W_1$ is nondegenerate and is $\Phi_1$-positive, where
$\Phi_1$ is the CM type on $\QQ[F_1]$ that is induced from $\Phi$,
\item the subspace
$W_2 = \left\{ y \in W|\ \B(w,y) = 0 \text{ for all } w \in W_1 \right\} $
is also preserved by $F$ and by $\tau$ and it is $\Phi_2$-positively polarized by
the restriction $\B|W_2$ where $\Phi_2$ is the CM type induced from $\Phi$ on 
$\QQ[F_2]$ (where $F_2 = F|W_2$), and
\item The module $W$ decomposes as an orthogonal sum $W = W_1 \oplus W_2$.\qedhere
\end{enumerate}
\end{proof}

\subsection{Proof of Lemma \ref{lem-fixed-iso-class}}\label{subsec-proof-iso}
Although Lemma \ref{lem-fixed-iso-class} of the Introduction is not used in the rest of
this paper, it provides some motivation for the definition of a real structure, 
and it is most convenient to give its proof here.

If $(T,F,\B)$ is a $\Phi_{\varepsilon}$-positively polarized Deligne module then 
$(T,V,-\B)$ is one also because the symplectic form $-\B$ satisfies the positivity 
condition (3) of \S \ref{subsec-polarizations} for $(T,V)$.

Now suppose that $T\otimes \QQ$ is a simple $\QQ[F]$ module and suppose the isomorphism 
class of $(T,F,\B)$ is fixed under this involution.  
Let $\psi:T \to T$ be an isomorphism such that $\psi F \psi^{-1} = V$ and $\psi^*(\B) = -\B$.  
We will prove that in fact $\psi^2= I$ so that $\psi$ is a real structure.  Let $v_0 \in T$
be a cyclic vector for the action of $F$, so that $\left\{ v_0, Fv_0,\cdots,
F^{n-1}v_0\right\}$ form a basis of $F \otimes \QQ$.  Therefore we can express
$\psi(v_0) = f(F)(v_0)$ where $f(F) = \sum_{i=0}^{n-1} a_i F^i$ is a polynomial in $F$.
Let $w \in T$ be arbitrary and similarly write $w = g(F)(v_0)$ for some polynomial 
$g(F) = \sum_ib_iF^i$. Then $\psi(w) = g(V)\psi(v_0) = g(V)f(F)v_0$.  Now calculate
in two ways:
\begin{align*}
 \omega(\psi(v_0),\psi(w)) &= \omega(f(F)v_0, f(F)g(V)v_0) = \omega(v_0, f(V)f(F)g(V)v_0)\\
&=-\omega(v_0,w) = \omega(w,v_0) = \omega(g(F)v_0, v_0) = \omega(v_0, g(V)v_0).
\end{align*}
Since $w$ was arbitrary this implies that $f(V)f(F) = I$.  Therefore
\[\psi^2(w) = \psi(g(V)f(F)v_0) = g(F) f(V) \psi(v_0) = f(V)f(F)g(F)v_0 = w.\qed\]


In order to  ``count" the number of real Deligne modules it is necessary to
describe them in terms of algebraic groups as follows.

\begin{lem}\label{subsec-real-standard-form}  Let $(T,F,\B,\tau)$ be a rank $2n$
Deligne module with real structure that is positively polarized with respect to a 
CM type $\Phi$ of $\QQ[F]$.  Then it is isomorphic to one of the form
$(L, \gamma, \B_0, \tau_0)$ where $L \subset \QQ^{2n}$ is a lattice,
$\gamma \in \GSp_{2n}(\QQ)$, $\B_0$ is the standard symplectic structure on
$\QQ^{2n}$, and $\tau_0$ is the standard involution, and where: 
\begin{itemize}
\item[($\gamma 1$)] $\gamma \in \GSp_{2n}(\QQ)$ is semisimple with multiplier $c(\gamma) = q$
\item[($\gamma 2$)] its characteristic polynomial is an ordinary Weil $q$-polynomial (see Appendix
\ref{sec-Weil-polynomials})
\item[$(\gamma 3$)] The bilinear form $R(x,y) = \B_0(x,\iota y)$ is symmetric and 
positive definite, where $\iota$ is any $\Phi_{\varepsilon}$-positive imaginary 
element of $\QQ[\gamma]$  (see Appendix \ref{subsec-polarizations}).
\item[($\gamma 4$)] the element $\gamma$ is \qreal, meaning that $\tau_0 \gamma \tau_0^{-1}
= q \gamma^{-1}$; 
\end{itemize}
and where $L\subset \QQ^{2n}$ is a lattice such that
\begin{itemize}
\item[(L1)] the symplectic form $\B_0$ takes integer values on $L$
\item[(L2)] $L$ is preserved by $\gamma$ and by $q\gamma^{-1}$.
\item[(L3)] the lattice $L$ is preserved by the standard involution $\tau_0$ 
\end{itemize}
With these choices, the group of self $\QQ$-isogenies of $(T,F,\B)$ 
(resp. of $(T,F,\B, \tau)$) is isomorphic to the
centralizer $Z_{\gamma}(\QQ)$ in $\GSp_{2n}(\QQ)$ (resp. in $\GL_n^*(\QQ))$.  Every element
$\phi \in Z_{\gamma}(\QQ)$ has positive multiplier $c(\gamma) >0$.
\end{lem}

\begin{proof} 
By Lemma \ref{lem-Darboux} and Proposition \ref{prop-classification}
 there is a basis $\phi:T \otimes \QQ \to \QQ^{2n}$ of $T \otimes\QQ$
so that $\B$ becomes $\B_0$ and so that $\tau$ becomes $\tau_0$.  
Take $\gamma = \phi F \phi^{-1}$.  This gives an isomorphism
$\QQ[F] \cong \QQ[\gamma]$ so we obtain a CM type $\Phi'$ on $\QQ[\gamma]$ for which
$\B_0$ is $\Phi'$-positive.  Taking $L=\phi(T)$ gives a lattice that is
preserved by $\gamma$ and $\tau$ on which $\B_0$ takes integral values,
which proves the first statements.  The
centralizer statement is clear.  If $\phi \in Z_\gamma(\QQ)$ then
\[ R(\phi(x),\phi(x)) = \B_0(\phi x, \iota \phi x) = \B_0(\phi x, \phi \iota x) 
= c(\phi) R(x,x)>0.  \qedhere\]
\end{proof}

\subsection{Totally real lattice modules}
In this section we add the simple observation that, up to $\QQ$-isogeny, a Deligne
module with real structure is determined by the $\tau$-fixed sublattice.
Fix $q = p^r$ and fix $n \ge 1$.  A totally real lattice module (of rank $n$ 
and characteristic $q$) is a pair $(L,A)$ where
$L$ is a free Abelian group of rank $n$ and $A:L \to L$
is a semisimple endomorphism whose eigenvalues $\alpha$
are totally real with $|\rho(\alpha)|<2\sqrt{q}$ for
every embedding $\rho:\QQ[\alpha] \to \RR$. The module $(L,A)$ is
{\em ordinary} if $\det(A)$ is not divisible by $p$.
The characteristic polynomial of an (ordinary) totally real lattice module
is an (ordinary) real Weil $q$-polynomial, cf. Appendix \ref{sec-Weil-polynomials}.
A level $N$ structure on $(L,A,R)$ is an isomorphism $\alpha:L/NL 
\to (\ZZ/N\ZZ)^n$ such that $\alpha \circ \overline{A} = \alpha$ 
where $\overline{A} = A \mod N$.

A Deligne module (of rank $2n$ over $\FF_q$) with real structure, 
$(T,F,\tau)$, gives rise to a pair $(L,A)$ 
by $L = T^{\tau}$ and $A = (F+V)|L$, in which case the characteristic polynomial 
of $A$ is the real counterpart of the characteristic polynomial of $F$.
If $\alpha:T/NT \to (\ZZ/N\ZZ)^{2n}$ is a level $N$
structure that is compatible with $\tau$ then its
restriction to the fixed point set $\beta:L/NL \to
(\ZZ/N\ZZ)^n$ is a level $N$ structure on $(L,A)$.

\begin{prop} 
The association $(T,F,\tau) \mapsto
(L=T^{\tau}, A=F+V)$ defines a functor from the
category of Deligne modules with real structure to the category of ordinary
totally real lattice modules.  It becomes an equivalence on the corresponding
categories up to $\QQ$-isogeny.
\end{prop}

\begin{proof} In both cases the $\QQ$-isogeny class is determined by the
characteristic polynomial (cf. Proposition \ref{prop-Qbar-conjugacy} below), 
so the result follows from Proposition \ref{prop-h(x)}.
\end{proof}

For use in Lemma \ref{prop-five-part} and Proposition \ref{prop-conjugation-Dieudonne}
 we will need the following.
\begin{prop}\label{lem-tau-switches}
\cite{Deligne} Let $(T,F)$ be a Deligne module.
The endomorphism $F$ determines a unique decomposition
\begin{equation}\label{eqn-decomp}
 T \otimes \ZZ_P \cong T' \oplus T^{\prime\prime}\end{equation}
such that $F$ is invertible on $T'$ and $F$ is divisible by $q$ on
$T^{\prime\prime}.$  If $\tau$ is a real structure on $(T,F)$ then
$\tau(T') = T^{\prime\prime}$ and $\tau(T^{\prime\prime}) = T^{\prime}$.
\end{prop}
\begin{proof} The decomposition (\ref{eqn-decomp}) is proven in \cite{Deligne}.
The module $T'\otimes \overline{\QQ}_p$ is the sum of the eigenspaces of $F$
whose eigenvalues in $\overline{\QQ}_p$ are $p$-adic units while 
$T^{\prime\prime}\otimes \overline{\QQ}_p$
is the sum of eigenspaces whose eigenvalues are divisible by $p.$  If $x$ is an
eigenvector of $F$ whose eigenvalue $\alpha$ is a $p$-adic unit then $x$ is an
eigenvector of $V$ with eigenvalue $q/\alpha$ hence $\tau(x)$ is an eigenvector
of $F$ whose eigenvalue is divisible by $q$. 
\end{proof}


\section{\qreal elements}\label{sec-q-real}
\subsection{}\label{subsec-q-real}  
Let $R$ be an integral domain and let $\B_0$ be the standard (strongly non-degenerate)
symplectic form on $R^{2n}$ corresponding to the matrix $J=\left(\begin{smallmatrix}
0 & I \\ -I & 0 \end{smallmatrix}\right)$.
Define the {\em standard involution} $\tau_0:R^{2n} \to R^{2n}$ by
$\tau_0 = \left(\begin{smallmatrix} -I & 0 \\ 0 & I \end{smallmatrix} \right)$.
The subgroup of $\GSp_{2n}(R)$ that is fixed under conjugation by $\tau_0$ is denoted $\GL_n^*(R)$, and it
is the image of the {\em standard embedding}
\[ \delta: R^{\times} \times \GL_n(R) \to \GSp_{2n}(R) \quad
\delta(\lambda,x) = \left(\begin{smallmatrix}
\lambda X & 0 \\ 0 & \tr{\!X}^{-1} \end{smallmatrix}\right),\]
see Appendix \ref{sec-involutions}.  Let us say that an element $\gamma \in \GSp_{2n}(R)$ is 
{\em \qreal} if it is semisimple, has multiplier $q$ and 
if\footnote{Compare the equation $\tau F \tau^{-1} = V$ of \S \ref{def-real-structure}}
\[ \tau_0 \gamma \tau_0^{-1} = q \gamma^{-1},\]
or equivalently if $\gamma = \left( \begin{smallmatrix} A & B \\ C & \tr{\!A}
\end{smallmatrix} \right)\in \GSp_{2n}(R)$ and $B, C$ are symmetric,
and $A^2-BC = qI.$  It follows that $B\tr{\!A} =  AB$ and $CA = \tr{\!A}C$. 

\begin{lem}\label{lem-q-real}
Let $\gamma = \left( \begin{smallmatrix} A & B \\ C & \tr{\!A}
\end{smallmatrix} \right)\in \GSp_{2n}(\QQ)$ be \qreal.  Then the following
statements are equivalent.  \begin{enumerate}
\item The matrices $A$, $B$, and $C$ are nonsingular.
\item The element $\gamma$ has no eigenvalues in the set $\left\{ \pm \sqrt{q},
\pm \sqrt{-q} \right\}$.
\end{enumerate}
If these properties hold then the matrix $A$ is semisimple, 
and the characteristic polynomial of $A$ is $h(2x)$, where $h(x)$ 
is the real counterpart (see \S \ref{subsec-real-counterpart}) to $p(x)$, 
the characteristic polynomial of $\gamma$.  
If $p(x)$ is also a Weil $q$-polynomial then every eigenvalue $\beta$ of $A$ satisfies
\begin{equation}\label{eqn-beta}
|\beta| < \sqrt{q}.\end{equation}
Conversely, let $A \in \GL_n(\QQ)$ be semisimple and suppose that its eigenvalues
$\beta_1,\cdots,\beta_n$ (not necessarily distinct) are totally real and that
$|\beta_r|<\sqrt{q}$ for $1 \le r \le n$.   Then
for any symmetric nonsingular matrix $C \in \GL_n(\QQ)$ such that
$\tr{\!A}C = C {A}$, the following element
\begin{equation}\label{eqn-constructing-gamma}
\gamma = \left( \begin{matrix} A & (A^2-qI)C^{-1} \\
C & \tr{\!A} \end{matrix} \right) \in \GSp_{2n}(\QQ)\end{equation}
is $q$-inversive and its eigenvalues, 
\begin{equation}\label{eqn-eigenvalues}
\alpha_r = \beta_r \pm \sqrt{\beta_r^2-q}\quad (1 \le r \le n),\end{equation}
are Weil $q$-numbers. 
\end{lem}

\begin{proof}
One checks that if $w = \left(\begin{smallmatrix}u\\v\end{smallmatrix}\right)$
is an eigenvector of $\gamma$ with eigenvalue $\lambda$ then \begin{enumerate}
\item[(a)] $\tau_0(w)=\left(\begin{smallmatrix}-u\\v\end{smallmatrix}\right)$
is an eigenvector of $\gamma$ with eigenvalue $q/\lambda$
\item[(b)] $u$ is an eigenvector of $A$ with eigenvalue $\frac{1}{2}\left(\lambda +
\frac{q}{\lambda}\right)$
\item[(c)] $v$ is an eigenvector of $\tr{\!A}$ with eigenvalue $\frac{1}{2}\left(\lambda +
\frac{q}{\lambda}\right).$
\end{enumerate}
Therefore, if $\lambda = \pm \sqrt{q}$ is an eigenvalue of $\gamma$ then it is also an
eigenvalue of $A$, hence $BC = A^2 -qI$ is singular.  If $\lambda = \pm \sqrt{-q}$ is an
eigenvalue of $\gamma$ then $\lambda + q/\lambda = 0$ so $A$ is singular.  Conversely, if
$A$ is singular, say $Au = 0$ then 
\[
\left(\begin{matrix} A & B \\ C & \tr{\!A} \end{matrix} \right)
\left(\begin{matrix} \sqrt{-q} u \\ Cu \end{matrix} \right) =
\left( \begin{matrix} BCu \\ \sqrt{-q}(Cu + CAu) \end{matrix} \right) =
\sqrt{-q}\left(  \begin{matrix} \sqrt{-q} u \\ Cu \end{matrix} \right) \]
so $\sqrt{-q}$ is an eigenvalue of $\gamma$.  If $C$ is singular then there exists a
nonzero vector $u$ so that $Cu=0$ and $Au = \sqrt{q}u$ so the vector $\left( \begin{smallmatrix}
u \\ 0 \end{smallmatrix}\right)$ is an eigenvector of $\gamma$ with eigenvalue $\sqrt{q}$.  If
$C$ is nonsingular but $B$ is singular then there exists a vector $u$ so that $Cu \ne 0$,
$BCu = 0$ and $Au = \sqrt{q}u$ hence the vector $\left(\begin{smallmatrix} 0 \\ Cu
\end{smallmatrix} \right)$ is an eigenvector of $\gamma$ with eigenvalue $\sqrt{q}$.  This
proves that conditions (1) and (2) are equivalent.  Since $\gamma$ is semisimple, 
points (a) and (b) above imply that $A$ is semisimple and that its characteristic polynomial is $h(2x)$.  
 The inequality (\ref{eqn-beta}) follows from Proposition \ref{prop-h(x)} part (2).

For the proof of the ``converse" statement, given $A$ with eigenvalues
$\beta_1,\cdots, \beta_n$ set $\theta_r = \sqrt{\beta_r^2-q}$. Let $C\in \GL_n(\QQ)$
be symmetric with $\tr{\!A}C = C A$.  One checks if
$x_r$ is an eigenvector of $\tr{\!A}$ with eigenvalue $\beta_r$ then
\[ w^{\pm}_r := \left( \begin{matrix} \pm \theta_rC^{-1}x_r \\
 x_r \end{matrix} \right)\]
are eigenvectors of $\gamma$ with eigenvalues $\beta_r \pm  \theta_r$.
These are Weil $q$-numbers because the collection $\{\beta_1,\cdots,\beta_n\}$
is a union of Galois conjugacy classes.\end{proof}

\subsection{Conjugacy of $q$-inversive elements}\label{subsec-q-inversive-conjugacy}
In this section we consider $\GL_n$ versus $\Sp_{2n}$-conjugacy of $q$-inversive elements.
Let $L \supset \QQ$ be a field and let $\gamma = \left( 
\begin{smallmatrix} A & B \\ C & \tr{A} \end{smallmatrix} \right) \in
\GSp_{2n}(L)$.  Let $x = \left( \begin{smallmatrix} \lambda X & 0 \\ 0 & \tr{X}^{-1}
\end{smallmatrix}\right) \in \GL_n^*(L)$.  Then
\begin{equation}\label{eqn-Gln-action}
x \gamma x^{-1} = \left( \begin{matrix}
XAX^{-1} & \lambda XB\tr{X} \\
\frac{1}{\lambda}\tr{X}^{-1} C X^{-1} & \tr{X}^{-1} \tr{A} \tr{X} \end{matrix} \right).
\end{equation} 
It follows that $\gamma$ is \qreal if and only if $x \gamma x^{-1}$ is \qreal.
We say that two elements of $\GSp_{2n}$ are $\GL^*_n$-conjugate if the conjugating element
lies in the image of $\delta.$ 

Suppose we attempt to diagonalize the matrix $\gamma$ using conjugation by elements
$x \in \GL_n^*(L)$.  First consider the case when the field $L$ contains all the 
eigenvalues of $A$.  Let $\beta \ne \mu$ be eigenvalues of $A$.  From
the equation $\tr{A}C = CA$ it follows that the eigenspaces $V_{\beta}$ and $V_{\mu}$ of
the matrix $A$ are
orthogonal with respect to the inner product defined by the symmetric matrix $C$.  Thus,
$C = \oplus_{\beta} C_{\beta}$ is an orthogonal direct sum of nondegenerate symmetric bilinear 
forms $C_{\beta}$ on the eigenspaces $V_{\beta}$.  Over the field $L$ we can diagonalize $A$,
grouping the eigenvalues together, from which we see that the centralizer of $A$ is
\[ Z(A) = \prod_{\beta}\GL(V_{\beta}).\]
The matrix $C$ then becomes a block matrix with one block for each eigenvalue $\beta$ of $A$, and
$B = (qI -A^2)C^{-1}$ is also diagonal.  This gives a {\em standard form} for $\gamma$.
If the field $L= {\RR}$ then each of the signatures $\sig(C_{\beta})$ is invariant under
 $Z(A)$-congruence\footnote{Symmetric matrices $S$ and $T$ are {\em congruent} if there exists
a matrix $X$ so that $T = XS\tr{\!X} $.}.
Therefore the matrix $C$ can be diagonalized so as to have $\pm 1$ diagonal entries, with
$\sig(C_{\beta})$ copies of $-1$ appearing in the $\beta$-block.  Let us denote this
collection $\{\sig(C_{\beta})\}$ of signatures (as $\beta$ varies over the eigenvalues
of $A$) by $\sig(A;C)$.

\begin{prop}\label{prop-q-conjugacy}
Let $\gamma_1,\gamma_2\in \GSp_{2n}(\QQ)$ be $q$-inversive, say
$\gamma_i = \left(\begin{smallmatrix} A_i & B_i \\ C_i & \tr{\!A}_i \end{smallmatrix} \right)$. Then

\begin{tabular}{lcl}
$\gamma_1,\gamma_2$ are $\GSp_{2n}(\overline{\QQ})$-conjugate 
&$\iff$ &they are $\GL_n^*(\overline{\QQ})$-conjugate\\
&$\iff$&$A_1,A_2$ are $\GL_n(\QQ)$-conjugate.\\
$\gamma_1,\gamma_2$ are $\Sp_{2n}(\RR)$-conjugate &$\iff$&$\gamma_1,\gamma_2$ are 
$\delta(\GL_n(\RR))$-conjugate\\
&$\iff$&$A_1,A_2$ are $\GL_n(\QQ)$-conjugate and\\
&&$\sig(A_1;C_1) = \sig(A_2;C_2)$.
\end{tabular}
\end{prop}
\begin{proof}  Taking $L = \overline{\QQ}$ in the preceding paragraph we can arrange that
$C = I$ and $B = (A^2-qI)$ which proves the first statement.  Taking $L=\RR$
gives the implication ($\implies$) in the second statement so we need to prove the reverse
implication. By replacing $\gamma_1, \gamma_2$ with $\delta(\GL_n(\RR))$-conjugates, we may assume
that $A_i, B_i, C_i$ are diagonal ($i=1,2$), the 
diagonal entries of $C_i$ consist of $\pm 1$, and repeated eigenvalues of $A_i$ are grouped together.
It follows that $A_1 = A_2$ since they have the same characteristic polynomial. 
We may express $\gamma_1 = \gamma_{1,1} \oplus \cdots \oplus \gamma_{1,m}$ as a direct sum of $q$-inversive
matrices of lower rank such that each corresponding $A_{1,j}$ is a scalar matrix.  In this way the
problem reduces to the case that $A_1 = A_2=\lambda.I_{n\times n}$ are scalar matrices, which we now suppose.  We may further assume that $C_1 = I_r$ consists of $r$ copies of $+1$ and $n-r$ copies of $-1$ along the
diagonal, and that $C_2 = I_s$.  This determines $B_1=dI_r$ and $B_2=dI_s$ where $d=\lambda^2-q$.  
Assuming that $\gamma_1,\gamma_2$ are $\Sp_{2n}(\RR)$-conjugate, we need to prove that $r=s$.

Suppose $h = \left(\begin{smallmatrix} X & Y \\ Z & W \end{smallmatrix} \right)\in \Sp_{2n}(\RR)$ and 
$\gamma_2 = h \gamma_1 h^{-1}$.  Subtracting $\lambda I_{2n \times 2n}$ from both sides 
of this equation leaves
\begin{equation}\label{eqn-tricky}
\left(\begin{matrix} X & Y \\ Z & W \end{matrix} \right)
\left(\begin{matrix} 0 & dI_r \\ I_r & 0 \end{matrix}\right) =
\left(\begin{matrix} 0 & dI_s \\ I_s & 0 \end{matrix} \right)
\left( \begin{matrix} X & Y \\ Z & W \end{matrix} \right) \end{equation}
or  $W = I_s X I_r$ and $Z = d^{-1} I_s Y I_r$.  Let 
$H = X + \frac{1}{\sqrt{d}} Y I_r \in \GL_{2n}(\CC)$.
Then  
\[H I_r \tr{\bar H} =  (X + {\textstyle{\frac{1}{\sqrt{d}}}} YI_r) I_r \tr{(}X - 
{\textstyle{\frac{1}{\sqrt{d}}}} YI_r) = I_s\]
for the real part of this equation comes from $X\tr{W}-Y\tr{Z}=I$ (\ref{eqn-symplectic-conditions3}) 
and the imaginary part follows similarly because $h \in \Sp_{2n}(\RR)$.  But $I_r$ and $I_s$ are Hermitian matrices so this equation implies that their signatures are equal, that is, $r=s$.  
\end{proof}

\subsection{}\label{subsec-delta}
Let $h(x)\in \ZZ[x]$ be a real, ordinary Weil $q$-polynomial,
(Appendix \ref{sec-Weil-polynomials}), that is, 
\begin{enumerate}
\item[(h1)] $h(0)$ is relatively prime to $q$
\item[(h2)] the roots $\beta_1,\beta_2,\cdots,\beta_n$ of $h$ are totally real and
$|\beta_i| <  2\sqrt{q}$ for $1 \le i \le n$. 
\end{enumerate}
Let $\mathcal S(h)$  be the algebraic variety, defined over $\QQ$, consisting of all pairs 
$(A_0,C)$ where $A_0,C \in \GL_n$, 
where $A_0$ is semisimple and its characteristic polynomial is equal to $h(2x)$,
where $C$ is symmetric and $\tr{\!A}_0C = CA_0$. As in Lemma \ref{lem-q-real} there is a natural mapping
\begin{equation}\label{eqn-theta}
\theta:\mathcal S(h)
\to \GSp_{2n},\quad (A_0,C) \mapsto \left( \begin{matrix} A_0 & B \\ C & \tr{A} 
\end{matrix} \right)\end{equation}
where $B = (A_0^2 -qI)C^{-1}$.  The image $\theta(\mathcal S(h)_{\QQ})$ of the set of rational 
elements consists of all $q$-inversive elements whose characteristic polynomial is the ordinary 
Weil $q$-polynomial $p(x) = x^nh(x+q/x)$  (see Appendix \ref{sec-Weil-polynomials}).
The image of $\theta$ is preserved by the action of $\GL_n$, which corresponds to the action
\[ X.(A_0,C) = (XA_0X^{-1}, \tr{\!X}^{-1} C X^{-1})\]
for $X \in \GL_n$.  In the notation of \S \ref{lem-q-real} above,
the orbits of $\GL_n(\RR)$ on $\mathcal S(h)_{\RR}$ are uniquely indexed by
the values $\left\{\sig(C_{\beta})\right\}$ of the signature of each of the 
quadratic forms $C_{\beta}$ on the eigenspace
$V_{\beta}$, as $\beta $ varies over the distinct roots of $h(x)$.  By abuse of
terminology we shall refer to the rational elements in the $\GL_n(\RR)$ orbit of
$(A_0,C) \in \mathcal S(h)_{\QQ}$ as the ``$\GL_n(\RR)$-orbit containing $(A_0,C)$".

\subsection{}\label{subsec-K1}
Let $(A_0,C_0) \in \mathcal S(h)_{\QQ}$ and set $\gamma = \theta(A_0)$ as in equation (\ref{eqn-theta}).  
The algebra $K=\QQ[\gamma]$ is isomorphic to a product of distinct CM fields corresponding to
the distinct irreducible factors of the characteristic polynomial of $\gamma$.  If 
$K_1$ is a CM field corresponding to a simple factor, say $K_1=\QQ[\gamma_1]$
then each eigenvalue  $\alpha\in\CC$ of $\gamma_1$ determines a homomorphism 
$\phi_{\alpha}:\QQ[\gamma_1]\to\CC$ by $\gamma_1 \mapsto\alpha$.  
Therefore a CM type $\Phi$ on $\QQ[\gamma]$ is determined by choosing a sign in 
equation (\ref{eqn-eigenvalues}) for each of the distinct eigenvalues $\beta$ of $A_0$ 
or equivalently, for each of the distinct roots of the polynomial $h(x)$.  Let us fix such a
choice and by abuse of notation, denote it also by $\Phi=\{\alpha_1,\cdots,\alpha_r\}$.  

 Recall from Appendix \ref{appendix-polarizations} that in
order for the pair $(\gamma, \B_0)$ (resp.~the triple $(\gamma, \B_0, \tau_0)$)
to give rise to a $\Phi$-polarized Deligne module (resp.~$\Phi$-polarized Deligne 
module with real structure), it is necessary and sufficient that $\gamma$ should be $\Phi$-viable.

\begin{prop}\label{prop-viable-pair}  Fix $h(x)$ and $\Phi$ as in \S \ref{subsec-delta}
and \S \ref{subsec-K1} above.
For any semisimple matrix $A_0\in\GL_n(\QQ)$ with characteristic polynomial equal to $h(2x)$
there exists a symmetric nonsingular element $C_0 \in \GL_n(\QQ)$ so that 
$(A_0,C_0) \in\mathcal S(h)_{\QQ}$ and so that $\gamma_0 = \theta(A_0,C_0) \in \GSp_{2n}(\QQ)$ 
is $\Phi$-viable.  For
every $(A,C) \in \mathcal S(h)_{\QQ}$ the corresponding element $\gamma = \theta(A,C)$ is $\Phi$-viable
if and only if it is $\delta(\GL_n(\RR))$-conjugate to $\gamma_0$.
\end{prop}

\begin{proof}  
Given $A_0$ we need to prove the existence of $C_0 \in \GL_n(\QQ)$ 
such that $(A_0,C_0)\in \mathcal S(h)$ is $\Phi$-viable.  By Proposition
 \ref{prop-real-existence} there is a $\Phi$-polarized 
Deligne module with real structure,  $(T,F,\B,\tau)$ whose characteristic
polynomial is $p(x)$.  Use Proposition \ref{prop-classification} to choose a 
basis $h:T\otimes \QQ \overset{\sim}{\rightarrow} \QQ^{2n}$ so that
that $h(T) \subset \QQ^{2n}$ is a lattice, so that $h_*(\B) = \B_0$ and that 
$h_*(\tau) = \tau_0$ in which case 
the mapping $F$ becomes a matrix $\gamma = \left(\begin{smallmatrix} A & B \\ C & \tr{\!A} 
\end{smallmatrix} \right)$.  It follows that $\gamma$ is viable and that the characteristic
polynomial of $A$ is equal to that of $A_0$.  So there exists $X \in \GL_n(\QQ)$ 
satisfying $A_0 = X A X^{-1}$.  Define $C_0 = \tr{\!X}^{-1} C {X}^{-1}$ so that $(A_0,C_0) = 
X\cdot(A,C)$.  Then $\gamma_0 = \theta(A_0,B_0) = \delta( X) \gamma \delta (X)^{-1}$
is $q$-inversive, its characteristic polynomial is $p(x)$, and by Proposition
\ref{prop-viable} it is viable.

For the second statement, $\gamma$ is $\Phi$-viable iff it is $\Sp_{2n}(\RR)$-conjugate to
$\gamma_0$, by Proposition \ref{prop-viable}.  This holds iff it is $\delta(\GL_n(\RR))$-conjugate to
$\gamma_0$, by Proposition \ref{prop-q-conjugacy}.
\end{proof}

\subsection{Remark} In the notation of the preceding paragraph, 
$\gamma=\theta(A,C)$ is $\Phi$-viable iff  $\sig(A,C) = \sig(A_0,C_0)$.  
If the roots of $h(x)$ are distinct then the CM field $\QQ[\gamma]$ has
$2^n$ different CM types, corresponding to the $2^n$ possible values of $\sig(A,C)$ (that is, 
an ordered $n$-tuple of $\pm 1$).  However, if $h(x)$ has repeated roots then there exist
elements $(A,C) \in \mathcal S(h)_{\QQ}$ such that $\gamma = \theta(A,C)$ is not viable for
any choice $\Phi$ of CM type on $\QQ[\gamma]$.




\section{$\overline{\QQ}$-isogeny classes}
The first step in counting the number of (principally polarized) Deligne modules (with or without
real structure) is to identify the set of $\overline{\QQ}$ isogeny classes of such modules, 
following the method of Kottwitz \cite{Kottwitz}. Throughout this and subsequent chapters we
shall only consider polarizations that are positive with respect to the CM type $\Phi_{\varepsilon}$
as described in \S \ref{subsec-phi-varepsilon}.
\begin{lem} \label{prop-five-part}
For $i=1,2$ let $(T_i,F_i)$ be a Deligne module with  ($\Phi_{\varepsilon}$-positive)
polarization $\B_i$. Let $p_i$ be the characteristic polynomial of $F_i.$
Then the following statements are equivalent.
\begin{enumerate}

\item\label{Q1} The characteristic polynomials are equal:  $p_1(x) = p_2(x).$
\item\label{Q2} The Deligne modules $(T_1,F_1)$ and $(T_2,F_2)$ are $\QQ$-isogenous.
\item\label{QN} The Deligne modules $(T_1,F_1)$ and $(T_2,F_2)$ are $\overline{\QQ}$-isogenous.
\item\label{Q3} The polarized Deligne modules $(T_1,F_1,\B_1)$ and $(T_2,F_2,\B_2)$ are
$\overline{\QQ}$-isogenous.
\end{enumerate}
For $i=1,2$ suppose the polarized Deligne module $(T_i,F_i,\B_i)$ admits a real structure 
$\tau_i$.   Then {\rm(\ref{Q1}), (\ref{Q2}), (\ref{QN}), (\ref{Q3})} are also equivalent 
to the following statements
\begin{enumerate}
\setcounter{enumi}{4}
\item\label{Q4} The real Deligne modules $(T_1,F_1,\tau_1)$ and $(T_2,F_2,\tau_2)$ are
$\QQ$-isogenous
\item\label{Q5} The real Deligne modules $(T_1,F_1,\tau_1)$ and $(T_2,F_2,\tau_2)$ are
$\overline{\QQ}$-isogenous 
\item\label{Q6} The real polarized Deligne modules $(T_1,F_1,\B_1,\tau_1)$
and $(T_2,F_2,\B_2,\tau_2)$ are $\overline{\QQ}$-isogenous.
\end{enumerate}\end{lem}
\begin{proof} 
Clearly, (\ref{Q3})$\implies$(\ref{QN})$\implies$(\ref{Q1}) and (\ref{Q2})$\implies$(\ref{Q1}).
The implication (\ref{Q1})$\implies$(\ref{Q2}) is a special case of a theorem of Tate, but in our
case it follows immediately from the existence of rational canonical form
(see, for example, \cite{Knapp1} p. 443) that is, by decomposing $T_i \otimes \QQ$ into
$F_i$-cyclic subspaces ($i=1,2$) and mapping cyclic generators in $T_1$ to corresponding
cyclic generators in $T_2$.

The proof that (\ref{Q2})$\implies$(\ref{Q3}) is a special
case of Kottwitz \cite{Kottwitz} p. 206, which proceeds as follows.
Given $\phi:(T_1\otimes\QQ,F_1) \to (T_2\otimes\QQ,F_2)$ define
$\beta  \in \End(T_1,F_1)\otimes\QQ$ by $\B_1(\beta x, y) = \B_2(\phi(x),\phi(y))$.
The {\em Rosati involution} ($\beta \mapsto \beta'$)
is the adjoint with respect to $\B_1$ and it fixes $\beta$ since
\[ \B_1(\beta' x, y) = \B_1(x, \beta y) = -\B_1(\beta y, x) = -\B_2(\phi(y),\phi(x)) = \B_1(\beta x, y).\]
By Lemma  \ref{lem-Rosati} there exists $\alpha \in
\End(T_1,F_1) \otimes \overline{\QQ}$ such that
$\beta = \alpha'\alpha$ which gives
\[\B_1(\alpha' \alpha x, y) = \B_1(\alpha x, \alpha y) = \B_2 (\phi(x), \phi(y)).\]
Thus $\phi \circ \alpha^{-1}:(T_1\otimes\overline{\QQ}, F_1) \to
(T_2 \otimes \overline{\QQ},F_2)$ is a
$\overline{\QQ}$-isogeny that preserves the polarizations.

Now suppose that real structures $\tau_1,\tau_2$ are provided.
It is clear that (\ref{Q6})$\implies$(\ref{Q5}) and (\ref{Q3}); also that
 (\ref{Q4})$\implies$(\ref{Q5})$\implies$(\ref{QN}).
  Now let us show (in the presence of $\tau_1,\tau_2)$) that (\ref{Q3})$\implies$(\ref{Q6}).  
The involution $\tau_i\in\GSp(T_i,\B_i)$ has multiplier $-1.$
So by Lemma \ref{lem-Darboux} and Proposition
\ref{prop-classification}  there exist
$\psi_i:T_i \otimes\QQ \to  \QQ^{2n}$ which takes the symplectic form $\B_i$ to the standard
symplectic form $\B_0$, and which takes the involution $\tau_i$ to the standard involution
$\tau_0.$
It therefore takes $F_i$ to some  $\gamma_i \in \GSp_{2n}(\QQ)$
which is \qreal with respect to the standard involution $\tau_0.$

 By part (\ref{Q3}) there is
a $\overline{\QQ}$ isogeny $\phi:(T_1,F_1,\lambda_1) \to (T_2,F_2,\lambda_2).$  This
translates into an element $\Phi \in \GSp_{2n}(\overline{\QQ})$ such that
$\gamma _2 = \Phi^{-1} \gamma_1 \Phi.$
\quash{
\begin{diagram}[size=2em]
T_1\otimes\overline{\QQ} & \rTo^{\phi} & T_2 \otimes \overline{\QQ}\\
\dTo^{\psi_1} &&\dTo_{\psi_2} \\
\overline{\QQ}^{2n} & \rTo_{\Phi} & \overline{\QQ}^{2n}
\end{diagram}
}
By Proposition \ref{prop-q-conjugacy} there
exists an element $\Psi \in \GL_n(\overline{\QQ})$ such that $\gamma_2 =
\Psi^{-1} \gamma_1 \Psi.$  In other words, $\Psi$ corresponds to a $\overline{\QQ}$-isogeny
$(T_1,F_1,\lambda_1,\tau_1) \to (T_2,F_2,\lambda_2,\tau_2).$ 

Now let us show that (\ref{Q5})$\implies$(\ref{Q4}).    
Let us suppose that $(T_1,F_1,\tau_1)$ and $(T_2,F_2,\tau_2)$ are
$\overline{\QQ}$-isogenous.  This implies that the characteristic polynomials $p_1(x)$
and $p_2(x)$ of $F_1$ and $F_2$ (respectively) are equal.  Moreover, Lemma
\ref{lem-tau-switches} implies that the $\pm 1$ eigenspaces of $\tau_1$ have the same
dimension (and that the same holds for $\tau_2$).  Set $V_1=T_1\otimes\QQ$ and
$V_2 = T_2\otimes\QQ$ and denote these eigenspace decompositions as follows,
\[ V_1 \cong V_1^{+} \oplus V_1^{-}\ \text{ and }\
V_2 \cong V_2^{+} \oplus V_2^{-}.\]
First let us consider the case that the characteristic polynomial $p_1(x)$ of $F_1$ is
irreducible.  In this case every non-zero vector in $V_1$ is a cyclic
generator of $V_1.$  Choose nonzero cyclic generators $v\in V_1^{+}$ and
$w \in V_2^{+},$  and define $\psi:V\to V_2 $ by
\[ \psi(F_1^r v) = F_2^r w\]
for $1 \le r \le \dim(T).$  This mapping is well defined because $F_1$ and $F_2$ satisfy
the same characteristic polynomial.  Clearly, $\psi \circ F_1 = F_2 \circ \psi.$  However
we also claim that $\psi \circ \tau_1 = \tau_2 \circ \psi.$  It suffices to check this on
the cyclic basis which we do by induction.  By construction we have that
$\psi \tau_1 v = \tau_2 \psi v = \tau_2 w$ so suppose we have proven that
$\psi \tau_1 F_1^{m}v =  \tau_2 \psi F_1^{m}v = \tau_2 F_2^mw$ for all $m \le r-1.$  Then
\begin{align*}
\psi \tau_1 F_1^r v &= \psi \tau_1 F_1 \tau_1^{-1} \tau_1 F_1^{r-1}v = 
q \psi F_1^{-1}\tau_1 F_1^{r-1} v\\
&= q F_2^{-1}\psi \tau_1 F_1^{r-1} v
= q F_2^{-1}\tau_2 \psi F_1^{r-1}v\\
&= \tau_2 F_2 \psi F_1^{r-1} v = \tau _2 \psi F_1^r v.
\end{align*}
Thus we have constructed a $\QQ$ isogeny between these two real Deligne modules.

Now let us consider the case that the characteristic polynomial $p_1(x)$ is reducible.
In this case, $F_1$ is semisimple so its minimum polynomial
 is a product of distinct irreducible polynomials.  Then there are decompositions
\[ V_1 \cong \underset{i=1}{\overset{k}{\oplus}}V_{1,i} \ \text{ and }\
V_2 \cong \underset{i=1}{\overset{k}{\oplus}}V_{2,i}\]
into $F_1$-cyclic (resp. $F_2$-cyclic) subspaces with irreducible characteristic polynomials.
Moreover, the involution $\tau_1$ (resp. $\tau_2$) preserves this decomposition:  this
follows from the fact that $\tau_1 F_1 \tau_1^{-1} = qF_1^{-1}$ satisfies the same characteristic
polynomial as $F_1$.  Thus we may consider the factors one at a time and this reduces
us to the case that the characteristic polynomial $p_1(x)$ is irreducible.
 \end{proof}

\begin{prop}\label{prop-Qbar-conjugacy}
  Associating the characteristic polynomial to each Deligne module induces a
canonical one to one correspondence between the following objects
\begin{enumerate}[label={\rm(\alph*)}]
\item The set of ordinary Weil $q$-polynomials $p(x)\in\ZZ[x]$ of degree $2n$ (see Appendix 
\ref{sec-Weil-polynomials})
\item The set of $\GSp_{2n}(\overline{\QQ})$-conjugacy classes of 
semisimple elements $\gamma \in \GSp_{2n}(\QQ)$ 
whose characteristic polynomial is an ordinary Weil $q$-polynomial
\item The set of $\QQ$-isogeny classes of Deligne modules $(T,F)$
\item The set of $\overline{\QQ}$-isogeny classes of ($\Phi_{\varepsilon}$-positively)
polarized Deligne modules $(T,F,\lambda)$
\end{enumerate}
and a one to one correspondence between the following objects
\begin{enumerate}[label={\rm($\rm\alph*^{\prime}$)}]
\item \label{S1} The set of ordinary real Weil $q$-polynomials (see Appendix 
\ref{sec-Weil-polynomials}) of degree $n$ 
\item \label{S2} The set of $\GL_n(\QQ)$-conjugacy classes of semisimple elements $A_0 \in \GL_n(\QQ)$
whose characteristic polynomial is $h(2x)$ where $h$ is an ordinary real Weil $q$-polynomial. 
\item \label{S4} The set of $\QQ$-isogeny classes of Deligne modules $(T,F,\tau)$ with real structure
\item \label{S5} The set of $\overline{\QQ}$-isogeny classes of ($\Phi_{\varepsilon}$-positively) 
polarized Deligne modules $(T,F,\B,\tau)$ with real structure.
\end{enumerate}
\end{prop}
\begin{proof}  The correspondence (a)$\rightarrow$(b) is given by Proposition \ref{prop-companion}
(companion matrix for the symplectic group) while (b)$\rightarrow$(a) associates to $\gamma$ its
characteristic polynomial.  This correspondence is one-to-one because semisimple elements in $\GSp_{2n}(\overline{\QQ})$ are conjugate iff their characteristic polynomials are equal.  
Items (b) and (c) are identified by the Honda-Tate theorem (\cite{Tate}), but can also be seen directly.
Given $\gamma$ one constructs a lattice $T\subset\QQ^{2n}$ that is preserved by $\gamma$
and by $q\gamma^{-1}$ by considering one cyclic subspace at a time (cf. Proposition 
\ref{prop-real-existence}) and taking $T$ to be the lattice spanned by $\{\gamma^mv_0\}$ 
and by $\{(q\gamma)^mv_0\}$ where $v_0$ is a cyclic vector. Lemma \ref{prop-five-part} 
may be used to complete the proof that the correspondence is one to one.  
Items (c) and (d) correspond by Proposition \ref{prop-real-existence} 
(existence of a polarization) and by Lemma \ref{prop-five-part}.  

The correspondence \ref{S1}$\leftrightarrow$\ref{S2} is standard.  For the correspondence 
\ref{S1}$\rightarrow$\ref{S4}, each ordinary real Weil $q$-polynomial $h(x)$ is the real counterpart of
an ordinary Weil $q$-polynomial $p(x)$ by Appendix \ref{sec-Weil-polynomials}.  It suffices to consider the
case that $p(x)$ is irreducible.  Let $\pi$ be a root of $p(x)$ so that $K = \QQ[\pi]$ is a CM field.
Set $T = \mathcal O_K$ (the full ring of integers), let $F = \pi:T \to T$ be multiplication by $\pi$ and
let $\tau$ denote complex conjugation.  Then $\tau$ preserves $\mathcal O_K$ and $\tau F \tau =
q F^{-1}$ because $\pi \bar\pi = q$.  Hence $(T,F,\pi)$ is a Deligne module with real structure whose
characteristic polynomial is $p(x)$. Lemma \ref{prop-five-part} says that this association
\ref{S1}$\rightarrow$\ref{S4} is one to one and onto.   A mapping \ref{S4}$\rightarrow$\ref{S5} is
given by Proposition \ref{prop-real-existence} (existence of a polarization) and this mapping is
one to one and onto by Lemma \ref{prop-five-part}. 
\end{proof}



\section{$\QQ$-isogeny classes within a $\overline{\QQ}$ isogeny class}  
\subsection{}  Let us fix a ($\Phi_{\varepsilon}$-positively) polarized Deligne module, 
which (by Lemma \ref{subsec-real-standard-form}) we may take to be
of the form $(L_0, \gamma_0, \B_0)$ (where $L_0 \subset \QQ^{2n}$ is a lattice and
where $\gamma_0 \in \GSp_{2n}(\QQ)$), to be used as a basepoint within its
$\overline{\QQ}$-isogeny class.  

\begin{prop}\label{prop-Q-Qbar} The choice of basepoint $(L_0,\gamma_0, \B_0)$ 
determines a canonical one to one correspondence  between the following (possibly infinite) sets.
\begin{enumerate}
\item The set of $\QQ$-isogeny classes of polarized Deligne modules $(T,F,\lambda)$
within the $\overline{\QQ}$-isogeny class of $(L_0,\gamma_0,\B_0)$
\item the set of $\Gsp_{2n}(\QQ)$-conjugacy classes of elements $\gamma \in \GSp_{2n}(\QQ)$
such that $\gamma, \gamma_0$ are conjugate by an element in $\GSp_{2n}(\RR)$
\item the elements of
$\ker\left(H^1(\Gal(\overline{\QQ}/\QQ), Z(\gamma_0))
\to H^1(\Gal(\CC/\RR), Z(\gamma_0)\right)$
\end{enumerate}
where $Z(\gamma_0)$ denotes the centralizer in $\GSp_{2n}$ of $\gamma_0$.  Moreover, $\GSp_{2n}(\RR)$
may be replaced by $\Sp_{2n}(\RR)$ in statement (2).
\end{prop}

\begin{proof}  For the correspondence (1)$\leftrightarrow$(2), given a polarized 
Deligne module $(T,F,\B)$ choose coordinates on $T\otimes\QQ \cong \QQ^{2n}$ so 
that $\B$ becomes the standard symplectic form $\B_0$.  Then  the resulting
element $\gamma \in \GSp_{2n}(\QQ)$ is well defined up to $\Sp_{2n}(\QQ)$-conjugacy and
by Proposition \ref{prop-q-conjugacy} it will be $\GSp_{2n}(\RR)$-conjugate to 
$\gamma_0$.  Moreover, $\QQ$-isogenous Deligne modules correspond to 
$\GSp_{2n}(\QQ)$-conjugate elements. Conversely if $\gamma \in \GSp_{2n}(\QQ)$ is 
$\GSp_{2n}(\RR)$-conjugate to $\gamma_0$ then by Proposition \ref{prop-viable}
it is viable so it comes from a polarized Deligne module.  For the 
``moreover'' part of the proposition,
if $\gamma = h \gamma_0 h^{-1}$ with $h \in \GSp_{2n}(\RR)$ then replacing $\gamma$ by the
$\Sp_{2n}(\QQ)$-conjugate element, $\tilde\gamma = \tau_0 \gamma \tau_0^{-1}$ if necessary, we may assume
that the multiplier $c(h)>0$ is positive, hence $\gamma = h' \gamma_0 (h')^{-1}$ where $h' = c(h)^{-1/2}h \in \Sp_{2n}(\RR)$. 
The correspondence between (2) and (3) is standard:  the set $H^1(\Gal(\overline{\QQ}/\QQ, Z(\gamma_0))$ indexes
$\GSp_{2n}(\QQ)$-conjugacy classes of elements $\gamma$ that are $\GSp_{2n}(\overline{\QQ})$-conjugate to
$\gamma_0$, and such a class becomes trivial in $H^1(\Gal(\CC/RR), Z(\gamma_0))$ if and only if $\gamma$ is
$\GSp_{2n}(\RR)$-conjugate to $\gamma_0$.
\end{proof}

\subsection{}
\label{sec-real-Q-isogeny}
This section is parallel to the preceding section but it incorporates a real structure.
Let us fix a polarized Deligne module with real structure, $(L_0,\gamma_0,\B_0,\tau_0)$ where 
\[
\gamma_0 = \left(\begin{matrix} A_0 & B_0 \\ C_0 & \tr{\!A}_0
\end{matrix}\right) \]
is \qreal, $\B_0$ is the standard symplectic form, $\tau_0$ is the standard involution, and
$L_0 \subset \QQ^{2n}$ is a lattice preserved by $\tau_0$, by $\gamma_0$ and by $q \gamma_0^{-1}$,
on which $\B_0$ takes integer values.
\quash{
where $A_0, B_0, C_0$ are nonsingular, where 
$B_0,C_0$ are symmetric and $A_0^2-B_0C_0 = qI.$  Since $(L_0,\gamma_0)$ is a
Deligne module corresponding to an ordinary Abelian variety it also follows from
Lemma \ref{lem-q-real} and Proposition \ref{prop-h(x)} that
the matrix $A_0 \in \GL_{n}(\QQ)$ is semisimple, its characteristic polynomial $h(x)$ is totally 
real, and the constant term of $h(x)$ is coprime to $q.$  
}
Let $Z_{\scriptstyle\rm{GL_n}(\QQ)}(A_0)$
denote the set of elements in $\GL_n(\QQ)$ that commute with $A_0$.

\begin{prop}\label{prop-real-Q-isogeny}  The association $B \mapsto \gamma = \left(
\begin{smallmatrix} A_0 & B \\ C & \tr{\!A}_0 \end{smallmatrix} \right)$, where
$C = B^{-1}(A_0 - qI)$, determines a one to one correspondence between the following sets,
\begin{enumerate}
\item elements $B \in \GL_n(\QQ)$,  one from
each $Z_{\scriptstyle\rm{GL_n}(\QQ)}(A_0)$-congruence class of matrices such that \begin{enumerate}
\item $B$ is symmetric and nonsingular
\item $A_0B=B\tr{\!A}_0$
\item $\sig(B;A_0) = \sig(B_0;A_0)$.
\end{enumerate}
\item  The set of $\QQ$ isogeny classes of real polarized Deligne modules
 $(T,F,\lambda,\tau)$ within the
$\overline{\QQ}$ isogeny class of $(L_0,\gamma_0,\B_0,\tau_0)$
\item The set of $\GL^*_{n}(\QQ)$-conjugacy classes of \qreal elements $\gamma \in
\GSp_{2n}(\QQ)$ such that
\begin{enumerate}
\item $\gamma, \gamma_0$ are conjugate by some element in  $\GL^*_{n}(\overline{\QQ})$
\item $\gamma, \gamma_0$ are conjugate by some element in $\GSp_{2n}(\RR)$
\end{enumerate}
\item The set of  $\GL^*_{n}(\QQ)$-conjugacy classes of \qreal elements $\gamma \in
\GSp_{2n}(\QQ)$ such that
\begin{enumerate}
\item $\gamma,\gamma_0$ are conjugate by some element in $\GL_{n}^*(\RR)\subset \Sp_{2n}(\RR)$
\end{enumerate}
\item the elements of $\ker\left(
H^1(\Gal(\overline{\QQ}/\QQ), I_0) \to H^1(\Gal(\CC/\RR), I_0)
\right)$
\end{enumerate}
where $I_0$ denotes the group of self isogenies of $(L_0,\gamma_0,\B_0,\tau_0)$,
that is,
\begin{equation}\label{eqn-identify-centralizer}
I_0 = Z_{\sGL_n^*}(\gamma_0) \cong Z_{\sGL_n}(A_0) \cap \GO(B_0)
\end{equation} where
\[ \GO(B_0) = \left\{ X \in \GL_n|\ X B_0 \tr{\!X} = \mu B_0
\text{ for some constant } \mu\ne 0\right\} \]
denotes the general orthogonal group defined by the symmetric matrix $B_0.$
\end{prop}

\begin{proof}
The isomorphism of equation \eqref{eqn-identify-centralizer}
follows immediately from equation (\ref{eqn-Gln-action}).
The equivalence of (3) and (4) follows from Proposition \ref{prop-q-conjugacy}.
The equivalence of (4) and (5) is similar to that in the proof of Proposition \ref{prop-Q-Qbar}
and it is standard.

To describe the correspondence (1)$\rightarrow$(4),  given $B$ set $\gamma =
\left(\begin{smallmatrix} A_0 & B \\ C & \tr{\!A}_0 \end{smallmatrix}\right)$ where
$C = B^{-1}(A^2-qI)$.  Since  $\sig(B;A_0) = \sig(B_0;A_0)$,  Proposition \ref{prop-q-conjugacy} 
implies that $\gamma, \gamma_0$ are conjugate by an element of $\delta(\GL_n(\RR))\subset \GL_n^*(\RR)$.

 Conversely, let $\gamma
= \left(\begin{smallmatrix} A & B \\ C & \tr{\!A} \end{smallmatrix} \right) \in \GSp_{2n}(\QQ)$
be $q$-inversive and $\GL_n^*(\RR)$-conjugate to $\gamma_0$.  
Then $A,A_0$ are conjugate by an element of
$\GL_n(\RR)$ so they are also conjugate by some element $Y \in \GL_n(\QQ)$.  Replacing
$\gamma$ with $\delta(Y) \gamma \delta(Y)^{1}$ (which is in the same $\delta(\GL_n(\QQ))$-conjugacy
class), we may therefore assume that $A = A_0$.  Proposition \ref{prop-q-conjugacy} then
says that $\sig(B;A_0) = \sig(B_0;A_0)$.  So we have a one to one correspondence
(1)$\leftrightarrow$(4).

We now consider the correspondence (2)$\rightarrow$(3).  Let $(L,\gamma, \B_0,\tau_0)$ 
be a polarized Deligne module with real structure that is
$\overline{\QQ}$-isogenous to $(L_0,\gamma_0,\B_0,\tau_0)$.
A choice of $\overline{\QQ}$ isogeny $\phi:(L,\gamma,\B_0,\tau_0) \to
(L_0,\gamma_0,\B_0,\tau_0)$ is an element $X \in \GSp_{2n}(\overline{\QQ})$
such that $\tau_0 X \tau_0^{-1} = X$ and such that $\gamma = X \gamma_0 X^{-1}.$
This proves part (3a), that $\gamma, \gamma_0$ are conjugate by an element
$X \in \GL^*_{n}(\overline{\QQ}).$ Proposition \ref{prop-Q-Qbar} 
 says that $\gamma,\gamma_0$ are also conjugate by an element
of $\GSp_{2n}(\RR).$   Moreover, the isogeny $\phi$ is a $\QQ$-isogeny
if and only if $X \in \GL^*_{n}( \QQ).$  Thus we have described a mapping
from the elements in (2) to the elements in (3).

To describe the inverse of this correspondence we start with the basepoint
$(L_0,\gamma_0,\B_0,\tau_0)$ and choose an element $\gamma \in \GSp_{2n}(\QQ)$ from
the set of elements that satisfy conditions (3a) and (3b), that is,
\begin{equation}\label{eqn-help1}
 \gamma = g \gamma_0 g^{-1} = t \gamma_0 t^{-1} = h \gamma_0 h^{-1}\end{equation}
for some $g \in \GL^*_{n}(\overline{\QQ})$, some $t \in \GL_{2n}(\QQ)$ and some
$h \in \GSp_{2n}(\RR).$ As in the proof of Proposition \ref{prop-Q-Qbar},
 by replacing $\gamma$ with $\tau_0 \gamma \tau_0^{-1}$ if
necessary, {\em we may assume that  $h \in \Sp_{2n}(\RR)$}.
 Since $\tau_0 \in \GL_{2n}(\ZZ)$ the set
\[ L':= \left( tL_0 \right) \cap \left( \tau_0 t L_0\right)\subset \QQ^{2n}\]
is a lattice, so there exists an integer $m$ such that $\B_0$ takes integer values on
$L := mL'$. We claim that $(L = mL', \gamma, \B_0,\tau_0)$ is a RPDM. 
 The verification is the same as that
in the proof of Proposition \ref{prop-viable} except that we also need to verify that
$\gamma L \subset L$ and that $q\gamma^{-1}L \subset L$, which follows from
equation (\ref{eqn-help1}).  \end{proof}
\quash{
If $x \in tL_0$ then
$\gamma x \in tL_0$ and $q\gamma^{-1}x \in tL_0$ as in equation (\ref{eqn-L-preserved}).  If $x \in \tau_0 tL_0$ then
\[ \gamma x \in \gamma \tau_0 tL_0 = \tau_0 q \gamma^{-1}tL_0 \subset
\tau_0 tL_0 \]
which shows that $\gamma$ preserves $(tL_0) \cap (\tau_0tL_0)$.  The proof
that $q \gamma^{-1}$ preserves $L$ is similar.}

There may be infinitely many $\QQ$-isogeny classes of polarized
Deligne modules with real structure within a given $\overline{\QQ}$-isogeny class.  However, it 
follows from Theorem \ref{prop-finite-isomorphism} that only finitely many of these 
$\QQ$-isogeny classes contain principally polarized modules.

\section{Isomorphism classes within a $\QQ$-isogeny class}\label{sec-isomorphism-classes}
\subsection{The category $\mathcal P_N$}  In this chapter and in all subsequent chapters we
fix $N \ge 1$, not divisible by $p$. Throughout this chapter we fix a 
($\Phi_{\varepsilon}$-positively) polarized Deligne 
module (over $\FF_q$) with real structure, which (by Lemma \ref{subsec-real-standard-form}) we may assume 
to be of the form $(T_0, \gamma_0, \B_0, \tau_0)$ where $T_0 \subset \QQ^{2n}$ is a lattice, 
$\gamma_0 \in \GSp_{2n}(\QQ)$ is a semisimple element whose characteristic polynomial 
is an ordinary Weil $q$-polynomial, and where $\B_0$ is the standard symplectic form and 
$\tau_0$ is the standard involution.  Following the method of  \cite{Kottwitz} we consider the
category $\mathcal P_N(T_0, \gamma_0, \B_0, \tau_0)$ for which an object is a collection
$(T,F, \B, \beta, \tau, \phi)$ where $(T,F,\B,\beta, \tau)$ is a {\em principally} 
polarized Deligne module with real structure $\tau$ and with principal level $N$
structure $\beta: T/NT \to (\ZZ/N\ZZ)^{2n}$ (that is compatible with $F\mod N$, with the symplectic 
structure $\overline{\B}=\B \mod N$ and with the real structure $\tau$, see \S \ref{def-real-structure}), 
and where $\phi:(T,F,\B,\tau) \to (T_0,\gamma_0, \B_0, \tau_0)$ is a  $\QQ$-isogeny of polarized
Deligne modules with real structure, meaning that:
\[\phi:T\otimes \QQ \overset{\sim}{\to} \QQ^{2n},\ 
\phi F = \gamma_0 \phi,\ 
\phi^*(\B_0) = c\B \text{ for some } c\in \QQ^{\times}, \text{ and } 
\phi \tau = \tau_0 \phi.\]
A morphism $\psi:(T,F,\B, \beta, \tau, \phi) \to 
(T', F', \B', \beta',\tau', \phi')$  is a group homomorphism
$\psi:T \hookrightarrow T'$ such that 
\[ \phi = \phi' \psi \text{ (hence } \psi  F = F' \psi),\
\B = \psi^*(\B'),\
\beta = \beta' \circ \psi, \text{ and }
\psi \tau = \tau' \psi.\]
Let $X$ denote the set of isomorphism classes in this category.  We obtain a natural 
one to one correspondence between the set of isomorphism classes of principally polarized 
Deligne modules with level structure and real structure within the 
$\QQ$-isogeny class of $(T_0, \gamma_0, \B_0, \tau_0)$, and the quotient
\begin{equation}\label{eqn-main-quotient}
 I_{\QQ}\backslash X\end{equation}
where $I_{\QQ}=I_{\QQ}(T_0,\gamma_0,\B_0,\tau_0)$ denotes the group of self $\QQ$-isogenies of $(T_0, \gamma_0,
\B_0,\tau_0)$.

\subsection{The category $\mathcal L_N$}\label{subsec-LN}
  Let $\mathcal L_N(\QQ^{2n}, \gamma_0,\B_0, \tau_0)$ be the 
category for which an object is a pair $(L,\alpha)$ where $L\subset T_0 \otimes\QQ$ is a lattice
that is symplectic (up to homothety), is preserved by $\gamma_0$, by $q \gamma_0^{-1}$ and by $\tau_0$,
and $\alpha:L/NL \to (\ZZ/N\ZZ)^{2n}$ is a compatible level structure, that is:
\[\tau_0L = L,\ \bar\tau_0 \alpha = \alpha \bar\tau_0,\ \gamma_0L \subset L,\
q \gamma_0^{-1}(L) \subset L,\ \alpha \gamma_0 = \alpha,\]
and there exists $c \in \QQ^{\times}$ so that
\[ L^{\vee} = cL\  \text{ and }\ \alpha_*(c\B_0) = \overline{\B}_0.\]

A morphism $(L,\alpha) \to (L',\alpha')$ is an inclusion $L \subset L'$ such
that $\alpha'|(L/NL) = \alpha$.  (Since $L \to L'$ is an inclusion it also commutes with $\gamma_0$
and $\tau_0$, and it preserves the symplectic form $\B_0$.)  In this
category every isomorphism class contains a unique object.

\subsection{The category $\widehat{\mathcal L}_N$}\label{subsec-adelic-category}
Given $\gamma_0 \in \GSp_{2n}(\QQ)$ as above, let
$\widehat{\mathcal L}_N(\AA_f^{2n}, \gamma_0, \B_0, \tau_0)$ be the category for which an object
is a pair $(\widehat{L}, \alpha)$ consisting of a lattice $\widehat{L} \subset \AA_f^{2n}$ that is
symplectic (up to homothety) and is preserved by $\gamma_0$, by $q \gamma_0^{-1}$ and by $\tau_0$,
and a compatible level $N$ structure $\alpha$, that is:\setcounter{equation}{-1}
\begin{equation}\label{eqn-9.3.0}
\tau_0 \widehat{L} = \widehat{L},\
\alpha \bar\tau_0 = \bar\tau_0\alpha,\end{equation}
\begin{equation}\label{eqn-9.3.1}
\gamma_0 \widehat{L} \subset \widehat{L},\
q\gamma_0^{-1} \widehat{L} \subset \widehat{L},\
\alpha \circ \gamma_0 = \alpha
\end{equation}
and there exists 
$c \in \QQ^{\times}$ such that
\begin{equation}\label{eqn-9.3.2}
\widehat{L}^{\vee} =  c \widehat{L}\ \text{ and }\
\alpha_*(c\B_0) = \overline{\B}_0.\end{equation}

A morphism in $\widehat{\mathcal L}_N$ is an inclusion $\widehat{L} \subset \widehat{M}$
that is compatible with the level structures. As in \cite{Kottwitz} we have the following:

\begin{prop}\label{prop-iso-in-isogeny}
The association 
\[(T,F,\B, \beta,\tau,\phi) \mapsto  \left(L=\phi(T), \alpha = \beta\circ\phi^{-1}\right)
\mapsto (\widehat{L} = {\textstyle\prod_v} L \otimes \ZZ_v, \alpha) \] 
determines  equivalences of categories
\[\mathcal P_N(T_0,\gamma_0,\B_0,\tau_0) \to \mathcal L_N(\QQ^{2n},\gamma_0,\B_0,\tau_0)
\to \widehat{\mathcal L}_N(\AA_f^{2n}, \gamma_0, \B_0, \tau_0).\]
\end{prop}

\begin{proof}  Given $(T,F,\B,\beta,\phi)$ let $L = \phi(T)$ and $\alpha = \beta \phi^{-1}$.
Then $\gamma_0L = \gamma_0\phi(T)
= \phi(FT) \subset \phi(T) = L$ and similarly $q \gamma_0^{-1}L\subset L.$
 Since $\B$ is a principal polarization we obtain
\[T = T^{\vee} = \left\{ u \in T \otimes \QQ|\ \B(u,v)\in \ZZ \
\text{ for all }\ v \in T\right\}.\]
Since $\phi$ is a $\QQ$-isogeny with multiplier $c \in \QQ^{\times}$ we have
that $\B_0(\phi(x),\phi(y)) =c \B(x,y)$ for all $x, y \in T \otimes\QQ$ so
\begin{align*}
L^{\vee} &= 
\left\{u \in L \otimes \QQ|\ \B_0(u,v)\in\ZZ \text{ for all } v \in L   \right\} \\
&= \left(\phi(T)\right)^{\vee} = c\phi(T^{\vee}) = c\phi(T) = cL.
\end{align*}
This implies that $c\B_0$ is integral valued on $L$ and hence
\[ \alpha_*(c\B_0) = \beta_*\phi^*(c\B_0) = \beta_*(\B) = \overline{\B}_0.\]
Hence the pair $(L,\alpha)$ is an object in $\mathcal L_N(\QQ^{2n}, \gamma_0,\B_0, \tau_0)$.
If $\psi:(T,F,\B,\phi) \to
(T',F',\B',\phi')$ is a morphism in $\mathcal P_N(T_0, \gamma_0, \B_0, \tau_0)$
then $\psi(T) \subset T'$ so $L = \phi(T) \subset L' = \phi(T')$ is a morphism
in $\mathcal L_N(\QQ^{2n}, \gamma_0,\B_0, \tau_0)$.

Conversely, given an object $(L,\alpha)$ in $\mathcal L_N$, that is, a
lattice $L \subset \QQ^{2n}$ preserved by $\gamma_0$ and $q\gamma_0^{-1}$ such
that $L^{\vee} = cL$, and a principal level structure $\alpha:L/NL \to (\ZZ/N\ZZ)^{2n}$
such that $\alpha_*(c\B_0) = \overline{\B}_0$, we obtain an object in $\mathcal P_N$,
\[ \left(T = L,\ F = \gamma_0|L,\ \B = c \B_0,\
\beta=\alpha,\ \phi=id\right)\]
such that $T^{\vee} = {\textstyle{\frac{1}{c}}}L^{\vee} = L = T$ and
such that $\beta_*(\B) = \alpha_*(c\B_0)= \overline{\B}_0$.  It follows that $\mathcal P_N \to
\mathcal L_N$ is an equivalence of categories.  

Finally, the functor $\mathcal L_N \to  \widehat{\mathcal L}_N$ is an equivalence of categories 
by Lemma \ref{lem-Platonov}.
\end{proof}

\subsection{Lattices at $p$}\label{subsec-lattices-at-p}
Let $\gamma_0 \in \GSp_{2n}(\QQ)$ be a semisimple element whose characteristic polynomial
is an ordinary Weil $q$-polynomial.  It induces a decomposition $\QQ_p^{2n} \cong W' \oplus
W^{\pp}$ that is preserved by $\gamma_0$ but exchanged by $\tau_0$, 
where the eigenvalues of $\gamma_0|W'$ are
$p$-adic units and the eigenvalues of $\gamma_0|W^{\pp}$ are non-units.  Define
\begin{equation}\label{eqn-alphaq}
\alpha_q = \alpha_{\gamma,q} = I' \oplus qI^{\pp}\end{equation} 
to be the identity on $W'$ and multiplication by $q$ on $W^{\pp}$.
The following lemma, which is implicit in \cite{Kottwitz} will be used 
in \S \ref{sec-counting-lattices} when we count the number of lattices in the category
$\widehat{\mathcal L}_N$.

\begin{lem}\label{lem-Smith-form}  Let $\gamma_0 \in \GSp_{2n}(\QQ)$ be a semisimple element
with multiplier equal to $q$.
Let $L_{0,p}= \ZZ_p^{2n}$ be the standard lattice in $\QQ_p^{2n}$.
Let $g \in G(\QQ_p)$ and let $c$ denote its multiplier.   Let $L = g(L_{0,p})$.  
Then $L^{\vee} = c^{-1}L$ and the following statements are equivalent:
\begin{enumerate}[label={\rm(\alph*)}]
\item The lattice $L$ is preserved by $\gamma_0$ and by $q\gamma_0^{-1}$.
\item  The lattice $L$ satisfies $qL \subset \gamma_0 L \subset L$.
\item $g^{-1} \gamma_0 g \in K_p A_q K_p$
\end{enumerate}
where $K_p = G(\ZZ_p)$ and $A_q = \left(\begin{smallmatrix} I & 0 \\ 0 & qI
\end{smallmatrix} \right)$. If the characteristic polynomial of $\gamma_0$ is an ordinary
Weil $q$-polynomial then conditions (a),(b),(c) above are also equivalent to:  \begin{enumerate}
\item[{\rm(d)}] $g^{-1} \alpha_q^{-1} \gamma_0 g \in K_p$
\end{enumerate}

\end{lem}

\begin{proof}  Clearly (a) and (b) are equivalent, and also to:  $qL \subset
q\gamma_0^{-1}L \subset L$. Hence\begin{enumerate}
\item[${\rm(b^{\prime})}$] $L/\gamma_0 L \cong \gamma_0^{-1}L/L \subset L/qL\cong (\ZZ/q\ZZ)^{2n}$.
\end{enumerate}
We now show that (b)$\implies$(c).  Since $\det(\gamma_0)^2 = q^{2n}$ we know that
$|\det(\gamma_0)| = \left| L/\gamma_0 L\right| = q^n$.  Condition (b) implies that $L/\gamma_0 L$ consists of
elements that are killed by multiplication by $q$.  Condition $(b^{\prime})$ implies that $L/\gamma_0L$
is free over $\ZZ/q\ZZ$.  Therefore 
\begin{equation}\label{eqn-ZqZn}
L_{0,p}/(g^{-1} \gamma_0 g)L_{0,p} \cong L/\gamma_0 L \cong (\ZZ/q\ZZ)^n.\end{equation}  
By the theory of Smith normal form for the symplectic group
(see \cite{Spence} or \cite{Andrianov} Lemma 3.3.6) we may
write $g^{-1} \gamma_0 g = u D v$ where $u,v \in G(\ZZ_p)$ and $D = \diag
(p^{r_1}, p^{r_2},\cdots, p^{r_{2n}})$ where $r_1 \le r_2 \le \cdots \le r_{2n}$.  This, together
with equation (\ref{eqn-ZqZn}) implies that $r_1=\cdots=r_n=0$ and $r_{n+1} = \cdots = r_{2n}=a$,
that is, $D = A_q$.  This proves that (b) implies (c). 

Now let us show that (c) implies (a).  Since $K_pA_q K_p\subset
M_{2n \times 2n}(\ZZ_p)$, condition (c) implies that $\gamma_0 g L_{0,p} \subset gL_{0,p}$.
Taking the inverse of condition (c) and multiplying by $q$ gives
\[ qg^{-1} \gamma_0^{-1} g \in K_pqA_q^{-1}K_p \subset M_{2n\times 2n}(\ZZ_p)\]
which implies that $q \gamma_0^{-1}L \subset L$.

Finally, if the characteristic polynomial of $\gamma_0$ is an ordinary Weil 
$q$-polynomial then the lattice $L$ decomposes (\cite{Deligne}) into 
$\gamma_0$-invariant sublattices, $L=L' \oplus L^{\pp}$ such that
$\gamma_0|L'$ is invertible and $\gamma_0|L^{\pp}$ is divisible by $q$, 
or $\gamma_0 L' = L'$ and $\gamma_0 L^{\pp} \subset q L^{\pp}$ which, 
in light of (d) implies that $\gamma_0L^{\pp} = qL^{\pp}$.
In summary, $\alpha_q^{-1}\gamma_0L = L$, which is equivalent to (d).
\end{proof}

\section{Counting lattices}\label{sec-counting-lattices}\subsection{}\label{subsec-KpKN}
Throughout this and subsequent chapters all polarizations are assumed to be $\Phi_{\varepsilon}$-positive.
In this section we fix a ($\Phi_{\varepsilon}$-positively) polarized Deligne module, which 
we may assume (cf.~Lemma \ref{subsec-real-standard-form}) to be of the form
 $(T, \gamma, \B_0)$ where $T\subset \QQ^{2n}$ is a lattice,  where $\B_0$ is the
standard symplectic form and where $\gamma \in \GSp_{2n}(\QQ)$ is an element whose characteristic
polynomial is an ordinary Weil $q$-polynomial.  If we remove all mention of the real structure 
($\tau, \tau_0$, etc.) in the preceding section, \S \ref{sec-isomorphism-classes}, then Proposition 
\ref{prop-iso-in-isogeny} becomes an equivalence of categories between the category
(which we will denote $\mathcal P_N(T, \gamma, \B_0)$) of principally polarized Deligne modules
with (prime to $q$) level $N$ structure together with a $\QQ$-isogeny to $(T, \gamma, \B_0)$, and the
category (which we will denote $\widehat{\mathcal L}_N(\AA_f^{2n}, \gamma, \B_0)$).  
Thus, the number  of isomorphism 
classes of principally polarized Deligne modules with level $N$ structure that are 
$\QQ$-isogenous to $(T,\gamma, \B_0)$ is 
\begin{equation}\label{eqn-WY}
 \left| Z_{\gamma}(\QQ) \backslash Y\right|\end{equation}
where $Y$ denotes the set of objects $(\widehat{L},\alpha)$
in the category $\widehat{\mathcal L}_N(\AA_f^{2n}, \gamma, \B_0)$, cf. equation
(\ref{eqn-main-quotient}).    We recall
from \cite{Kottwitz} the argument that is used to count the number of such lattices modulo the
group of $\QQ$-self isogenies of the basepoint $(T, \gamma, \B_0)$.

Let $G = \GSp_{2n}$ and let $Z_{\gamma}(\QQ)=Z_{G(\QQ)}(\gamma)$ be the centralizer of
$\gamma$ in $\GSp_{2n}(\QQ)$.  
Let $\AA_f^p = \prod'_{v \ne p, \infty} \QQ_p$ denote the ad\`eles away from $p$, let
$\widehat{\ZZ}^p =  \prod_{v \ne p, \infty} \ZZ_v$ so that $\widehat{\ZZ} =
\ZZ_p.\widehat{\ZZ}^p$.  Then 
\[\widehat{K}_N  =  \ker(\GSp_{2n}(\widehat{\ZZ}) \to \GSp_{2n}(\ZZ/N\ZZ))= \widehat{K}^p_NK_p\] 
where
\begin{equation}\label{eqn-Khat}
\widehat{K}^p_N = G(\widehat{\ZZ}^p)  \cap \widehat{K}_N\ \text{ and }\
K_p = G(\ZZ_p).\end{equation}
Let  $f^p$ be the characteristic function of
$\widehat{K}^p_N$,  let $\chi_p$ be the characteristic function of $K_p$ and
let $f_p$ be the characteristic function of
$ K_p  \left(\begin{smallmatrix} I&0\\0 & qI \end{smallmatrix}\right) K_p$.
Choose the Haar measure on $G(\AA_f)$ that gives measure one to the
compact group $\widehat{K}_N$.

\begin{prop}\label{lem-Y}
The number  of isomorphism classes of principally polarized Deligne modules 
with level $N$ structure that are $\QQ$-isogenous to $(T,\gamma,\B_0)$ is
\[
\vol(Z_{\gamma}(\QQ) \backslash Z_{\gamma}(\AA_f))
\cdot\mathcal O^p_{\gamma}\cdot \mathcal O_{\gamma, p}\]
where
\[ \mathcal O_{\gamma}^p = \int_{Z_{\gamma}(\AA^p_f)\backslash G(\AA^p_f)} f^p(g^{-1}\gamma g)dg\]
and
\[ \mathcal O_{\gamma,p} = \int_{Z_{\gamma}(\QQ_p)\backslash G(\QQ_p)} f_p(g^{-1} \gamma g)dg
= \int_{Z_{\gamma}(\QQ_p) \backslash G(\QQ_p)} \chi_p(g^{-1} \alpha_q^{-1} \gamma g)dg.\]
\end{prop}
\begin{proof}  
According to equation (\ref{eqn-WY}) we need to determine the number
$|Z_{\gamma}(\QQ) \backslash Y|$.
First we will show that there is a natural identification  
$Y \cong Y^p \times Y_p$ where
\begin{equation}\label{eqn-YpYp}\begin{split}
Y^p &= \left\{ g \in G(\AA_f^p)/\widehat{K}_N^p|\ g^{-1} \gamma g \in \widehat{K}^p_N\right\}\\
Y_p &= \left\{ g \in G(\QQ_p)/K_p|\ g^{-1} \gamma g\in K_p
\left(\begin{smallmatrix} I & 0 \\ 0 & qI \end{smallmatrix}\right)K_p \right\}\\
&=\left\{ g \in G(\QQ_p)/K_p|\ g^{-1} \alpha_{q}^{-1} \gamma g \in K_p \right\}
\end{split}\end{equation}
(the last equality by Lemma \ref{lem-Smith-form}) with $\alpha_q = \alpha_{\gamma,q}$ as in equation
\ref{eqn-alphaq}.
In \S \ref{subsec-adelic-lattices} the quotient 
\[G(\AA_f)/\widehat{K}_N = (G(\AA_f^p)/\widehat{K}^p_N) \times (G(\QQ_p)/K_p)\]
is identified with the set of pairs $(\widehat{L},\alpha)$ where $\widehat{L}
\subset \AA_f^{2n}$ is a symplectic (up to homothety) lattice and $\alpha$ is a level $N$ structure.
To conclude that $Y = Y^p \times Y_p$ we need to examine the effect of the conditions 
(\ref{eqn-9.3.1}) (stability under $\gamma$ and $q \gamma^{-1}$).
 Write $\widehat{L}_0 = L^p_0 \times L_{0,p}$ for the
components away from $p$ and at $p$ respectively, where $L_{0,p} ={\ZZ}_p^{2n}$.
Given a lattice $\widehat{L}\subset\AA_f^{2n}$ 
write $\widehat{L} = L^p \times {L}_p$ for its components in
$(\AA_f^p)^{2n}\times (\QQ_p)^{2n}.$  If $\widehat{L} = g \widehat{L}_0$ then
${L}^p = g^p {L}_0^p$ and $L_p = g_p (\ZZ_p^{2n})$ where
\[ g = (g^p,g_p) \in G(\AA_f^p) \times G(\QQ_p) = G(\AA_f).\] 
Stability under $\gamma$ implies
\[ \gamma g^p {L}_0^p \subset g^p {L}_0^p\]
while stability under $q \gamma^{-1}$ implies the reverse inclusion.  
Compatibility of $\gamma$ with the level structure (\ref{eqn-9.3.1}) implies that
\[ (g^p)^{-1} \gamma g^p \in  G(\widehat{\ZZ}^p)\cap \widehat{K}_N = \widehat{K}^p_N\]
hence $g^p \in Y^p$. At $p$, stability under $\gamma, q\gamma^{-1}$ becomes
\[\gamma g_p (L_{0,p}) \subset L_{0,p}\ \text{ and }\
q \gamma^{-1} g_p(L_{0,p}) \subset L_{0,p}\]
or
\[ qL_{0,p} \subset g_p^{-1} \gamma g_p L_{0,p} \subset L_{0,p}.\]
For any $k_1,k_2 \in K_p$ this is the same
\[ qL_{0,p} \subset k_1 (g_p^{-1} \gamma g_p) k_2 L_{0,p} \subset L_{0,p}\]
so this condition depends only on the double coset $K_p h_p K_p$ where
$h_p = g_p^{-1} \gamma g_p.$  Using Lemma \ref{lem-Smith-form} we conclude that
\[ h_p = k_1 \left(\begin{smallmatrix} I & 0 \\ 0 & qI \end{smallmatrix} \right) k_2\]
for some $k_1,k_2 \in K_p.$  
Thus, $g_p \in Y_p$.  This proves that $Y \cong Y^p \times Y_p$.  Consequently 
\begin{align*}
\left| Z_{\gamma}(\QQ) \backslash Y \right| &= 
 \int_{Z_{\gamma}(\QQ) \backslash G(\AA_f)}
f^p(g^{-1} \gamma g)f_p(g^{-1} \gamma g)\\
&={\vol}(Z_{\gamma}(\QQ) \backslash Z_{\gamma}(\AA_f)) 
\int_{Z_{\gamma}(\AA_f) \backslash G(\AA_f)} 
f^p(g^{-1} \gamma g)f_p(g^{-1}\gamma g)dg\\
&={\vol}(Z_{\gamma}(\QQ) \backslash Z_{\gamma}(\AA_f) ) \cdot \mathcal O^p_{\gamma}\cdot \mathcal O_{\gamma, p}.
\qedhere
\end{align*} 
\end{proof}


\subsection{Counting real lattices}\label{sec-counting-real}  
In this section we fix a ($\Phi_{\varepsilon}$-positively) polarized Deligne
module $(T, \gamma, \B_0, \tau_0)$ with real structure (where the lattice $T$
is contained in $\QQ^{2n}$ and where $\gamma \in \GSp_{2n}(\QQ)$ is $q$-inversive), and a level $N$ that
is prime to $p$.  We wish to count the number 
of isomorphism classes of principally polarized Deligne
modules with level $N$ structure and with real structure that are $\QQ$-isogenous
to $(T, \gamma, \B_0, \tau_0)$.  By equation (\ref{eqn-main-quotient}) and Proposition
\ref{prop-iso-in-isogeny} this number is
\[ \left| S(\QQ) \backslash X\right| \]
where $X$ denotes the set of objects $(\widehat{L},\alpha)$ in the category
$\widehat{\mathcal L}_N(\AA_f^{2n}, \gamma, \B_0, \tau_0)$
 of \S \ref{subsec-adelic-category}
and where $S(\QQ)$ denotes the group of (involution preserving) $\QQ$-self isogenies of  
$(T, \gamma, \B_0, \tau_0)$.  It may be identified with the centralizer,
\[S_{\gamma}(\QQ) = \left\{ x \in \GL_n^*(\QQ) \left| \gamma x = x \gamma \right.\right\}.\]
(Note that $\gamma \notin \GL_n^*(\QQ)$.)
Following Proposition \ref{prop-cohomology-lattice} 
the $\GL_n^*(\AA_f)$-orbit containing a given object $(\widehat{L},\alpha)$ is determined by its
cohomology class  $[\widehat{L},\alpha] \in {H}^1=H^1(\langle \tau_0 \rangle, \widehat{K}^0_N)$ 
of equation (\ref{eqn-hatH}).  For simplicity we will now also assume that
$N$ is even and that $p \ne 2$: this implies that the contributions from
different cohomology classes are independent of the cohomology class, as explained in
the following paragraph.

Fix such a class $[t] \in {H}^1$, corresponding to some element $t \in \widehat{K}_N^0$
with $t \tilde t = 1$.  Let
\[ X_{[t]}= \left\{(\widehat{L},\alpha)\in X\left|\ [(\widehat{L},\alpha)] = [t] \in H^1\right. \right\} \]
denote the set of objects $(\widehat{L},\alpha)$ whose associated cohomology class is $[t]$.
It consists of a single $\GL_n^*(\AA_f)$-orbit.
We wish to count the number of elements in the set $ S_{\gamma}(\QQ) \backslash X_{[t]}$.
  Since $N$ is even, the cohomology class $[t]$ vanishes in the cohomology of 
$\Sp_{2n}(\widehat{\ZZ})$, by Proposition \ref{prop-trivial-cohomology}.
  This means that $t = g^{-1} \widetilde{g}$ for some $g \in \Sp_{2n}(\widehat{\ZZ})$. 

Let $\widehat{L}_0 = \widehat{Z}^{2n}$
and $\alpha_0:\widehat{L}_0/N\widehat{L}_0 \to (\ZZ/N\ZZ)^{2n}$ denote the standard lattice and the
standard level $N$ structure.  Then $(g\widehat{L}_0, \alpha_0 \circ g^{-1})$ is
a lattice with real structure and level $N$ structure, whose cohomology class equals
$[t] \in H^1$.  Its isotropy group under the action of $\GL_n^*(\AA_f)$ is
\[ 
\widehat{\Gamma}_N = \GL_n^*(\AA_f) \cap g\widehat{K}_N g^{-1} = \GL_n^*(\AA_f) \cap \widehat{K}_N\]
(since $\widehat{K}_N$ is a normal subgroup of $\Sp_{2n}(\widehat{\ZZ})$) 
and is therefore independent of the cohomology class $[t]$.  Hence
\[ X_{[t]} \cong \widehat{\Gamma}_N \backslash \GL_n^*(\AA_f)\]
is a finite-ad\`elic analog of the space $X_{\CC}$ described in \S \ref{subsec-intro1}.
Choose the Haar measure on $\GL_n^*(\AA_f)$ that  gives measure one to this group. 

As in equation (\ref{eqn-Khat}) write $\widehat{K}_N = \widehat{K}^p_NK_p$. Define $\chi^p$ to be the 
characteristic function on $\GSp_{2n}(\AA^p_f)$ of the subgroup $\widehat{K}_N^p$
and define $\chi_p$ to be the characteristic function on $\GSp_{2n}(\QQ_p)$ of 
$K_p = \GSp_{2n}(\ZZ_p)$.  Let $H = \GL_n^*$.

\begin{prop}\label{prop-number-real-lattices} Suppose that $N$ is even and $p\nmid N$.
Then 
\[|S_{\gamma}(\QQ)\backslash X_{[t]}|= \vol(S_{\gamma}(\QQ) \backslash S_{\gamma}(\AA_f))
\cdot I^p_{\gamma} \cdot I_{\gamma, p}\]
where 
\[ I^p_{\gamma} = \int_{S_{\gamma}(\AA^p_f)\backslash H(\AA^p_f)}\chi^p(x^{-1} \gamma x) dx
\]
and
\[ I_{\gamma, p} = \int_{S_{\gamma}(\QQ_p) \backslash H(\QQ_p)} \chi_p(x^{-1} \alpha_q^{-1} \gamma x)dx.\]
\end{prop}\noindent
Here, $\alpha_q = \alpha_{\gamma,q}$ is defined in equation (\ref{eqn-alphaq}). 

\begin{proof}[Proof of Proposition {\rm \ref{prop-number-real-lattices}}]
By Proposition \ref{prop-cohomology-lattice} each $(\widehat{L},\alpha)\in X_{[t]}$  
has the form  $xg.(\widehat{L}_0,\alpha_0)$ for some 
\[x = (x^p,x_p) \in \GL_n^*(\AA^p_f) \times \GL_n^*(\QQ_p)=  \GL^*_{n}(\AA_f)\]
where $t = g^{-1}\tilde{g}$ as above, with $g \in \Sp_{2n}(\widehat{\ZZ})$.
Write $\widehat{L} = L^p \times L_p$ for its component away from $p$ and component at $p$ respectively.
The conditions (\ref{eqn-9.3.1}) give $ \gamma x^pgL^p_0  = x^pgL^p_0$.  Hence
\[g^{-1}(x^p)^{-1} \gamma x^pg \in  \widehat{K}^p_N \]
or  $\chi^p((x^p)^{-1} \gamma x^p) = 1$ (since $\widehat{K}^p_N$ is
normal in $\GSp_{2n}(\widehat{\ZZ}^p)$). 
Similarly, from Lemma \ref{lem-Smith-form} (and since $g_p \in K_p$),
\[x_p^{-1} \alpha_q^{-1}\gamma x_p \in K_p. \]
In this way we have identified $\widehat{X}_{[t]}$ with the product $X_{[t]}^p \times X_p$ where
\begin{align*}
X^p_{[t]} &= \left\{x\in \GL_n^*(\AA_f^p)/\widehat{\Gamma}^p_N\left|\ 
x^{-1} \gamma x \in \widehat{K}^p_N  \right.\right\}\\
X_{p} &= \left\{x \in \GL^*_n(\QQ_p)/\GL_n^*(\ZZ_p)\left|\ x^{-1} \alpha_q^{-1}\gamma x 
\in K_p \right.\right\}.
\end{align*}
 In summary,
\begin{align*}
|S_{\gamma}(\QQ) \backslash X_{[t]}| &= \int_{S_{\gamma}(\QQ)\backslash \rm{GL}_n^*(\AA_f)}
\chi^p(x^{-1} \gamma x) \chi_p(x^{-1}\alpha_q^{-1}\gamma x) dx\\
&= \vol(S_{\gamma}(\QQ) \backslash S_{\gamma}(\AA_f)) 
\cdot I^p_{\gamma}\cdot I_{\gamma,p}.  \qedhere\end{align*}
\end{proof} 
(If $N$ is odd the same analysis gives a
similar formula but the integrand will depend on the cohomology class.)
The results of the preceding sections are summarized in Theorem \ref{thm-statement1} 
equation (\ref{eqn-statement1}) and Theorem \ref{thm-statement2} equation 
(\ref{eqn-statement2a}), which involve ``ordinary" orbital integrals.  In the 
next section we analyze the Dieudonn\'e module of $(T,F)$ in order to obtain 
the twisted orbital integrals that appear in equations (\ref{eqn-statement1b}) and 
(\ref{eqn-statement2b}).


\section{The Tate and Dieudonn\'e modules}\label{sec-Tate-Dieudonne}
\subsection{}
As in Section \ref{sec-ordinary} let $A$ be an ordinary Abelian variety of dimension $n$, with
Deligne module $(T,F)=(T_A,F_A)$, defined over the finite field
 $k = \FF_q$, where $q = p^a.$ Let $\pi=\sigma^a$ be the topological generator for the
Galois group $\Gal = \Gal(\bar k/k)$ where $\sigma(x) = x^p$ for $x \in \bar k$.
 Let $\ell \ne p$ be a (rational) prime.
There is a canonical isomorphism of the Tate module of $A$
\[ T_{\ell}A = \underset{\longleftarrow}{\lim}A[\ell^r]
\cong T_A \otimes \ZZ_{\ell}\]
such that the action of $\pi\in\Gal(\bar k/k)$ is given by the action of
$F_A\otimes 1.$  A polarization and/or a real structure on $(T,F)$ induces
one on $T_A \otimes \ZZ_{\ell}$ in a canonical way.

 Let $W$ denote the Witt ring functor, so that
 $W(k)$, $W(\bar k)$ are the rings of (infinite) Witt
vectors over $k, \bar k$ respectively, with fraction fields $K(k)= W(k)\otimes\QQ_p$ and
$K(\bar k)= W(\bar k)\otimes\QQ_p$ respectively.  Let $W_0(\bar k)$ denote the maximal
unramified extension of $W(k)$.  We may identify $W(\bar k)$ with the completion of
$W_0(\bar k)$.

Let $\mathcal E= \mathcal E(k)$ denote the Cartier-Dieudonn\'e ring, that is, the ring of
noncommutative $W(k)$-polynomials in two variables $\mathcal F, \mathcal V$, subject to the relations
$\mathcal F(wx) = \sigma(w)\mathcal F(x)$, $\mathcal V(wx) = \sigma^{-1}(w) \mathcal V(x)$,
and $\mathcal F \mathcal V = \mathcal V \mathcal F = p$, where $w \in W(k)$ and
$x\in \mathcal E.$

Let $M$ denote the covariant Dieudonn\'e functor (see, for example, \cite{OortChai} \S B.3.5.6 or
 \cite{Goren} p. 245 or \cite{Pink})
which assigns to each $p$-divisible group
\[G = \ldots \ToFrom G_r \ToFrom G_{r+1} \ToFrom\ldots\]
 the
corresponding module $M(G) = \underset{\longleftarrow}{\lim}M(G_r)$
over the Dieudonn\'e ring $\mathcal E.$  For simplicity,
write $M(T):= M(A[p^{\infty}]).$  Recall the canonical decomposition
$T \otimes \ZZ_p = T' \oplus T^{\prime\prime}$ where $F$ is invertible on $T'$ and
is divisible by $q$ on $T^{\prime\prime}$.  The proof of the following proposition will be given
in Appendix \ref{appendix-Dieudonne}.

\begin{prop}\label{prop-Dieudonne-module}
The covariant Dieudonn\'e module $M(T) = M(T')\oplus M(T^{\prime\prime})$ 
is given by the following canonical isomorphisms,
\begin{align}
M(T') &\cong \left(T' \otimes W(\bar k) \right)^{\Gal}\label{eqn-MAprime} \\
M(T^{\prime\prime}) &\cong (T^{\prime\prime}\otimes W(\bar k))^{\Gal}
\label{eqn-MAprimeprime}\end{align}
where the action of $\Gal(\bar k/k)$ on these modules is given by
\begin{equation}\begin{split}\label{eqn-Galois-action}
\pi.(t'\otimes w) &= F(t') \otimes \sigma^a(w) \\
\pi(t^{\prime\prime}\otimes w) &= V^{-1}(t^{\prime\prime})\otimes \sigma^{a}(w)
= q^{-1}F(t^{\pp})\otimes \sigma^a(w).
\end{split}\end{equation}
for any $t'\in T'$, $t''\in T''$. Both $M(T')$ and $M(T^{\prime\prime})$ are free modules
over $W(k)$ of dimension $n$.  The actions of $F$ and $V$ preserve the
submodules $M(T')$ and $M(T^{\pp})$ of $T \otimes W(\bar k)$.  The actions of $\mathcal F$
and $\mathcal V$ on the Dieudonn\'e modules are induced from the following:
\begin{align}\label{eqn-curlyF}
  &\begin{split}
\mathcal F( t'\otimes w) &= pt'\otimes \sigma(w) \\
\mathcal F( t^{\prime\prime}\otimes w) &=  t^{\prime\prime} \otimes \sigma(w)
  \end{split}
    &
  \begin{split}
\mathcal V(t'\otimes w) &=t'\otimes \sigma^{-1}(w)\\
\mathcal V(t^{\prime\prime}\otimes w) &= pt^{\prime\prime}\otimes \sigma^{-1}(w).
  \end{split}
\end{align}

\end{prop}
\subsection{}\label{subsec-Ffactors}
In particular, the operator $\mathcal F$ is $\sigma$-linear; it is invertible on $M(T^{\pp})$ 
and it is divisible by $p$ on $M(T')$.  If $\alpha\in M(T)$ then it is invariant under
$\Gal$ hence  $F(\alpha) = \mathcal V^a (\alpha)$, that is, the mapping $F$ has been factored
as $F = \mathcal V^a$. On the {\em{contravariant}} Dieudonn\'e module $N(T) = N(T')\oplus
N(T^{\pp})$ the mapping $\mathcal F$ is invertible on $N(T')$, divisible by $p$ on
$N(T^{\pp})$ and one has $F = \mathcal F^a$.  Despite this confusion we use the covariant
Dieudonn\'e module because the equations are a bit simpler.

\subsection{}\label{subsec-normalized} 
The mapping $F\otimes \ZZ_p:T \otimes \ZZ_p \to T \otimes\ZZ_p$ preserves the submodules $T'$ and
$T^{\prime\prime}$.  Define $A_q(t' + t^{\prime\prime}) = t' + qt^{\prime\prime}$ and 
$A_p(t' + t^{\prime\prime})=t'+pt^{\prime\prime}$.  These commute with $F$ and $\mathcal F$.
 If $\tau:(T,F) \to (T,F)$ is
a real structure then $\tau$ exchanges $T'$ with $T^{\prime\prime}$ hence
\begin{equation}\label{eqn-tau-A}
\tau A_q \tau^{-1} = q A_q^{-1}. \end{equation}

\begin{cor}\label{cor-module-invariants}  Let $(T,F)$ be a Deligne module.  Then its covariant
Dieudonn\'e module $M(T)$ is 
\begin{equation}\label{eqn-M(T)}
 M(T) = \left\{ x \in T \otimes W(\bar k)\left|\ F(x) = A_q\sigma^{-a}(x) \right.\right\}.\end{equation}
The action of $\mathcal F$ on $M(T)$ is given by 
$\mathcal F(t\otimes w) = pA_p^{-1}(t) \otimes \sigma(w)$. 
\end{cor}

\section{Polarization of the Dieudonn\'e module}\label{sec-polarization-dieudonne}
\subsection{} Let $(T,F)$ be a Deligne module. In this section we describe the effect 
on the Dieudonn\'e module of a polarization of  $(T,F)$.
Let\ $\B:T \times T \to \ZZ$ be an alternating form such that
\[ \B(Fx,y) = \B(x,Vy)\]
for all $x,y\in T$.
Extending scalars to $W(\bar k)$ gives an alternating form, which we still denote by $\B$,
\[ \B: (T \otimes_{\ZZ}W(\bar k)) \times (T \otimes_{\ZZ} W(\bar k)) \to W(\bar k).\]

\begin{prop}\label{prop-Dieudonne-polarization}  
The restriction of $\B$ to the Dieudonn\'e module
\[M(T) \subset T \otimes W(\bar k)\] 
takes values in $W(k)$ and the resulting pairing
\[ \B_p:M(T) \times M(T) \to W(k)\]
satisfies  $\B_p(\mathcal Fx,y) = \sigma \B_p (x,\mathcal Vy)$.
If the original alternating form $\B$ is nondegenerate over $\QQ$ or $\QQ_p$ then the resulting
form $\B_p$ is nondegenerate over $K(k)$.  If the original form $\B$ is nondegenerate over
$\ZZ$ or $\ZZ_p$ then the form $\B_p$ is nondegenerate over $W(k)$.
\end{prop}

\begin{proof} 
Extending scalars to $T \otimes \ZZ_p = T' \oplus T^{\prime\prime}$, we see that if $x,y\in T^{\prime}$ then $\B(x,y)$ is divisible by $q^m$ for all 
$m$ so it is zero, and similarly $\B$ vanishes on $T^{\prime\prime}$. Hence $T'$ and $T^{\pp}$ are
Lagrangian subspaces of $T\otimes\ZZ_p$.
Let $e'_1,e'_2,\cdots,e'_n$ be a basis for $T'$ and let 
$e_1^{\prime\prime}, \cdots, e_n^{\prime\prime}$ be a basis for $T^{\prime\prime}$ and let
$\B_{ij} = \B(e'_i,e^{\prime\prime}_j)\in \ZZ_p$.  
Let $x'=\sum_{i=1}^n e'_i\otimes u_i \in M(T') \subset T'\otimes W(\bar k)$ and
$x^{\pp}=\sum_{j=1}^n e^{\pp}_j\otimes v_j \in M(T^{\pp})$ be Galois invariant,
where $u_i,v_j \in W(\bar k)$.  Then
\begin{align}
\B_p(x', x^{\pp}) &= \sum_{i=1}^n \sum_{j=1}^n\B_{ij}u_iv_i \label{eqn-BiCi} \\
&= \sum_{i=1}^n \sum_{j=1}^n \B\left(F(e'_i)\otimes \sigma^{a}(u_i), 
V^{-1}(e^{\pp}_j) \otimes \sigma^{a}(v_j)\right)\notag \\
&= \sigma^a\left(\sum_{i,j} \B(e'_i, e^{\prime\prime}_j)u_iv_j\right) =\sigma^a\B_p(x',x^{\prime\prime})
\label{eqn-sigmaBiCi} \end{align}
using equation (\ref{eqn-Galois-action}).  Hence the bilinear form $\B_p$ takes values in $W(k)$.
Moreover,
\[
\B_p(\mathcal F x', x^{\pp}) = \sum_{i=1}^n \sum_{j=1}^n \B(pe'_i \otimes \sigma(u_i), e^{\pp}_j \otimes v_j)
= p\sum_{i,j} \sigma(u_i)v_j\B_{ij} \]
but
\[\B_p(x', \mathcal V(x^{\pp})) = \sum_{i=1}^n \sum_{j=1}^n \B(e'_i \otimes u_i, 
pe^{\pp}_j \otimes \sigma^{-1}(v_j))
= p\sum_{i,j} u_i \sigma^{-1}(v_j)\B_{ij}\]
which implies that $\B_p(\mathcal F x', x^{\pp}) = \sigma \B_p(x', \mathcal V(x^{\pp})).$
Using $\mathcal F \mathcal V = p$ and the antisymmetry of $\B_p$, this formula also implies
 that $\B_p(\mathcal F y^{\pp}, y') = \sigma \B_p(y^{\prime\prime},
\mathcal V(y^{\prime}))$ for all $y' \in M(T')$ and $y^{\pp} \in M(T^{\pp})$.

Let us now verify the last statement in Proposition \ref{prop-Dieudonne-polarization}
and assume that the symplectic form $\B$ is strongly nondegenerate (over $\ZZ_p$). 
Then it induces an isomorphism 
\[  T' \oplus T^{\prime\prime} \to
\hHom(T',\ZZ_p) \oplus \hHom(T^{\prime\prime},\ZZ_p)\]
which therefore gives isomorphisms
\begin{equation}\label{eqn-duality}
 T'\otimes\ZZ_p \cong \hHom(T^{\prime\prime},\ZZ_p)\ \text{ and }
T^{\prime\prime}\otimes\ZZ_p \cong \hHom(T',\ZZ_p).\end{equation}
Given a $\ZZ_p$-basis $\{e'_1,\cdots,e'_n\}$ for $T'$ there is a  
unique dual basis 
$e^{\pp}_1,\cdots,e^{\pp}_n$ of $T^{\pp}$ so that $\B(e'_i, e^{\pp}_j) = \delta_{ij}$. 

Using Proposition \ref{prop-Dieudonne-module} the module $M(T')$ is free over
$W(k)$ of dimension $n$ so it has a basis, say $\{ b'_1,\cdots,b'_n\}$ with
$b'_i = \sum_{s=1}^n e'_s \otimes U_{si}$ and $M(T^{\prime\prime})$ has a
$W(k)$ basis consisting of vectors
$b^{\prime\prime}_j = \sum_{t=1}^n e^{\prime\prime}_j \otimes V_{jt}$ with 
$U=(U_{si})$ and $V=(V_{tj})$ in $\GL_n(W(\bar k))$.  The calculation
\[ \B_p(b'_i,b^{\prime\prime}_j) 
=  \sum_s \sum_t \delta_{st} U_{si}V_{tj} = \sum_s U_{si}V_{sj} =
(\tr{U}V)_{ij}\]
implies that the matrix of $\B_p$ with respect to this basis is
$\left(\begin{smallmatrix} 0 & \tr{U}V \\ -\tr{V}U & 0 \end{smallmatrix}\right)$.
This matrix is in $\GL(2n, W(\bar k))$ so its determinant is a $p$-adic unit.  Moreover, 
equation (\ref{eqn-sigmaBiCi}) says
that the entries in this matrix are in $W(k)$.  It follows that the matrix is in $\GL(2n, W(k))$, 
meaning that $\omega_p$ is strongly nondegenerate.  

If the original symplectic form $\B$ is nondegenerate over $\QQ_p$ then the same argument
may be used to show that the resulting pairing on $M(T)\otimes K(k)$ is nondegenerate
over $K(k)$.\end{proof}

\section{Real structures on the Dieudonn\'e module}\label{sec-real-Dieudonne}
\subsection{}
Let $(T,F)$ be a Deligne module with a polarization $\B:T \times T \to \ZZ$.  Let
$\B_p$ denote the resulting symplectic form on the covariant Dieudonn\'e module $M(T)$. 
Let $\tau:T \to T$ be a real structure on $(T,F)$ that is compatible with the polarization.
Unfortunately, the mapping $\tau$ does not induce an involution on the Dieudonn\'e module
$M(T)$ without making a further choice.
Following Appendix \ref{sec-Witt-involution}, we choose and fix, once and for all, a continuous
$K(k)$-linear  involution $\taub:K(\bar k) \to K(\bar k)$ that preserves $W(\bar k)$, so that 
$\taub \sigma^a(w) = \sigma^{-a} \taub(w)$.
\begin{prop}\label{prop-conjugation-Dieudonne}
The following mapping $\tau_p: T \otimes W(\bar k) \to T \otimes W(\bar k)$ defined by
$\tau_p(x \otimes w) = \tau(x) \otimes \taub(w)$ is continuous and $W(k)$-linear.  It 
preserves the Dieudonn\'e module $M(T)$ and it satisfies $\tau_p \mathcal F^a = 
\mathcal V^a \tau_p$ and 
\begin{equation}\label{eqn-Bptaup}
\B_p(\tau_p x, \tau_py)= -\B_p(x,y) \text{ for all } x, y \in M(T).\end{equation} 
\end{prop}
\begin{proof}
As indicated in Lemma \ref{lem-tau-switches}, the mapping $\tau$ takes $T'$ to 
$T^{\prime\prime}$ and vice versa.  If $x'\otimes w \in T'\otimes W(\bar k)$ then
\begin{align*}
 \tau_p \pi.(x' \otimes w) &= \tau_p(F(x') \otimes \sigma^a(w))\\
&= V \tau(x') \otimes \sigma^{-a}\taub(w)\\
&= \pi^{-1}(\tau(x')\otimes\taub(w))\\
&= \pi^{-1} \tau_p(x' \otimes w)\end{align*}
which shows that $\tau_p$ takes $M(T')$ to $M(T^{\prime\prime})$ (and vice versa).
Similarly,
\begin{align*}
\tau_p \mathcal F^a(x' \otimes w) &= \tau_p(x' \otimes q\sigma^a(w))\\
&= \tau(x')\otimes q\sigma^{-a}\taub(w)\\
&= \mathcal V^a(\tau(x') \otimes \taub(w))\\
&= \mathcal V^a \tau_p(x' \otimes w). 
\end{align*}  
Similar calculations apply to any element $x^{\pp}\otimes w \in T^{\pp}\otimes W(\bar k)$.

We now wish to verify equation (\ref{eqn-Bptaup}).  
Let $V = T\otimes\QQ$.  Use Proposition \ref{prop-real-existence} to decompose 
$V = V_1 \oplus \cdots \oplus V_r$ into an orthogonal direct sum of simple $\QQ[F]$ modules that
are preserved by $\tau$.  This induces a similar $\B_p$-orthogonal decomposition of
\[ M_{\QQ}(T) = M(V) = M(T)\otimes_{W(k)}K(k)\]
into submodules $M_i = M(V_i)$ over the rational Dieudonn\'e ring 
\[\mathcal A_{\QQ} = \mathcal A \otimes K(k) = K(k)[F,V]/(\text{ relations}),\]
each of which is preserved by $\tau_p$.  Since this is an orthogonal direct sum, 
it suffices to consider a single factor, that is, we may assume that $(V,F)$ 
is a simple $\QQ[F]$-module.

As in [Deligne] set $V \otimes_{\QQ}\QQ_p = V' \oplus V^{\pp}$ where the eigenvalues of $F|V'$
are $p$-adic units and the eigenvalues of $F|V^{\pp}$ are divisible by $p$.  Then the same holds
for the eigenvalues of $\mathcal F^a$ on each of the factors of
\[ M(V) = M(V') \oplus M(V^{\pp}).\]
Moreover, these factors are cyclic $\mathcal F^a$-modules and $\tau_p$ switches the two factors.
It is possible to find a nonzero vector $v \in M(V')$ so that $v$ is $\mathcal F^a$-cyclic in
$M(V')$ and so that $\tau_p(v)$ is $\mathcal F^a$-cyclic in $M(V^{\pp})$.  It follows that
$w=v \oplus \tau_p(v)$ is a cyclic vector for $M(V)$ which is fixed under $\tau_p$, that is,
$\tau_p(w)=w$.  We obtain a basis of $M(V)$:
\[ w, \mathcal F^a w, \cdots, \mathcal F^{a(2n-1)}w.\]
The symplectic form $\B_p$ is determined by its values $\B_p(w, \mathcal F^{aj}w)$
for $1 \le j \le 2n-1$.  But
\begin{align*}
\B_p(\tau_pw, \tau_p\mathcal F^{aj}w) &=\B_p(w, \tau_p \mathcal F^{aj}\tau_p w)\\
&=q^{j}\B_p(w, \mathcal F^{-aj} w)\\
&=q^jq^{-j}\B_p(\mathcal F^{aj}w, w) = -\B_p(w, \mathcal F^{aj}w).
\qedhere\end{align*}
\end{proof}

\subsection{Remark}
For non-negative integers $r,s$ the {\em Manin module}
\[ M_{r,s} = \mathcal E(k)/(\mathcal F^r - \mathcal V^s)\]
does not admit a real structure unless $r=s$.  However the sum
$M_{r,s} \oplus M_{s,r}$ admits a real structure $\tau_p$ which
swaps the two factors and exchanges $\mathcal F$ (in one factor) with
$\mathcal V$ (in the other factor).  One verifies that $\tau_p \mathcal F^a
\tau_p^{-1}(x,y) = \mathcal V^a(x,y)$ for any $(x,y) \in M_{r,s}\oplus
M_{s,r}$ (where $a = [k:\FF_p]$) but it is generally
not true that $\tau_p \mathcal F \tau_p^{-1}(x,y) = \mathcal V(x,y)$ because
$\mathcal F$ is not $W(k)$-linear.

\section{Basis of the Dieudonn\'e module}

\subsection{}\label{subsec-tilde}
In \S \ref{sec-real-Dieudonne} we fixed a continuous $K(k)$-linear involution 
$\taub: K(\bar k) \to K(\bar k)$ that preserves $W(\bar k)$ and is the identity on $K(k)$.  
Let $\bar\tau_0 = \tau_0\otimes\taub:K(\bar k)^{2n} \to K(\bar k)^{2n}$.  For any
$ B = \left(\begin{smallmatrix} X & Y \\ Z & W \end{smallmatrix}\right) \in \GSp(2n, K(\bar k))$ let
\[ \widetilde{B} = \bar\tau_0(B)\tau_0=\left( \begin{matrix} \bar\tau(X) & -\bar\tau(Y) \\ 
-\bar\tau(Z) & \bar\tau(W)\end{matrix}\right).\]
Then  
$\widetilde{ABC} = \widetilde{A} \widetilde{B}\widetilde{C} \
\text{ whenever }\ A, C \in \GSp(2n, K(k)).$

\subsection{}\label{subsec-coordinates}
Let $(T,F,\B)$ be a polarized Deligne module with its corresponding  Dieudonn\'e module
\[ M(T)=M(T,F) = (T \otimes W(\bar k))^{\Gal}.\]
Let $M_{\QQ}(T)= M(T)\otimes K(k) = (T \otimes K(\bar k))^{\Gal}$.
As in \S \ref{subsec-normalized}
we have a decomposition \cite{Deligne} $T \otimes \ZZ_p = T' \oplus T^{\pp}$, preserved
by $F$, such that $F':=F|T'$ is invertible and $F^{\prime\prime}:=F|T^{\pp}$ is divisible by $q$.
Let $A_q(x' + x^{\pp}) = x' + q x^{\pp}$ (and similarly for $A_p$) where $x' \in T'\otimes\QQ_p$
and $x^{\pp} \in T^{\pp}\otimes\QQ_p$.
In the following proposition we will construct a diagram
\begin{equation}\label{diag-JZ}
\begin{diagram}[size=1.5em]
T\otimes_{\ZZ} K(\bar k) &\rTo^{\Phi\otimes K(\bar k)}& K(\bar k)^{2n} 
&\lTo^{\quad\quad B\quad\quad }& K(\bar k)^{2n}\\
 \uTo && \uTo && \uTo \\
  M_{\QQ}(T) &\rTo^{\cong}& \mathcal J_{\QQ}(\gamma) & \lTo^{\cong} & K(k)^{2n}
\end{diagram}\end{equation}

\begin{prop}\label{prop-dieudonne-basis}
There exist isomorphisms
$\Phi:T\otimes_{\ZZ_p}\QQ_p \to \QQ_p^{2n}$ and $B:K(\bar k)^{2n} \to
K(\bar k)^{2n}$  with the following properties.

\begin{enumerate}[leftmargin=0cm,labelsep=.5em,labelwidth=.5cm,itemsep=.5em,itemindent=\labelwidth+\labelsep]
\item[$(1)$]  The mapping $\Phi \otimes K(\bar k)$ takes the symplectic form $\B$ on
$T \otimes K(\bar k)$ to the standard symplectic form $\B_0$ on $K(\bar k)^{2n}$.  The
mapping $B$ preserves the standard symplectic form on $K(\bar k)^{2n}$ so we may also
consider $B$ to be a matrix in $\Sp_{2n}(K(\bar k))$.
\item[$(2)$]
Let $\gamma = \Phi F \Phi^{-1}$, let $\alpha_q = \Phi A_q \Phi^{-1}$ and
$\alpha_p = \Phi A_p \Phi^{-1}$.  Then $\gamma \in  \GSp_{2n}(\QQ_p)$ and the mapping 
\[\Phi\otimes K(\bar k):T\otimes K(\bar k) \to K(\bar k)^{2n}\] restricts to a symplectic
isomorphism between $M_{\QQ}(T)=M_{\QQ}(T')\oplus M_{\QQ}(T^{\prime\prime})$ and 
\begin{equation} \label{eqn-Jgamma}
\mathcal J_{\QQ}(\gamma) = \left\{ b \in K(\bar k)^{2n} \left| \
\gamma b = \alpha_q\sigma^{-a}b \right. \right\}.\end{equation}
\item[$(3)$]
The mapping $B$ restricts to an isomorphism $K(k)^{2n} \cong \mathcal J_{\QQ}(\gamma)$,
hence $\gamma B = \alpha_q\sigma^{-a}(B)$.
\item[$(4)$]
The action of $\mathcal F$ on $M(T)$ equals the action of $p\alpha_p^{-1}\sigma$ on
$\mathcal J_{\QQ}(\gamma)$ so it becomes an action 
\[ x \mapsto B^{-1}p\alpha_p^{-1} \sigma(Bx) = \delta \sigma(x)\]
for any $x \in K(k)^{2n}$, where 
\begin{equation}\label{eqn-delta}
\delta = B^{-1}p\alpha_p^{-1}\sigma(B)\end{equation}
is in $\GSp_{2n}(K(k))$ and its norm,
\[ {\sf{N}}(\delta) = \delta \sigma(\delta)\cdots \sigma^{a-1}(\delta) = B^{-1} q \alpha_q^{-1}\sigma^a(B)
 = B^{-1} q\gamma^{-1}B\]
is $G(K(\bar k))$-conjugate to $q\gamma^{-1}$.
\item[$(4')$]  The action of $\mathcal V$ on $M(T)$ becomes the action of $\eta\sigma^{-1}$
where $\eta = B^{-1} \alpha_p \sigma^{-1}(B)$.

\item[$(5)$] If, in addition to the assumptions above, $\tau: T\to T$ is a real
structure on $(T,F,\B)$ then the mapping
$\Phi$ may be chosen so that $\Phi_*(\B) = \B_0$ (hence $\gamma$ is $q$-inversive),
and so that $\Phi \tau \Phi^{-1} = \tau_0$ is the standard involution
(hence $\Phi$ takes the involution $\tau_p = \tau \otimes \bar\tau$ to the involution
$\bar\tau_0 = \tau_0 \otimes \bar\tau$), and the matrix $B\in \Sp_{2n}(K(\bar k))$ may
be chosen so that $\widetilde{B} = B$. 
\end{enumerate}
\end{prop}
These properties may be summarized in the following table, where $u_p = B^{-1} \alpha_p B$.\\

\newcommand\Tstrut{\rule{0pt}{2.5ex}}  
\centerline{
\begin{tabular}{|c|c|c|c|}
\hline
{$T\otimes\ZZ_p$}\Tstrut & \multicolumn{3}{|c|}{
$T \otimes W(\bar k) \rightarrow W(\bar k)^{2n} \leftarrow W(\bar k)^{2n}$} \\ \hline 
 \Tstrut & $M_{\QQ}(T)$& \ $ \mathcal J_{\QQ}(\gamma)$\  &\ $K(k)^{2n}$ \\ \hline\hline
\Tstrut$F$ &$F$ & $\gamma$ & $B^{-1} \gamma B$ \\ \hline
\Tstrut$A_p$ & $A_p$ & $\alpha_p$ & $u_p$ \\ \hline
&\Tstrut $\mathcal F$   &   $p\alpha_p^{-1}\sigma$    & $\delta\sigma$    \\ \hline
&\Tstrut $\mathcal V$   &   $\alpha_p\sigma^{-1} $    & $p \sigma^{-1}\delta^{-1} $ \\ \hline
$\B$ & $\B_p$ & $\B_0$ & $\B_0$ \\  \hline
$\tau$ & $\tau_p=\tau\otimes\bar\tau$ & $\bar\tau_0=\tau_0\otimes\bar\tau$& $\bar\tau_0$ \\ \hline
\end{tabular}
} 
\smallskip

\begin{proof}  
Use Lemma \ref{lem-Darboux} (taking $R = \QQ_p$) to obtain an isomorphism
$\Phi:T\otimes\QQ_p \to \QQ_p^{2n}$ which takes $\B$ to $\B_0$.  By Proposition
\ref{prop-Dieudonne-module} (which describes the Galois action), such a
mapping $\Phi$ automatically restricts to an isomorphism $\Phi:M_{\QQ}(T) \to \mathcal J_{\QQ}(\gamma)$. 
The Dieudonn\'e module $M(T)$
is a free module of rank $2n$ over the ring $W(k)$  (see Proposition \ref{prop-Dieudonne-module}).  Consequently the vector space $\mathcal J_{\QQ}(\gamma)$ has dimension $2n$ over the field $K(k)$.  

By Proposition \ref{prop-Dieudonne-polarization} the symplectic form $\B$ on $T\otimes K(\bar k)$ 
restricts to a nondegenerate symplectic form $\B_p$ on $M_{\QQ}(T)$ and hence the 
symplectic form $\B_0$ on $K(\bar k)^{2n}$ restricts to a nondegenerate symplectic
form on $\mathcal J_{\QQ}(\gamma)$.  Now use Lemma \ref{lem-Darboux} again (with $R = K(k)$) to
obtain a symplectic basis $B:K(k)^{2n} \to \mathcal J_{\QQ}(\gamma)$.  Such a
mapping $B$ will automatically satisfy $\gamma B = \alpha_q\sigma^{-a}(B)$.  This proves
parts (1), (2) and (3).

For part (4) one checks that $\sigma^{-a}(\delta) = \delta$ hence $\delta \in \GSp_{2n}(K(k))$.

Now consider part (5) in which we are given a real structure $\tau:T \to T$.   By  Proposition \ref{prop-classification} the mapping $\Phi$  may be chosen so that
$\Phi_*(\B) = \B_0$ and so that $\Phi \tau \Phi = \tau_0$.  Proposition
 \ref{prop-conjugation-Dieudonne} guarantees that $\tau_p = \tau\otimes\bar\tau$ preserves $M_{\QQ}(T)$,
hence $\bar\tau_0 = \tau_0\otimes \taub$ preserves $\mathcal J_{\QQ}(\gamma)$.
Lemma \ref{lem-Darboux} gives a matrix $B_1 \in \Sp(2n, K(\bar k))$ whose columns form a symplectic
$K(k)$-basis of $\mathcal J_{\QQ}(\gamma)$, hence  $\gamma B_1 = \alpha_q \sigma^{-a}(B_1)$.  
Set $t = B_1^{-1} \widetilde{B}_1 \in \Sp(2n, K(\bar k))$.   
Using \S  \ref{subsec-tilde} and equation (\ref{eqn-tau-A}),
\begin{align*}
\sigma^{-a}(t) &=  \sigma^{-a}(B_1^{-1}) \widetilde{\sigma^aB}_1\\
&= B_1^{-1} \gamma^{-1} \alpha_q \tilde{\gamma}^{-1}\tilde{\alpha}_q \widetilde{B}_1\\
&= B_1^{-1} \widetilde{B}_1 =t.
\end{align*}
Therefore $t \in \Sp(2n, K(k))$ and $ t \tilde{t} = I$. 
So $t$ defines a 1-cocycle in $H^1\left(\langle\tau_p\rangle,
\Sp(2n, K(k))\right)$, which is trivial, by Proposition \ref{prop-trivial-cohomology}
in Appendix \ref{appendix-cohomology}.  Thus, there
exists $\beta \in \Sp(2n, K(k))$ so that $t = \beta^{-1}
\widetilde{\beta}$.  Set $B = B_1 \beta^{-1}\in \Sp(2n, K(\bar k))$.  
Then the columns of the matrix $B$ form a symplectic $K(k)$-basis of
$\mathcal J_{\QQ}(\gamma)$ and $\widetilde{B} = B$.  This completes
the proof of Proposition \ref{prop-dieudonne-basis}.
\quash{
To verify (\ref{eqn-tilde-delta}), use $\alpha_pB\delta = p \sigma(B)$  and $\widetilde{B} = B$ 
to find:
\[\tilde{\alpha}_p {B} \tilde{\delta} = p \sigma^{-1}(B).\]
Choose a basis of $T\otimes \ZZ_p$ that is compatible with the decomposition 
$T\otimes \ZZ_p = T' \oplus T^{\pp}$.  With respect to this basis, the matrix of $A_p$ is
$\left(\begin{smallmatrix} I & 0 \\ 0 & pI \end{smallmatrix}\right)$ and the matrix of $\tau$
is $\left(\begin{smallmatrix} 0 & T \\ T^{-1} & 0 \end{smallmatrix} \right)$ for some matrix
$T \in \GL_n(\ZZ_p)$.  It follows that $\tau A_p \tau^{-1} = p A_p^{-1}$ hence
$\tilde{\alpha}_p = p \alpha_p^{-1}$.  It follows that 
\[\tilde{\delta} = B^{-1} \alpha_p \sigma^{-1}(B) = p \sigma^{-1}(\delta^{-1}).\]
Finally, $\tau_p\mathcal F(x) = \tau_p \delta \sigma(x) = p \sigma^{-1}(\delta^{-1})
 \sigma^{-1}(\tau_p x) = \mathcal V \tau_p(x)$ for any $x \in K(k)^{2n}$. 
{\bf\color{red} We need to check this last paragraph.}
}
\end{proof}
We remark that Lemma \ref{lem-Darboux} actually gives the following slightly stronger statement:
if $\B:T\otimes T \to \ZZ$ is a strongly nondegenerate
symplectic form such that $\B(Fx,y) = \B(x,Vy)$ then $\Phi$ may be chosen so that
$\Phi_*(\B) = \B_0$ and so that $\gamma = \Phi F \Phi^{-1} \in \GSp_{2n}(\ZZ_p)$, and the
matrix $B$ may be chosen to lie in $\Sp_{2n}(W(\bar k))$.  

\begin{lem}\label{lem-conj-B}
Let $(T,F,\B)$ be a polarized Deligne module with Dieudonn\'e module $M(T)$, and let 
$\Phi, B, \gamma, \delta$ be as described in Proposition \ref{prop-dieudonne-basis} 
above.  Then conjugation by $B$ determines an injective homomorphism
\begin{equation}\label{eqn-conj-B}
\mathfrak c_B:Z_{\gamma}(\QQ_p) \to Z'_{\delta}(K(k))\end{equation}
from the centralizer of $\gamma$ in $\GSp_{2n}(\QQ_p)$ to the twisted centralizer
\begin{equation}\label{eqn-twisted-delta}
 Z'_{\delta}(K(k)) =\left\{ u \in \GSp_{2n}(K(k))|\
u^{-1} \delta \sigma(u) = \delta \right\}.\end{equation}
\end{lem}
\begin{proof}
Let $h \in Z_{\gamma}(\QQ_p)$. Observe that $h \alpha_q = \alpha_q h$.  This follows from the fact that 
$\gamma$ may be expressed as $\gamma' \oplus \gamma^{\prime\prime}$ with respect to the decomposition
$\ZZ_p^{2n} = \Phi(T') \oplus \Phi(T^{\prime\prime})$.  Since $\gamma'$ and $\gamma^{\prime\prime}$ have
different eigenvalues, the mapping $h$ must preserve this decomposition, $h = h' \oplus h^{\prime\prime}$.
But $\alpha_q$ acts as a scalar on each of these factors so it commutes with $h$. 
Using this, we also conclude that $\mathfrak c_B(h):=B^{-1}hB \in G(W(k))$ because
\[ \sigma^{a}(B^{-1}hB) = B^{-1} \alpha_q^{-1} \gamma h \gamma^{-1} \alpha_q B = B^{-1} h B.\]
One then checks directly that $u=B^{-1}hB$ satisfies $u^{-1} \delta \sigma(u) = \delta$. \end{proof}


\section{Lattices in the Dieudonn\'e module}\label{sec-dieudonne-lattices}

\subsection{}
 Let $(T,F,\B)$ be a polarized Deligne module with Dieudonn\'e module $M(T)\subset T \otimes W(\bar k)$.
Using the decomposition $T \otimes \ZZ_p = T' \oplus T^{\pp}$ as in \S \ref{subsec-normalized} define
$A_p(x' + x^{\pp}) = x' + p x^{\pp}$.  We have the following analog of Lemma \ref{lem-Smith-form}.

\begin{lem}
Let $M \subset M_{\QQ}(T)$ be a $W(k)$-lattice. 
The following statements are equivalent.
\begin{enumerate}
\item The lattice $M$ is preserved by $\mathcal F$ and by $\mathcal V$.
\item $pM \subset \mathcal F M \subset M$.
\item $pM \subset \mathcal V M \subset M$.
\item $A_p^{-1}\mathcal V M = M$.
\end{enumerate}\end{lem}
\begin{proof}
The equivalence of (1), (2) and (3) is straightforward.  
Using the same argument as in \cite{Deligne}, the lattice $M$ decomposes
as $M' \oplus M^{\pp}$ such that $\mathcal V$ is invertible on
$M'$ and is divisible by $p$ on $M^{\pp}$ from which it 
follows\footnote{recall that $F = \mathcal V^a$} that
$M ' = M \cap T'$ and $M^{\pp} = M \cap T^{\pp}$,
that $\mathcal F$ and $\mathcal V$ commute with $A_p$, and that
\[\mathcal V M  \subset A_pM\
\text{ and }\
\mathcal F M \subset pA_p^{-1}M.\]
Together these imply (4).  Conversely, suppose that $A_p^{-1}\mathcal V M = M$.
Then $\mathcal V M = A_pM \subset M$ and $\mathcal F M
= p \mathcal V^{-1}M = p A_p^{-1} M \subset M$.
\end{proof}

Let  $G=\GSp_{2n}, \Phi, B, \gamma, \delta,u_p$ be as described in 
Proposition \ref{prop-dieudonne-basis} and in the table that follows it.
Using the bases $\Phi$ and $B$, it follows from the above lemma (and Smith normal form,
as in Lemma \ref{lem-Smith-form}) that the set $\mathcal L_p(M_{\QQ}(T), \mathcal F, \B_p)$
of lattices $M \subset M_{\QQ}(T)$ that are preserved by $\mathcal F, \mathcal V$, 
and satisfy $M^{\perp} = cM$ (for some $c \in K(k))$ may in turn be identified with the set of elements 
\begin{align}
\mathcal Y_p &= \left\{ \beta \in G(K(k))/G(W(k))|\ 
\beta^{-1} \delta \sigma(\beta) \in G(W(k)) \left(\begin{smallmatrix} I & 0 \\ 0 & pI \end{smallmatrix}
\right) G(W(k)) \right\}\label{eqn-curly-Y1}\\
&= \left\{ \beta \in G(K(k))/G(W(k))|\ \beta^{-1} u_p \delta \sigma(\beta)  \in G(W(k))\right\}.
\label{eqn-curly-Y2}\end{align}
by the mapping $\mathcal S:\mathcal Y_p \to \mathcal L_p(M_{\QQ}(T), \mathcal F, \B_p)$ given by
$S(\beta) = \Phi^{-1}B\beta\Lambda_0$.
Similarly, as in Proposition \ref{lem-Y} the set $\mathcal L_p(T\otimes \QQ_p, F, \B)$ of lattices $L \subset
T \otimes \QQ_p$ that are preserved by $F,V$ and satisfy $L^{\perp} = cL$ (for some $c \in \QQ_p$), may be
identified with the set of elements $Y_p$ of equation (\ref{eqn-YpYp}) by the mapping
$S:Y_p \to \mathcal L_p(T\otimes\QQ_p, F, \B)$ given by $S(g) = \Phi^{-1}(gL_0)$.

  To each $\ZZ_p$-lattice $L \subset T\otimes\QQ_p$ that is preserved by $F$
and $V$ we obtain a $W(k)$-lattice
\[ M(L) = (L\otimes W(\bar k))^{\Gal} \subset M_{\QQ}(T)\]
where the Galois action is given, as in equation (\ref{eqn-Galois-action}), by 
$\pi.(t\otimes w) = FA_q^{-1}(t) \otimes \sigma^a(w)$ for $t \in L$ and $w \in W(\bar k)$
and where $\mathcal F$ is given by equation (\ref{eqn-curlyF}), that is,
$\mathcal F(t\otimes w) = pA_p^{-1}(t)\otimes \sigma(w)$.  
\begin{prop}\label{cor-group-description}{\rm(\cite{Kottwitz})}
The association $L \mapsto \Lambda= M(L)$ induces  one to one correspondences (the vertical
arrows in the following diagram),
\begin{diagram}[size=2em]
Y_p & \rTo^{\scriptstyle\cong}_{\scriptstyle S} &\mathcal L_p(T\otimes\QQ_p,F,\B)\\
\dTo && \dTo \\
\mathcal Y_p & \rTo^{\scriptstyle\cong}_{\scriptstyle\mathcal S} & \mathcal L_p(M_{\QQ}(T),\mathcal F, \B_p)
\end{diagram}
\quash{
\begin{diagram}[size=2em]
\mathcal L_p(T\otimes\QQ_p, F, \B) &\rTo& \mathcal L_p(\QQ_p^{2n},\gamma,\B_0) 
&\lTo^{\scriptstyle S}& Y_p\\
\dTo^{M} && \dTo && \dTo\\
\mathcal L_p(M_{\QQ}(T), \mathcal F, \B_p) &\rTo& \mathcal L_p(\mathcal J_{\QQ}(\gamma),\mathcal F, \B_0)
& \lTo^{\scriptstyle\mathcal S} & \mathcal Y_p
\end{diagram}
}
that are equivariant with respect to the action of $\mathfrak c_B: Z_{\gamma}(\QQ_p) \to
Z'_{\delta}(K(k))$ on $Y_p$ and $\mathcal Y_p$ respectively.

If $\tau:T \to T$ is a real structure on $(T,F,\B)$ with corresponding real structure
$\tau_p$ on $M_{\QQ}(T)$ then the lattices $L\subset T\otimes\QQ_p$  that are
preserved by $\tau$ correspond exactly to the lattices $\Lambda \subset M_{\QQ}(T)$ 
that are preserved by $\tau_p$, and to the lattices $\Lambda \subset \mathcal J_{\QQ}(\gamma)$
 that are preserved by $\bar\tau_0=\tau_0\otimes\bar\tau$. 
\end{prop}

\begin{proof} 
\quash{
 The mapping $S$ 
is a one to one correspondences by Lemma \ref{lem-Smith-form}.
Similarly, the mapping $\mathcal S$ is a one to one correspondence.
Given $\Lambda \subset \mathcal J_{\QQ}(\gamma)\subset W(\bar k)^{2n}$ so that $\Lambda^{\vee}
= c\Lambda$, there exists
$\beta\in \GSp_{2n}(K(k))/GSp_{2n}(W(k))$, so that $\Lambda = B\beta \Lambda_0$ (where $\Lambda_0 = K(k)^{2n}$).
Using the same method as that of Lemma \ref{lem-Smith-form}, the conditions $p\Lambda \subset \mathcal F\Lambda
\subset \Lambda$ are seen to be equivalent to the statement that
\[ \beta^{-1} \delta \sigma(\beta) \in G(W(k))A_pG(W(k))\]
where $A_p = \left(\begin{smallmatrix} I & 0 \\ 0 & pI \end{smallmatrix}\right)$.  
}
The proof involves two steps, first associating $L \subset T\otimes\QQ_p$ to $\bar\Lambda := 
L \otimes W(\bar k)\subset T\otimes K(\bar k)$ 
and second, associating $\bar\Lambda \mapsto \Lambda = \Phi(\bar\Lambda) \cap \mathcal J_{\QQ}(\gamma)$.
If $L \subset T\otimes\QQ_p$ satisfies $FL\subset L$, $V L \subset L$ and $L^{\vee}
= cL$ (for some $c \in \QQ_p$) then the same is true for $\bar\Lambda$ as well as the additional fact that
$\sigma \bar\Lambda = \bar\Lambda$.  Conversely, suppose $\bar\Lambda\subset T\otimes K(\bar k)$ is a
$W(\bar k)$-lattice such that 
\begin{equation}\label{eqn-stuff}
\sigma\bar\Lambda = \bar\Lambda,\ \bar\Lambda^{\vee} = c \bar\Lambda,\
\gamma \bar\Lambda\subset\bar\Lambda,\  q\gamma^{-1}\bar\Lambda
\subset\bar\Lambda\end{equation}
(for some $c \in K(\bar k)$), then  $c$ may be chosen to lie in $\QQ_p$ and we may write
$\Phi(\bar\Lambda) = \beta \bar\Lambda_0$ where $\bar\Lambda_0 = W(\bar k)^{2n}$ and $\beta \in 
\GSp_{2n}(K(\bar k))$ has multiplier $c$.  It follows that 
$\beta^{-1}\sigma(\beta) \in \Sp_{2n}(W(\bar k))$ which has trivial Galois cohomology so
$\beta^{-1}\sigma(\beta) = s^{-1} \sigma(s)$ for some $s \in \Sp_{2n}(W(\bar k))$.  Therefore
$h := \beta s^{-1} \in \GSp_{2n}(\QQ_p)$ and $\Phi(\bar\Lambda) = h \bar\Lambda_0 = (hL_0)\otimes W(\bar k)$
where $L_0 = \ZZ_p^{2n}$.  In other words, $\bar\Lambda$ comes from $\Phi^{-1}hL_0\subset T\otimes\QQ_p$.

For the second step, given $\bar\Lambda\subset T\otimes K(\bar k)$ satisfying (\ref{eqn-stuff}) let
$\Lambda = \bar\Lambda^{\Gal}=\bar\Lambda \cap M_{\QQ}(T)$.  The same proof as that in
Proposition \ref{prop-Dieudonne-module} implies that $\Lambda$ is a $W(k)$-lattice that is
preserved by $\mathcal F$ and $\mathcal V$.  Conversely, if $\Lambda \subset M_{\QQ}(T)$
is a $W(k)$-lattice that is preserved by $\mathcal F$ and $\mathcal V$ then 
$\gamma \Lambda \subset \Lambda$ and $q\gamma^{-1}\Lambda \subset\Lambda$ so the same is true
of $\bar\Lambda = \Lambda \otimes W(\bar k)$.
We claim that $\bar\Lambda \cap (M_{\QQ}(T))) = \Lambda$.
Choose a $W(k)$-basis $b_1,b_2,\cdots, b_{2n} \in T\otimes K(\bar k)$ of $\Lambda$.   If
$v=\sum_i s_i b_i \in \bar\Lambda \cap (M_{\QQ}(T)))$ with $s_i \in 
W(\bar k)$ then 
\[
v=\sum_is_ib_i = \gamma^{-1} \sigma^{-a}A_q\sum_i s_ib_i = \sum_i \sigma^{-a}(s_i) \gamma^{-1} A_q
\sigma^{-a}(b_i) = \sum_i \sigma^{-a}(s_i)b_i
\]
which implies that $s_i \in W(k)$. Therefore $v \in \Lambda$.

For the equivariance statement, suppose that $L \in \mathcal L_p(\QQ_p^{2n},\gamma,\B_0)$ is a lattice, say,
$L = gL_0 \subset \QQ_p^{2n}$. Let $\Lambda\subset\mathcal J_{\QQ}(\gamma)$ 
be the corresponding Dieudonn\'e module, say, $\Lambda =  B\beta \Lambda_0$ where $\beta \in \GSp_{2n}(K(k))$.
To show that the mapping $g \mapsto \beta$ is equivariant we must verify the following:
  if $h \in Z_{\gamma}(\QQ_p)$ and if $L' = (hg)L_0=hL$ then the corresponding 
Dieudonn\'e module (or $W(k)$-lattice) is 
\[\Lambda' = B(\mathfrak c_B(h)\beta)\Lambda_0 = B B^{-1}hB\beta\Lambda_0 = h\Lambda.\]
This amounts to the straightforward verification that $h\Lambda = ((hL)\otimes W(\bar k))^{\Gal}$.
\end{proof}

\begin{cor}\label{cor-lattice-count-p}
As in Proposition \ref{lem-Y} let
\[ \mathcal O_{\gamma,p} = \int_{Z_{\gamma}(\QQ_p)\backslash G(\QQ_p)} f_p(y^{-1}\gamma y) dy
= \int_{Z_{\gamma}(\QQ_p)\backslash G(\QQ_p)} k_p(y^{-1} \alpha_p \gamma y) dy\]
where $f_p$ denotes the characteristic function on $G(\QQ_p)$ of 
$K_p \left(\begin{smallmatrix} I & 0 \\ 0 & qI
\end{smallmatrix}\right)K_p$ and $k_p$ denotes the characteristic function of $K_p$.  Let
\[T\mathcal O_{\delta} = \int_{Z'_{\delta}(K(k))\backslash G(K(k))}\phi_p(y^{-1} \delta \sigma(y))dy
= \int_{Z'_{\delta}(K(k))\backslash G(K(k))}\psi_p(y^{-1} u_p \delta \sigma(y)) dy
\] where $\phi_p$ denotes the characteristic function 
on $G(K(k))$ of  $G(W(k)) \left(\begin{smallmatrix} I & 0 \\ 0 & pI \end{smallmatrix}
\right) G(K(k))$ and where $\psi_p$ denotes the characteristic function of $G(W(k))$.  Then
the number of elements in $Z_{\gamma}(\QQ_p)\backslash Y_p$ is
\[ \left|Z_{\gamma}(\QQ_p)\backslash Y_p\right| = \mathcal O_{\gamma,p}
= \vol\left(\mathfrak c_B(Z_{\gamma}(\QQ_p))\backslash Z'_{\delta}(K(k)) \right) T\mathcal O_{\delta}.
\qedhere\]
\end{cor}


\subsection{}  
Continue with the same notation $T,F,\gamma, \B, \Phi, B, \delta$ as in the previous paragraph but
now suppose that $\tau$ is a real structure on the polarized Deligne module $(T,F,\B)$ and let
$\tau_p, \bar\tau_p$ be the corresponding involution on $T\otimes \QQ_p$ and on
$M_{\QQ}(T) \cong \mathcal J_{\QQ}(\gamma)$ as in Proposition \ref{prop-dieudonne-basis}. 
Let  $\mathcal L_p(T\otimes \QQ_p, F, \B,\tau)$ denote the set of $\ZZ_p$-lattices 
$L \subset T \otimes \QQ_p$ such that
\[ L^{\vee} = cL, FL \subset L, VL \subset L, \tau L =L\]
(for some $c\in\QQ_p$).  Using the mapping $\Phi$ we may identify this with the
set $\mathcal L_p(\QQ_p^{2n}, \gamma, \B_0,\tau_0)$ of lattices $L\subset \QQ_p^{2n}$ such
that $L^{\vee} = cL$, $\tau_0L=L$ and $qL \subset \gamma L \subset L$.
Lemma \ref{lem-Smith-form} and Proposition \ref{prop-number-real-lattices} 
give a one to one correspondence
\begin{equation*}\begin{CD}
X_p @>{\cong}>> \mathcal L_p(\QQ_p^{2n}, \gamma, \B_p, \tau_0).
\end{CD}\end{equation*}
The same argument gives a one to one correspondence
\begin{equation*}\begin{CD}
\mathcal X_p @>{\cong}>> \mathcal L_p(\mathcal J_{\QQ}(\gamma), \mathcal F, \B_p, \bar\tau_0)
\end{CD}\end{equation*}
between
\begin{equation}\label{eqn-Zp}\begin{split}
\mathcal X_p := &\left\{ x \in H(K(k))/H(W(k))|\ x^{-1} \delta \sigma(x) \in H(W(k)) A_p H(W(k))\right\}\\
= &\left\{x\in H(K(k))/H(W(k))|\
x^{-1} u_p^{-1} \delta \sigma(x) \in H(W(k))\right\}
\end{split}\end{equation}
and  the set
$\mathcal L_p(\mathcal J_{\QQ}(\gamma), \mathcal F, \B_0,\bar\tau_0)$ of $W(k)$-lattices $\Lambda
\subset \mathcal J_{\QQ}(\gamma)$ such that $\Lambda^{\vee} = c\Lambda$, $\bar\tau_0\Lambda=\Lambda$ 
and $p\Lambda \subset \mathcal F\Lambda \subset \Lambda$.  By Proposition \ref{cor-group-description} 
the vertical arrows in the following diagram are bijective.
\begin{diagram}[size=2em]
\mathcal L_p(T\otimes\QQ_p, F, \B,\tau) &\rTo& \mathcal L_p(\QQ_p^{2n},\gamma,\B_0,\tau_0) 
&\lTo& X_p
\\
\dTo && \dTo && \dTo  \\
\mathcal L_p(M_{\QQ}(T), \mathcal F, \B_p,\tau_p) &\rTo& 
\mathcal L_p(\mathcal J_{\QQ}(\gamma),\mathcal F, \B_0,\bar\tau_0)
& \lTo & \mathcal X_p
\end{diagram}

As in Lemma \ref{lem-conj-B}, conjugation by $B$ induces an injective homomorphism
\[ \mathfrak c_B:S_{\gamma}(\QQ_p) \to S'_{\delta}(K(k))\]
from the centralizer of $\gamma$ in $H(\QQ_p)$ to the twisted centralizer
\begin{equation}
\label{eqn-twisted-centralizer} S'_{\delta}(K(k)) = \left\{ x \in \GL_n^*(K(k))|\
x^{-1} \delta \sigma(x) = \delta \right\}.\end{equation}

\begin{cor}\label{cor-twisted}
Let
\[ I_{\gamma,p} = \int_{S_{\gamma}(\QQ_p)\backslash H(\QQ_p)}
\chi_p(y^{-1} \alpha_q^{-1} \gamma y) dy\]
where $\chi_p$ is the characteristic function of $K_p = G(\ZZ_p)$.  Let
\[ TI_{\delta} = \int_{S'_{\delta}(K(k))\backslash H(K(k))}\kappa_p(z^{-1} u_p \delta 
\sigma(z)) dz\]
where $\kappa_p$ is the characteristic function of $G(W(k))$.  Then the number of elements in
$S_{\gamma}(\QQ_p)\backslash \mathcal X_p$ is
\begin{align*}
\left| S_{\gamma}(\QQ_p) \backslash X_p\right| &= 
\left|S_{\gamma}(\QQ_p)\backslash \mathcal X_p\right|
= I_{\gamma,p} \\
&= \vol(S_{\gamma}(\QQ_p) \backslash S'_{\delta}(K(k)))\cdot TI_{\delta}.
\end{align*}
\end{cor}

\section{The counting formula}\label{sec-statement-of-results}
\subsection{}  Throughout this chapter all polarizations are considered to be
$\Phi_{\varepsilon}$-positive, see \S \ref{subsec-polarizations}.
Fix a finite field $k = \FF_q$ with $q$ elements, and characteristic $p>0$.
Fix an embedding $\varepsilon:W(\overline{k}) \to \CC$.  As described in
\S \ref{sec-ordinary}  and Appendix \ref{appendix-polarizations} this provides an
equivalence of categories between the category of polarized Deligne modules and
the category of polarized ordinary Abelian varieties over $k$.  It also determines
a choice of CM type $\Phi_{\varepsilon}$ on the CM algebra $\QQ[F]$ for every
Deligne module $(T,F)$.  Let $\B_0$ denote
the {\em standard symplectic form} of equation (\ref{eqn-J}), that is,
$\B_0(x,y) = \tr{\!x}Jy\ \text{ where }\
J = \left( \begin{smallmatrix}
0&I\\-I&0\end{smallmatrix}\right)$ and let $G=\GSp_{2n}$ denote the algebraic
group of symplectic similitudes as in Appendix \ref{subsec-preliminaries}.
Let $\AA_f = \AA_f^p \QQ_p$ denote the finite ad\`eles of $\QQ$ as in
\S \ref{subsec-KpKN}.  Let $N$ be a positive integer relatively prime to $p$, or
let $N=1$. Let $\widehat{K}_N = \widehat{K}^p_NK_p$ be the principal congruence subgroup of
$\GSp_{2n}(\AA_f)$, see (\ref{eqn-Khat}). If $\gamma \in \GSp_{2n}(\QQ)$ is semisimple let
\[\mathcal O_{\gamma}^p = \int_{Z_{\gamma}({\AA^p_f})\backslash G(\AA^p_f)} f^p(x^{-1}\gamma x)dx
\ \text{ and }\
 \mathcal O_{\gamma,p} = \int_{Z_{\gamma}(\QQ_p)\backslash G(\QQ_p)} f_p(y^{-1}\gamma y) dy
\]
where  $f^p$ is the characteristic function on $G(\AA^p_f)$ of
$\widehat{K}^p_N$, and $f_p$ is the characteristic function on $G(\QQ_p)$ of
$K_p \left(\begin{smallmatrix} I & 0 \\ 0 & qI \end{smallmatrix} \right)K_p$,
see \S \ref{subsec-KpKN} and Proposition \ref{lem-Y}.

\begin{thm}\cite{Kottwitz, Kottwitz2}  \label{thm-statement1}
The number $A(q)$ of principally polarized ordinary Abelian
varieties with principal level $N$ structure, 
over the field $k = \FF_q$, is finite and is equal to
\begin{align}\label{eqn-statement1}
&\sum_{\gamma_0} \sum_{\gamma\in \mathcal C(\gamma_0)} 
\vol\left(Z_{\gamma}({\QQ})\backslash Z_{\gamma}({\AA_f})\right)\cdot
\mathcal O_{\gamma}^p \cdot \mathcal O_{\gamma,p}\\
\label{eqn-statement1b}
=&\sum_{\gamma_0}\sum_{\gamma \in \mathcal C(\gamma_0)} \vol
\left(Z_{\gamma}(\QQ)\backslash Z_{\gamma}(\AA^p_f)\times Z'_{\delta}(K(k))
\right)\cdot O^p_{\gamma}\cdot T\mathcal O_{\delta} 
\end{align}
\end{thm}

\subsection{Explanation and proof} \label{subsec-EandP}
The proof of equation (\ref{eqn-statement1})
follows five remarkable pages (pp.~203-207) in \cite{Kottwitz}.
Roughly speaking the first sum indexes the $\overline{\QQ}$-isogeny classes,
the second sum indexes the $\QQ$-isogeny classes within a given
$\overline{\QQ}$-isogeny class, and the orbital integrals count
the number of isomorphism classes within a given $\QQ$-isogeny class.

 The first sum is over rational representatives $\gamma_0\in \GSp_{2n}(\QQ)$, 
one from each $\GSp_{2n}(\overline{\QQ})$-conjugacy
classes of semisimple elements such that the characteristic polynomial
of $\gamma_0$ is an ordinary  Weil $q$-polynomial (see Appendix \ref{sec-Weil-polynomials}).  
By Proposition \ref{prop-viable}, each element $\gamma_0$ may be chosen to be 
{\em viable} with respect to the CM type $\Phi_{\varepsilon}$ (see \S \ref{subsec-viable}).  

Now let us fix such an element $\gamma_0$.
The second sum is over representatives $\gamma \in \GSp_{2n}(\QQ)$, one from each
$\GSp_{2n}(\QQ)$-conjugacy class, such that \begin{enumerate}
\item $\gamma, \gamma_0$ are $\GSp_{2n}(\overline{\QQ})$-conjugate and
\item $\gamma, \gamma_0$ are $\GSp_{2n}(\RR)$-conjugate. \end{enumerate}
Fix such an element $\gamma$.
As in Proposition \ref{prop-viable}, by conjugating by $\tau_0=\left(
\begin{smallmatrix} -I & 0 \\ 0 & I \end{smallmatrix}\right)$ if necessary, we may assume
that $\gamma, \gamma_0$ are $\GSp_{2n}(\RR)^+$-conjugate, which implies that $\gamma$
is also $\Phi_{\varepsilon}$-viable.

By Proposition \ref{prop-Qbar-conjugacy}, the terms in the first sum correspond to 
$\overline{\QQ}$-isogeny classes of polarized Deligne modules (see \S \ref{sec-isogeny}).
The chosen element $\gamma_0$ corresponds   to 
 some polarized Deligne module.  According to Lemma \ref{subsec-real-standard-form},
we may assume that it has the form $(L_0,\gamma_0, \B_0)$ where $L_0 \subset \QQ^{2n}$
is a lattice such that \begin{enumerate}
\item[(a)] $L_0$ is preserved by $\gamma_0$ and by $q \gamma_0^{-1}$ and
\item[(b)] the standard symplectic form $\B_0$ takes integral values on $L_0$.  
\end{enumerate}
According to Proposition \ref{prop-Q-Qbar} the terms in the second sum correspond to the
$\QQ$-isogeny classes of polarized Deligne modules within the $\overline{\QQ}$-isogeny
class of $(L_0, \gamma_0, \B_0)$.  The chosen element $\gamma$ arises from
some polarized Deligne module, say, $(T, \gamma, \B_0)$ where $T\subset \QQ^{2n}$ is a lattice
that also satisfies (a) and (b) above.  

The set of isomorphism classes of
principally polarized Deligne modules within the $\QQ$-isogeny class of $(T,\gamma,\B_0)$ is
identified, using Proposition \ref{prop-iso-in-isogeny}, with the quotient 
$Z_{\gamma}(\QQ)\backslash Y$ (see also equation \ref{eqn-WY}), where $Z_{\gamma}(\QQ)$ 
is the centralizer of $\gamma$ in $\GSp_{2n}(\QQ)$
and where $Y$ denotes the set of pairs $(\widehat{L},\alpha)$ consisting of a lattice
$\widehat{L} \subset \AA_f^{2n}$ and a level $N$ structure $\alpha$, 
satisfying (\ref{eqn-9.3.1}) and (\ref{eqn-9.3.2}), that is, $\widehat{L}$ 
is a lattice that is {\em symplectic up to homothety} (see \S \ref{subsec-adelic-lattices}) and is preserved by 
$\gamma$ and by $q\gamma^{-1}$, and the
level structure is compatible with $\gamma$ and with the symplectic structure.  
Decomposing the lattice $\widehat{L}$ into its ad\`elic components gives a product 
decomposition $Y \cong Y^p \times Y_p$ as described in Proposition \ref{lem-Y}.
This in turn leads to the product of orbital integrals in equation (\ref{eqn-statement1}).

Although the second sum in (\ref{eqn-statement1}) may have infinitely many terms, only finitely many of the
orbital integrals are non-zero.  This is a consequence of \cite{StableTrace} Prop. 8.2.

The integral $\mathcal O_{\gamma,p}= |Z_{\gamma}(\QQ_p)\backslash Y_p|$ can be expressed 
in terms of the (covariant) Dieudonn\'e module, 
\[M(T) = (T \otimes W(\bar k))^{\Gal},
\] 
(described in Proposition \ref{prop-Dieudonne-module} and Appendix \ref{appendix-Dieudonne}).  
Let $W(k)$ be the ring of Witt vectors of $k$ and let $K(k)$ be its fraction
field.  Recall that $M(T)$ is a free module of rank
$2n$ over $W(k)$, which carries a $\sigma$-linear operator $\mathcal F$ and a 
$\sigma^{-1}$-linear
operator $\mathcal V$ such that $\mathcal F \mathcal V = \mathcal V \mathcal F = p$
where $\sigma \in \Gal(K(k)/K(\FF_p))$ is the canonical lift of the Frobenius
mapping $x \mapsto x^p$ for $x\in k$.  When restricted to $M(T)$ the mapping $F$
factors as $F =\mathcal V^a$ (see \S \ref{subsec-Ffactors}).
Let $M_{\QQ}(T) = M(T)\otimes\QQ_p$ be the associated vector space.  
A choice of basis of $M_{\QQ}(T) \cong K(k)^{2n}$ (see Proposition 
\ref{prop-dieudonne-basis}) determines an element $\delta \in \GSp_{2n}(K(k))$, 
see equation (\ref{eqn-delta}), whose norm $\sf N(\delta)$ is
$G(W(\bar k))$-conjugate to $q\gamma^{-1}$, such that the action of $\mathcal F$
becomes the action of $\delta\sigma$ (cf. \cite{Kottwitz}).

Proposition \ref{cor-group-description} says that the mapping
\[ L_p \mapsto \Lambda = (L_p \otimes W(\bar k))^{\Gal} \subset M_{\QQ}(T)\]
defines a one to one correspondence between the set  of
symplectic (up to homothety) lattices $L_p \subset T\otimes \QQ_p$ that are preserved by
$F$ and $V$ and the set of symplectic (up to homothety) lattices
$\Lambda \subset M_{\QQ}(T)$ that are preserved by 
 $\mathcal F$ and $\mathcal V$.  This defines an isomorphism $Y_p \cong \mathcal Y_p$ where,
as in equation (\ref{eqn-curly-Y1}),
\[\mathcal Y_p = \left\{ \beta \in G(K(k))/G(W(k))|\ 
\beta^{-1} \delta \sigma(\beta) \in G(W(k)) \left(\begin{smallmatrix} I & 0 \\ 0 & pI \end{smallmatrix}
\right) G(W(k)) \right\}.\]
Let $\phi_p$ be the characteristic function on $G(K(k))$ of $G(W(k))
\left(\begin{smallmatrix} I & 0 \\ 0 & pI \end{smallmatrix}\right)
G(W(k))$.  We define
\[T\mathcal O_{\delta} =  | Z'_{\delta} \backslash \mathcal Y_p| = 
\int_{Z'_{\delta}(K(k))\backslash G(K(k))}\phi_p(y^{-1} \delta \sigma(y))dy
\]
where $Z'_{\delta}(K(k))$ denotes the $\sigma$-twisted centralizer of $\delta$ in $G(K(k))$, see
equation (\ref{eqn-twisted-delta}).   Corollary \ref{cor-lattice-count-p}
provides the factor that is needed to convert the orbital integral $\mathcal O_{\gamma,p}$ in
(\ref{eqn-statement1}) into the twisted orbital integral $T\mathcal{O}_{\delta}$ in 
(\ref{eqn-statement1b}), that is,
\[ \left|Z_{\gamma}(\QQ_p)\backslash Y_p\right| = \mathcal O_{\gamma,p}
= \vol\left(\mathfrak c_B(Z_{\gamma}(\QQ_p))\backslash Z'_{\delta}(K(k)) \right) T\mathcal O_{\delta}.
\]

This completes the proof of Theorem \ref{thm-statement1}.  \qed

\subsection{Counting real structures}  
Let $\tau_0$ be the standard involution on $\QQ^n \oplus \QQ^n$ 
(see Appendix \ref{sec-involutions}).
For $g \in \GSp_{2n}$ let $\tilde g = \tau_0 g \tau_0^{-1}$.  Define
$H = \GL_n^*\cong \GL_1 \times\GL_n$ to be the fixed point subgroup of this action, 
as in \S \ref{subsec-q-inversive-conjugacy}.  If $\gamma \in \GSp_{2n}$ let 
\[ S_{\gamma} = \left\{ x \in \GL_n^*|\ x \gamma = \gamma x \right\}.\]
Assume that $p \ne 2$, that the level $N$ is even and not divisible by $p$.
Let $\chi^p$ denote the characteristic function of
$\widehat{K}^p_N$ (see \S \ref{subsec-KpKN}) and let $\chi_p$ denote the characteristic function
of $K_p = \Gsp_{2n}(\ZZ_p)$.

\quash{
Let $I^p_{\gamma,t} = 
\int_{S_{\gamma}(\AA^p_f)\backslash H(\AA^p_f)}f^p_t(x^{-1} \gamma x) dx$
where $f^p_t$ is the characteristic function on $G(\AA^p_f)$ of the coset 
$g \widehat{K}^p_N g^{-1}$.  Definition of $t = g^{-1}\tilde{g}$ and $K^p_N$ etc.
are needed.
}

\begin{thm}  \label{thm-statement2}
The number of isomorphism classes of principally polarized
ordinary Abelian varieties with real structure is finite and is equal to:
\begin{align}\label{eqn-statement2a}
 &\sum_{A_0} \sum_{C} \left|\widehat{H}^1\right|
\vol\left( S_{\gamma}(\QQ) \backslash S_{\gamma}(\AA_f) \right)
\int_{S_{\gamma}(\AA_f)\backslash H(\AA_f)}
\chi^p(x^{-1} \gamma x)
\chi_p(x^{-1} \alpha_q^{-1}\gamma x) dx \\  \label{eqn-statement2b}
=&\sum_{A_0} \sum_{C}  \left|\widehat{H}^1\right|\vol(S_{\gamma}(\QQ) \backslash
(S_{\gamma}(\AA^p_f) \times S'_{\delta}(K(k)) )\cdot I^p_{\gamma} \cdot TI_{\delta}.
\end{align}
\end{thm}

\subsection{Explanation and proof}
As in Theorem \ref{thm-statement1} the first sum indexes the $\overline{\QQ}$-isogeny
classes, the second sum indexes $\QQ$-isogeny classes within a given 
$\overline{\QQ}$-isogeny class, and the orbital integrals count the  number of
isomorphism classes within a $\QQ$-isogeny class.

 The first sum is over representatives, one from each
$\GL_{n}({\QQ})$-conjugacy class (which is the same as the $\GL_{n}(\overline\QQ)$ conjugacy class)
of semisimple elements $A_0 \in \GL_{n}(\QQ)$ whose
characteristic polynomial $h(x)=b_0 + b_1x + \cdots + x^n\in \ZZ[x]$ satisfies
(see Appendix \ref{sec-Weil-polynomials})
\begin{enumerate}
\item[(h1)] $b_0\ne 0$ and $p \nmid b_0$ 
\item[(h2)] the roots $\beta_1,\beta_2,\cdots,\beta_n$ of $h$ are totally real and
$|\beta_i| <  \sqrt{q}$ for $1 \le i \le n$. 
\end{enumerate}
By Proposition \ref{prop-Qbar-conjugacy} the terms in this sum correspond to $\overline{\QQ}$-isogeny classes
of polarized Deligne modules with real structure.



Fix such an element $A_0 \in \GL_n(\QQ)$.  By Proposition \ref{prop-viable-pair}
there exist $B_0,C_0$ so that the element 
\[\gamma_0 = \left( \begin{smallmatrix} A_0 & B_0 \\ C_0 & \tr{\!A}_0 \end{smallmatrix}
\right) \in \GSp_{2n}(\QQ)\] 
is $q$-inversive (\S \ref{sec-q-real})   and viable (\S \ref{subsec-viable}) with
respect to the CM type $\Phi_{\varepsilon}$.  It corresponds to some polarized Deligne module
with real structure which (by Lemma \ref{subsec-real-standard-form}) may be taken to be of the form
 $(T_0, \gamma_0, \B_0, \tau_0)$
where $T_0 \subset \QQ^{2n}$ is a lattice that is preserved by $\tau_0$ and by $\gamma_0$. 
The second sum in (\ref{eqn-statement2a}) is over representatives $C\in \GL_n(\QQ)$, one from
each $Z_{{\rm GL_n}(\QQ)}(A_0)$-congruence class (\S \ref{subsec-q-inversive-conjugacy})
of matrices such that \begin{enumerate}
\item $C$ is symmetric and nonsingular
\item $A_0C=C\tr{\!A}_0$
\item  $\sig(A_0;C) = \sig(A_0;C_0)$ (cf.~\S \ref{subsec-q-inversive-conjugacy}).
\end{enumerate}
According to Proposition \ref{prop-real-Q-isogeny}, the elements in this sum 
correspond to $\QQ$-isogeny classes
of polarized Deligne modules with real structure that are in the same $\overline{\QQ}$-isogeny
class as $(T_0, \gamma_0, \B_0, \tau_0)$.  Let us fix such an element $C$ and let
 $\gamma= \left(\begin{smallmatrix}
A_0 & B \\ C & \tr{\!A} \end{smallmatrix} \right)$ be the corresponding element from Proposition 
\ref{prop-real-Q-isogeny} (where $B = (A_0^2-qI)C^{-1}$).  Then $\gamma$ is $q$-inversive and viable
and it corresponds to some polarized Deligne module with real structure,
say $(T,\gamma,\B_0,\tau_0)$ which we will use as a ``basepoint" in the 
${\QQ}$-isogeny class determined by $A_0,B$. 

(In fact, the first two sums may be replaced by a single sum over $\GL_n(\QQ)$-conjugacy
classes of semisimple elements $\gamma \in \GSp_{2n}(\QQ)$ that are $q$-inversive, whose characteristic
polynomial is an ordinary Weil $q$-polynomial, and that are $\Phi_{\varepsilon}$-viable.)

According to Proposition \ref{prop-iso-in-isogeny} the isomorphism classes of principally
polarized Deligne modules with real structure and level $N$ structure that are $\QQ$-isogenous
to $(T, \gamma, \B_0, \tau_0)$ correspond to isomorphism classes of pairs $(\widehat{L},\alpha)$ 
(consisting of a lattice $\widehat{L}\subset \AA_f^{2n}$ and a level structure) that 
satisfy (\ref{eqn-9.3.0}), (\ref{eqn-9.3.1}) and (\ref{eqn-9.3.2}).  
In Proposition \ref{prop-cohomology-lattice} these lattices are divided into
cohomology classes 
\[ [t] \in \widehat{H}^1 = 
H^1(\langle\tau_0\rangle, \widehat{K}_N^0)-.\]
Each cohomology class provides the same contribution, which accounts for the
factor of $|\widehat{H}^1|$.   The number of isomorphism
classes of pairs $(\widehat{L},\alpha)$ corresponding to each cohomology class is proven,
in Proposition \ref{prop-number-real-lattices}, to equal the value of the orbital integral in 
equation (\ref{eqn-statement2a}).
This completes the proof of equation (\ref{eqn-statement2a}).

The second sum in equation (\ref{eqn-statement2a}) (that is, the sum over $B$) may have
infinitely many terms.  However it follows from Theorem \ref{prop-finite-isomorphism} that only finitely
many of those terms are non-zero.

For equation (\ref{eqn-statement2b}) a few more words of explanation are needed.
As in the previous paragraph we fix a polarized Deligne module with real structure
and level structure, $(T,\gamma, \B_0,\tau_0)$, which we consider to be a ``basepoint"
within its $\QQ$-isogeny class. As in the previous section, each cohomology class
$t \in \widehat{H}^1$ provides the same contribution, which is given in
Proposition  \ref{prop-number-real-lattices}, namely, the integral
in (\ref{eqn-statement2a}) is described as a product of an integral away from $p$,
\[ I^p_{\gamma} = \int_{S_{\gamma}(\AA^p_f)\backslash H(\AA^p_f)} \chi^p(x^{-1} \gamma x) dx\]
and an integral at $p$,
\[ I_{\gamma,p} = \int_{S_{\gamma}(\QQ_p)\backslash H(\QQ_p)}
\chi_p(y^{-1} \alpha_q^{-1} \gamma y) dy.\]
The integral $I_{\gamma,p}$ ``counts" lattices in $\QQ_p^{2n}$ that are preserved by $\tau_0$,
and by $\gamma$ and $q \gamma^{-1}$.  We wish to express this count in terms of lattices in
the Dieudonn\`e module $M(T)$. 
 
 The involution $\tau_0$ does not automatically induce
an involution on the Dieudonn\`e module $M(T)$.  Rather, it is first necessary to
make a universal choice of continuous involution $\bar\tau$ on $W(\bar k)$  so that 
$\bar\tau \sigma^a(w) = \sigma^{-a}(\bar\tau w)$ (see Proposition \ref{prop-tau-Witt}).
The resulting involution $\tau_p = \tau_0 \otimes \bar\tau$ preserves the Dieudonn\`e module
and it exchanges the actions of $\mathcal F^a$ and $\mathcal V^a$ (see 
Proposition \ref{prop-dieudonne-basis}). Proposition \ref{cor-group-description} states
that $\ZZ_p$-Lattices in $\QQ_p^{2n}$ that are preserved 
by $\tau_0$, by $\gamma$ and by $q\gamma^{-1}$ correspond to $W(k)$-lattices in 
$K(k)^{2n}$ that are preserved by $\bar\tau\otimes\tau_0$, by $\mathcal F$ and by $\mathcal V$. 

As in \S \ref{subsec-EandP} the action of $\mathcal F$ is given by 
$\delta\sigma$ where $\delta \in \GSp_{2n}(K(k))$.
Then, in Corollary \ref{cor-twisted}, it is shown that 
\[I_{\gamma,p}= \vol(S_{\gamma}(\QQ_p) \backslash S'_{\delta}(K(k)))\cdot TI_{\delta} \]
where $TI_{\delta}$ is the twisted orbital integral
\[ TI_{\delta} = \int_{S'_{\delta}(K(k))\backslash H(K(k))}
\kappa_p(z^{-1} u_p \delta \sigma(z))dz.\]
Here, $S'_{\delta}$ is the twisted centralizer of $\delta$, cf. (\ref{eqn-twisted-centralizer}),
 and $\kappa_p$ is the characteristic function of $G(W(k))$, see Corollary \ref{cor-twisted}.  
This completes the proof of equation (\ref{eqn-statement2b}).

\subsection{Counting totally real lattice modules}  Let $q = p^r$ and fix $n \ge 1$.
Fix $N \ge 1$ not divisible by $p$.
Let $f$ be the characteristic function of the principal congruence subgroup
$\widehat{K}_N$ in $\GL_n(\AA_f)$.  Using arguments that are similar (but simpler)
than those in the preceding sections, we find that
{\em the number of isomorphism classes of totally real lattice modules of rank $n$ and
characteristic $q$ and level $N$ is equal to
\[ \sum_{\{A\}} \int_{Z_A(\AA_f)\backslash GL_n(\AA_f)} f(x^{-1} A x) dx\]
where the sum is taken over all $\GL_n(\QQ)$-conjugacy classes of semisimple
elements $A$ whose characteristic polynomial is a real Weil $q$-polynomial.}
For $n=1$ and $N=1$ if $\sqrt{q}$ is an integer, this number is $2(\sqrt{q}-1)$, so
the resulting zeta function is rational and is equal to the following,
\[Z(T) = \left(\frac{1-T}{1-\sqrt{q}T}\right)^2.\]

\section{Further questions}
\subsection{}
We do not know whether the count of the number of ``real" polarized Deligne modules has
a rational zeta-function interpretation.

\subsection{}
We do not know whether the notion of an anti-holomorphic involution makes sense for
general Abelian varieties over $\FF_q$.  It is not clear what to do, even in the case
of supersingular elliptic curves.

\subsection{}
In \cite{GT1} the authors showed that certain arithmetic hyperbolic 3-manifolds (and more
generally, certain arithmetic quotients of quaternionic Siegel space) can be viewed as
paramet\-rizing Abelian varieties with anti-holomorphic multiplication by the integers 
$\mathcal O_d$ in a quadratic imaginary number field.  It should be possible to 
mimic these constructions using Deligne
modules.  Define an anti-holomorphic multiplication on a Deligne module $(T,F)$
by an order $\mathcal O$ in a CM field $E$  to be a homomorphism
$\mathcal O \to \End(T)$ such that each purely imaginary element $u \in \mathcal O$ acts
in an anti-holomorphic manner, that is, $uF = Vu$.
One could then attempt to count the number of isomorphism classes of principally
polarized Deligne modules with level structure and with
anti-holomorphic multiplication by $\mathcal O$.

\appendix{}

\section{Weil polynomials and a real counterpart}\label{sec-Weil-polynomials}
\subsection{}\label{subsec-realpoly}
Let $\pi$ be an algebraic integer.  It is {\em totally real} if $\rho(\pi) \in \RR$
for every embedding $\rho:\QQ(\pi) \to \CC.$  It is a {\em Weil $q$-integer} if
$|\rho(\pi)|^2 = q$ for every embedding $\rho:\QQ(\pi) \to \CC.$  (In this case
the field $\QQ(\pi)$ is either a CM field, which is the usual case,
or it is $\QQ(\sqrt{q}),$ the latter case occurring iff $\pi = \pm \sqrt{q}.$)
A monic polynomial $p(x) \in \ZZ[x]$ is {\em totally real} if all of its roots are totally real
algebraic integers.  A {\em Weil $q$-polynomial} is a monic polynomial $p(x) \in \ZZ[x]$
of even degree, all of whose roots are Weil $q$-integers.  Let us say that a Weil $q$-polynomial
$p(x) = \sum_{i=0}^{2n}a_ix^i$ is {\em ordinary} if the middle coefficient $a_n$ is nonzero and is
coprime to $q.$  This implies that half of its roots are $p$-adic units and half of its roots
are divisible by $p$; also that $x^2 \pm q$ is not a factor of $p(x)$, hence $p(x)$ has
no roots in the set $\left\{ \pm \sqrt{q}, \pm \sqrt{-q}\right\}$.
A {\em real} (resp.~{\em real ordinary}) Weil $q$-polynomial of degree $n$ is a 
monic polynomial $h(x)$ such that
the polynomial $p(x) = x^n h(x+q/x)$ is a Weil $q$-polynomial (resp.~an ordinary Weil $q$-polynomial) 
(see also \cite{Howe_Lauter1, Howe_Lauter2}).
  The characteristic polynomial of Frobenius associated to an
Abelian variety $B$ of dimension $n$ defined over the field $\FF_q$ is a Weil
$q$-polynomial.  It is ordinary if and only if the variety $B$ is ordinary,
cf. \S \ref{sec-ordinary}.

\subsection{Real counterpart}\label{subsec-real-counterpart}
Let $q\in \QQ.$  Let us say that a monic polynomial $p(x) = x^{2n} + a_{2n-1}x^{2n-1} +
\cdots + a_0\in \CC[x]$ is {\em $q$-palindromic} if it has even degree and
if $a_{n-r} = q^ra_{n+r}$ for $1 \le r \le n,$ or equivalently if
\[ q^{-n} x^{2n} p\left(\frac{q}{x}\right) = p(x).\]
Thus $p(x)$ is $q$-palindromic iff the following holds:  for every root $\pi$ of $p(x)$
the number $q\pi^{-1}$ is also a root of $p(x)$, and if $\pi$ is a real root of $p(x)$
then $\pi = \pm \sqrt{q}$ and its multiplicity is even.
It is easy to see that every Weil $q$-polynomial is $q$-palindromic but the converse is not
generally true. Let
\[p(x) = \prod_{j=1}^n\left(x-\alpha_j\right)\left(x-\frac{q}{\alpha_j}\right)
= \sum_{i=0}^{2n}a_ix^i\]
be a $q$-palindromic polynomial with no real roots.  Define the {\em associated real
counterpart} 
\[ h(x) = \prod_{j=1}^n\left(x-\left(\alpha_j+\frac{q}{\alpha_j}\right)\right)
= \sum_{i=0}^n b_ix^i\]
or equivalently, $p(x) = x^nh(x + q/x)$.  If $h(x)\in\ZZ[x]$ has integer coefficients then the 
same is true of $p(x)$.  The following proposition gives a converse to this statement.

\begin{prop}\label{prop-h(x)}  Fix $n,q \in \ZZ$ with $n>0$ and $q>0.$
There exists a universal $(n+1) \times (2n+1)$ integer matrix $A$ and a
universal $(n+1) \times (n+1)$ integer matrix $B$ with the following
property. For every
$q$-palindromic polynomial $p(x) = \sum_{t=0}^{2n} a_tx^t\in \CC[x]$ with no real roots,
if $h(x) = \sum_{k=0}^n b_kx^k$
is the associated real counterpart, then for all $i,$ $0 \le i \le 2n$ we have:
\[a_t =  \sum_{k=0}^n A_{tk}b_k \text{ and }\
b_k = \sum_{s=0}^{n}B_{ks}a_{n+s}.\]
In particular, $p(x)$ has integer coefficients iff $h(x)$ has integer coefficients.  Moreover,
\begin{enumerate}
\item A $q$-palindromic polynomial $p(x)\in \ZZ[x]$ of even degree
is a Weil $q$-polynomial if and only if
the corresponding polynomial $h(x)$ is totally real.
\item A totally real polynomial $h(x) \in \ZZ[x]$ is the real counterpart to a Weil $q$-polynomial
$p(x)$ with no real roots if and only if the roots $\beta_1,\beta_2,\cdots, \beta_n \in \RR$ of
$h(x)$ satisfy $|\beta_i| < 2\sqrt{q}$ for $i = 1,2,\cdots,n.$
\item A Weil $q$-polynomial $p(x)\in \ZZ[x]$ is
ordinary if and only if the constant coefficient $h(0)=b_0$ of the real counterpart
is nonzero and is coprime to $q$.  In this case, $p(x)$ is irreducible over $\QQ$ if and only
if $h(x)$ is irreducible over $\QQ$.
\end{enumerate}
\end{prop}
\begin{proof}
Let $p(x)=\sum_{k=0}^{2n} a_kx^k\in \CC[x]$ be a $q$-palindromic polynomial with roots
\[\left\{\alpha_1, \frac{q}{\alpha_1},\cdots, \alpha_n, \frac{q}{\alpha_n}\right\}.\]
The real counterpart is
$ h(x)=\sum_{j=0}^n b_jx^j = \prod_{i=1}^n(x - \beta_i)$ where
$\beta_i = \alpha_i + \frac{q}{\alpha_i},$  hence
\begin{align*} p(x) &= x^nh\left(x + \frac{q}{x}\right)\\
&= \sum_{j=0}^nb_j\ \sum_{t=0}^j\binom{j}{t}q^{j-t}x^{n-j+2t}.
\end{align*}
Set $r = n-j+2t.$  Then $n-j \le r \le n+j$ and $r-(n-j)$ is even, hence
\[ p(x)=\sum_{r=0}^{2n}a_rx^r = \sum_{j=0}^n \ 
\sum_{r=n-j}^{n+j}A_{rj}b_jx^r
\]
where
\[ A_{rj} = \binom{j}{\frac{r+j-n}{2}}q^{\frac{n-r+j}{2}}\]
provided that $r+j-n$ is even and that $n-j \le r \le n+j$, and
$A_{rj}= 0$ otherwise.

In particular, $A_{n+s,s} = 1$ for all $0 \le s \le n,$ and the lower half of the
matrix $A$ is nonsingular with determinant equal to 1.  We may take $B$ to be the
inverse of the lower half of $A.$
For $n=4$ we have: \renewcommand{\t}{{\scriptscriptstyle{\cdot}}}
\[A =\left[\begin{matrix}
\t&\t&\t&\t&q^4\\
\t&\t&\t&q^3&\t\\
\t&\t&q^2&\t&4q^3\\
\t&q&\t&3q^2&\t&\\
1&\t&2q&\t&6q^2\\
\t&1&\t&3q&\t\\
\t&\t&1&\t&4q\\
\t&\t&\t&1&\t\\
\t&\t&\t&\t&1\end{matrix}\right]\]
This proves the first part of the Proposition.

To verify statement (1) let $p(x)$ be a Weil $q$-polynomial.  If it has any real roots
then they must be of the form $\alpha = \pm \sqrt{q}$ so $\alpha + q/\alpha = \pm 2\sqrt{q}$
which is real.  Every pair $\{\alpha, q/\alpha\}$ of complex roots are necessarily
complex conjugate hence  $\beta=\alpha + q/\alpha$ is real.  Since
$h(x)$ has integer coefficients this implies that every Galois conjugate of $\beta$ is
also real, hence $h(x)$ is a totally real polynomial.  Conversely, given $p(x)$, if the
associated polynomial $h(x)$ is totally real then for each root $\beta = \alpha + q/\alpha$
of $h(x)$, the corresponding pair of roots $\{\alpha, q/\alpha\}$ are both real or else they
are complex conjugate, and if they are real then they are both equal to $\pm \sqrt{q}$.  This
implies that $p(x)$ is a Weil $q$-polynomial.

For part (2) of the proposition, each root $\beta_i\in \RR$ of $h(x)$ is a sum $\beta_i = \alpha_i +
q/\alpha_i$ of complex conjugate roots of $p(x).$  Hence $\alpha_i$ and $q/\alpha_i$ are the
two roots of the quadratic equation
\[ x^2 - \beta_ix + q = 0\]
which has real solutions if and only if ${\beta_i^2-4q}\ge 0.$  Thus, $p(x)$ has no real
roots if and only if $|\beta_i| < 2 \sqrt{q}$ for $i = 1, 2, \cdots, n.$

For part (3), the polynomial $p(x)$ is ordinary if and
only if exactly one of each pair of roots $\alpha_i, q/\alpha_i$ is
a $p$-adic unit, from which
it follows that each $\beta_i = \alpha_i + q/\alpha_i$ is a $p$-adic unit, hence the
product $b_0 = \prod_{i=1}^n \beta_i$ is a $p$-adic unit (and it is nonzero).
Conversely, if $b_0$ is a $p$-adic
unit then so is each $\beta_i$ so at least one of the elements in each pair
$\alpha_i, q/\alpha_i$ is a unit.  But in \cite{Howe,Deligne} it is shown that this implies
that
exactly one of each pair of roots is a $p$-adic unit, so $p(x)$ is ordinary.  The
irreducibility statement follows from the formula $p(x) = x^n h(x+q/x)$.
\end{proof}

\begin{lem}\label{lem-palindromic}
Let $\gamma \in \Gsp_{2n}(\QQ)$ with multiplier $q \in \QQ.$  Then the characteristic
polynomial $p(x)$ of $\gamma$ is $q$-palindromic.
\end{lem}

\begin{proof}
First consider the case that $\gamma$ is semisimple.  Then over
$\overline{\QQ}$ it can be diagonalized,
$\gamma = \left(\begin{smallmatrix} D & 0 \\ 0 & D' \end{smallmatrix}\right)$ where
$D = \Diag(d_1,d_2,\cdots,d_n)$ and $D'=\Diag(d'_1,d'_2,\cdots,d'_n)$ are diagonal
matrices with $DD'=qI.$  Consequently $d'_i = q/d_i$ so
\[ p(x) = \prod_{i=1}^n(x-d_i)(x-d'_i) = \prod_{i=1}^n (x^2 -2\alpha_ix+q)\]
where $\alpha_i = \frac{1}{2}\left(d_i + {q}/{d_i}\right).$  If $d_i =
\pm \sqrt{q}$ then $d'_i = \pm\sqrt{q}$ so the polynomial $x^2-2\alpha_ix+q$ is
$q$-palindromic.  But a product of $q$-palindromic polynomials is also
$q$-palindromic, hence $p(x)$ is $q$-palindromic.

For the general case, use the Jordan decomposition to write $\gamma$ as a commuting
product,  $\gamma = \gamma_s \gamma_u$ where  $\gamma_s\in \gsp_{2n}( \overline{\QQ})$
is semisimple and $\gamma_u\in \gsp_{2n}( \overline{\QQ})$
is unipotent.  Since $\gamma_u$ has multiplier $1$, the
semisimple element $\gamma_s$ has multiplier $q.$  Moreover,
$\gamma_s$ lies in the closure of the conjugacy class of $\gamma$
(see for example, \cite{Humphreys} \S 1.7), so the characteristic polynomial of
$\gamma_s$ coincides with that of $\gamma.$
\end{proof}

The following proposition is a converse to Lemma \ref{lem-palindromic}.  It provides a 
{\em companion matrix} for the symplectic group.  Companion matrices for the symplectic
group were described in \cite{Kirby} but the following construction provides a matrix
representative $\gamma$ that is also $q$-inversive (see \S \ref{subsec-q-real}), 
meaning that $\tau_0 \gamma \tau_0^{-1}
= q \gamma^{-1}$ where $\tau_0$ is the {\em standard involution} (\S \ref{sec-involutions}).
\begin{prop}\label{prop-companion}
Let $p(x) = \sum_{i=0}^{2n} a_ix^i\in \QQ[x]$ be a $q$-palindromic polynomial 
of degree $2n.$  Then there exists a \qreal element $\gamma \in \Gsp_{2n}(\QQ)$ with 
multiplier $q,$ whose characteristic polynomial is $p(x).$  Moreover,
$\gamma$ may be chosen to be semisimple, in which case it is uniquely determined
up to conjugacy in $\GSp_{2n}(\overline{\QQ})$ by its characteristic polynomial $p(x)$.
\end{prop}

\begin{proof}  There are an even number of real roots.  Write $p(x) = r(x)p'(x)$ where
$r(x)$ has only real roots and $p'(x)$ has only complex roots.  Then 
\[ r(x) = \prod_{i=1}^{a}(x-\sqrt{q})^2\prod_{j=1}^{b}(x+\sqrt{q})^2
= \prod_{k=1}^{\min(a,b)}(x^2-q)^2\prod_{\ell=1}^{|a-b|}(x^2 \pm 2\sqrt{q}x +q)^2.\]
Since $p(x) \in \QQ[x]$, either $a = b$ or else $\sqrt{q}$ is rational (in which
case it is an integer).  The first factor corresponds to a product of symplectic
matrices $\left(\begin{smallmatrix} 0 & q \\ 1 & 0 \end{smallmatrix} \right)$ while the
second factor corresponds to the symplectic matrix $\pm\left( \begin{smallmatrix}
\sqrt{q}I & 0 \\ 0 & \sqrt{q}I \end{smallmatrix} \right)$, both of which have multiplier $q$.

Thus we may assume the polynomial $p(x)$ has only complex roots, so it
can be factored as follows:
\[ p(x) = \prod_{i=1}^n (x-\lambda_i)(x-q/\lambda_i) =
\prod_{i=1}^n(x^2 - (\lambda_i + q/\lambda_i)x + q)\]
where $\lambda_1,\cdots,\lambda_n, q/\lambda_1,\cdots,q/\lambda_n$ are the
roots of $p(x).$  Set $\alpha_i = \frac{1}{2}(\lambda_i + q/\lambda_i$) and
define
\[ h(x) = \prod_{i=1}^n(x - \alpha_i)
= -h_0 -h_1x -\cdots -h_{n-1}x^{n-1}
+ x^n.\]
(For convenience in this section, the signs of the coefficients of $h(x)$ have been
modified from that of the preceding section.)
The desired element $\gamma$ is $\gamma = \left(
\begin{smallmatrix} A & B \\ C & {}^t{A} \end{smallmatrix} \right)$  where the matrices
$A,B,C$ are defined as follows.  The matrix $A$ is the companion matrix
for the polynomial $h(x)$, that is,
\[ A = \left( \begin{array}{cccccc}
0&0&0&\cdots& 0 & h_0\\
1&0&0&\cdots&0&h_1\\
0&1&0&\cdots&0&h_2\\
0&0&1&\cdots&0&h_3\\
&&&\cdots&&\\
0&0&0&\cdots&1&h_{n-1}
\end{array} \right).\]
It is nonsingular (but not necessarily semisimple unless the roots of $h(x)$ are
distinct).  Now define
\[ B = \left( \begin{array}{cccccc}
&&&&h_0&0\\
&&&h_0&h_1&0\\
&&h_0&h_1&h_2&0\\
&&&\cdots&&\\
h_0&h_1&h_2&\cdots&h_{n-1}&0\\
0&0&0&\cdots&0&1
\end{array} \right).\]
Then $B$ is symmetric and nonsingular, and one checks directly that $AB = B\tr{\!A}.$
Define $C = B^{-1}\left(A^2-qI\right)$   so that $A^2-BC = qI.$  These conditions
guarantee that $\gamma \in \GSp_{2n}(\QQ),$ its multiplier is $q$, and it is \qreal.
Since the characteristic polynomial of $A$ is $h(x)$, Lemma \ref{lem-q-real} implies
that the characteristic polynomial of $\gamma$ is $p(x).$  

If the roots of $p(x)$ are distinct then this element $\gamma$ is semisimple.  However
if $p(x)$ has repeated roots it is necessary to proceed as follows.  Factor
$h(x) = \prod_{j=1}^r h_j^{m_j}(x)$ into its irreducible factors over $\QQ.$  This
corresponds to a factorization $p(x) = \prod_{j=1}^r p_j^{m_j}(x)$ into $q$-palindromic
factors.  Take $A = \Diag(A_1^{\times m_1},\cdots,A_r^{\times m_r})$ to be a 
block-diagonal matrix with $m_j$ copies of the matrix $A_j.$  Then $B,C$ will
also be block-diagonal matrices, and $\gamma$ will be the corresponding product
of \qreal symplectic matrices $\gamma_j.$    It suffices to show that each nonzero
$\gamma_j$ is semisimple.    Since $h_j(x)$ is irreducible over $\QQ$, its
roots are distinct, and the roots of $p_j(x)$ are the solutions to $x^2-2\alpha x +q=0$
where $h_j(\alpha)=0.$  If $\pm\sqrt{q}$ is not a root of $h_j(x)$ then the roots of
$p_j(x)$ are distinct, hence $\gamma_j$ is semisimple.  If $\pm\sqrt{q}$ is a root of
$h_j(x)$ then $p(x) = (x-\sqrt{q})^2$ or $p(x) = (x^2-q)^2$  depending on whether or not
$\sqrt{q} \in \QQ.$  In the first case we may take $A_j=\sqrt{q}$; $B_j=C_j=0$ and in
the second case we may take $A_j = \left(\begin{smallmatrix} 0&1\\q&0\end{smallmatrix}\right)$;
$B_j = C_j = 0.$
\end{proof}

\quash{
\begin{cor}\label{cor-companion}
Let $A_0 \in \GL_n(\QQ)$ be semisimple and suppose that its characteristic polynomial is an
ordinary Weil $q$-polynomial.  Then there exist $B_0,C_0$ so that the matrix $\gamma_0 = 
\left(\begin{smallmatrix} A_0 & B_0 \\ C_0 & \tr{A}_0 \end{smallmatrix} \right)$ is in
$\GSp_{2n}(\QQ)$ and is $q$-inversive and semisimple.
\end{cor}
\begin{proof}
The matrix $A_0$ is $\GL_n(\QQ)$-conjugate to some block diagonal matrix 
\[A =
\Diag(A_1^{\times m_1} , \cdots, A_r^{\times m_r})\] 
as above,  where each $A_i$ is a semisimple companion matrix.  There exists $u \in \GL_n(\QQ)$
such that $A_0 = u A u^{-1}$.   The above procedure produces a $q$-inversive
matrix $\gamma = \left( \begin{smallmatrix} A & B \\ C & \tr{A} \end{smallmatrix}\right)$.  Then
$\gamma_0 = U\gamma U^{-1}$ has the desired properties, where $U = \left(
\begin{smallmatrix} u & 0 \\ 0 & {}^t{u}^{-1} \end{smallmatrix}\right) \in \Sp_{2n}(\QQ)$ has
$\widetilde U = U$. 
\end{proof}
}

\section{Symplectic group}\label{appendix-symplectic}
\subsection{}\label{subsec-preliminaries}
Let $R$ be an integral domain and let $T$ be a free, finite dimensional $R$ module.
Let us say that an alternating form $\B:T \times T \to R$
 is {\em strongly non-degenerate}, if the induced mapping
$\B^{\sharp}:T \to \Hom_R(T,R)$ is an isomorphism, and it is {\em weakly non-degenerate}
if $\B^{\sharp} \otimes K$ is an isomorphism, where $K$ is the fraction field of $R$.  
Denote by $\GSp(T,\B)$ the set of $g \in \GL(T)$
such that $\B(gx,gy) = \lambda \B(x,y)$ for some $\lambda = \lambda(g) \in
R^{\times}.$  Then $\lambda$ is a character of $\GSp(T,\B)$
and we say that $g \in \GSp(T,\B)$ has {\em multiplier} $\lambda(g).$  The
{\em standard symplectic form} on $T = R^{2n}$ is
\begin{equation}\label{eqn-J}
\B_0(x,y) = \tr{\!x}Jy\ \text{ where }\
J = \left( \begin{matrix}
0&I\\-I&0\end{matrix}\right).\end{equation}
If $\B:T\times T \to R$ is a symplectic form then a {\em symplectic basis} of
$T$ is an isomorphism $\Phi:T \to R^{2n}$ which takes $\B$ to the standard symplectic
form $\B_0$.
\quash{
or equivalently, it is  a basis  $\left\{ e_1,
\cdots, e_n, f_1,\cdots, f_n\right\}$ such that
\[ \B(e_i,e_j) = \B(f_i,f_j) = 0\ \text{ and }\ \B(e_i,f_j) = \delta_{ij}.\]
}
By abuse of notation we will write
\[\GSp_{2n}(R)=\GSp_{2n}(R)=\GSp(R^{2n},\B_0) = \GSp(R^{2n}, J)\]
for the group of automorphisms of $R^{2n}$ that preserve the standard symplectic form.
If $\gamma \in
\GSp_{2n}(R)$ then so is $\tr\gamma^{-1},$ hence $\tr{\gamma}$ is also.  In
this case, expressing $\gamma$ as a block matrix,
\[ \gamma = \left( \begin{matrix} A & B \\ C & D \end{matrix} \right)\]
the symplectic condition $\tr{\gamma} J {\gamma} = qJ$ is equivalent to:
\begin{equation}\label{eqn-symplectic-conditions}
 \tr{\!A}C, \tr{B}D\ \text{ are symmetric, and }
\tr{\!A}D-\tr{C}B = qI\end{equation}
or equivalently,
\begin{equation}\label{eqn-symplectic-conditions2}
q\gamma^{-1} =  \left( \begin{matrix}
\tr{D} & -\tr{B} \\ -\tr{C} & \tr {A} \end{matrix} \right)
\end{equation}
or
\begin{equation}\label{eqn-symplectic-conditions3}
A \tr{B}, C \tr{D} \text{ are symmetric, and } A\tr{D}-B\tr{C} = qI.
\end{equation}

The following standard fact is an integral version of Darboux' theorem on symplectic coordinates.
\begin{lem}\label{lem-Darboux}  Let $R$ be an integral domain and 
let $\B:T\times T {\longrightarrow}  R$ be a strongly non-degenerate symplectic form.  
Then $T$ admits a  symplectic basis.  If $L', L^{\prime\prime}\subset T$ are Lagrangian
submodules such that $T = L' \oplus L^{\prime\prime}$ then the basis may be chosen
so that $L'$ and $L^{\prime\prime}$ are spanned by basis elements.  
\end{lem}

\begin{proof} We prove the first statement by induction.  Since $\B$ is strongly non-degenerate
there exist $x_1,y_1 \in T_1$ so that $\B(x_1,y_1) = 1.$  Let $T_1$ be the span
of $\left\{x_1,y_1\right\}.$  Then the restriction of $\B$ to $T_1$ is a strongly
non-degenerate symplectic form.  If $\dim(T) = 2$ then we are done.  Otherwise
we claim that $T = T_1 \oplus T_1^{\perp}.$  Clearly, $T_1 \cap (T_1)^{\perp} = 0.$
Moreover, given any $v \in T$ let
$u = v - \B(v,y_1)x_1 - \B(v,x_1)y_1.$  Then $u \in (T_1)^{\perp}$ hence
$T = T_1 \oplus T_1^{\perp}.$  Since $\B$ is strongly nondegenerate this implies
that $\B$ induces an isomorphism
\[ T_1 \oplus T_1^{\perp} \cong \Hom(T,R) \cong \Hom(T_1,R)\oplus \Hom(T_1^{\perp},R)\]
hence the restriction of $\B$ to $T_1^{\perp}$ is also strongly nondegenerate.  So by
induction there is a symplectic basis of $T_1^{\perp}.$  Combining this with
$x_1,y_1$ gives a symplectic basis for $T.$

If $T = L' \oplus L^{\prime\prime}$ is a decomposition into Lagrangian submodules, then
the symplectic form induces an isomorphism $L^{\prime\prime} \cong \Hom_R(L',R)$.  Therefore
an arbitrary basis of $L'$ together with the dual basis of $L^{\prime\prime}$ will constitute
a symplectic basis for $T$.
\end{proof}
\quash{
The last statement is proven in \cite{Freitag} Satz A.5.4, but here is an outline.  We use the standard
symplectic form on $R^{2n}$.  It suffices to prove there
exists $g \in \Sp_{2n}(R)$ so that $gv = \tr{(}1,0,\cdots,0)$ for then $g^{-1}$ takes the standard 
basis to the desired basis. Let $v = (a_1,\cdots,a_n, b_1,\cdots, b_n)$.  Acting by permutations
we can arrange that $|a_1|<|a_j|$ for $2 \le j \le n$.  Acting by
$\left(\begin{smallmatrix} A & 0 \\ 0 & \tr{A}^{-1}\end{smallmatrix}\right)$ we can subtract 
multiplies of $a_1$ from
the other $a_i$ and continuing in this way (by the Euclidean algorithm) we can eventually arrange
that $a_j = 0$ for $2 \le j \le n$.  Similarly, acting by $\left(\begin{smallmatrix}
I & 0 \\ S & I \end{smallmatrix} \right)$ we can arrange that $|b_i|<|a_1|$ for $1 \le i \le n$.
Then using $\left( \begin{smallmatrix}0 & I \\ -I & 0\end{smallmatrix} \right)$ we can switch the
$a's$ and the $b's$.  Continuing in this way we can arrange that $b_j = 0$ for all $j$.  This
implies that $a_1$ is a unit, so we can adjust it to equal one. 
}

\section{Polarizations and positivity}\label{appendix-polarizations}

\subsection{}\label{subsec-polarizations} 
Let $(T,F)$ be a Deligne module.
Let $K=\QQ[F,V]=\QQ[F]\subset\End(T,F).$  Then $K$ is isomorphic to a product of distinct
CM fields (\cite{Howe} \S 4.6) corresponding to the distinct irreducible factors of the
characteristic polynomial of $\gamma$, and the canonical complex conjugation on $K$ 
(which we denote with a bar) interchanges $F$ and
$V.$ An element $\iota \in K$ is {\em purely imaginary} if $\bar{\iota^{}} = -\iota.$ 
A CM type $\Phi$ on $K$ is a subset $\Phi \subset \Hom(K,\CC)$ containing exactly one element
from each complex conjugate pair of homomorphisms $K \to \CC$.
Let $\Phi$ be a CM type on $\QQ[F]$.  A purely imaginary element $\iota\in K$ is 
{\em totally $\Phi$-positive imaginary} if $\phi(\iota)\in \CC$ is positive imaginary 
for all $\phi \in \Phi.$    A polarization $\omega$ of the Deligne module 
$(T,F)$ that is {\em positive with respect to the CM type $\Phi$} 
is defined to be an alternating bilinear form 
$\omega:T \times T \to \ZZ$ such that \begin{enumerate}
\item[(0)] $\omega(x,y) = - \omega(y,x)$ for all $x,y\in T$,
\item $\omega:(T\otimes \QQ) \times (T \otimes \QQ) \to \QQ$ is nondegenerate,
\item $\omega(Fx,y) = \omega(x,Vy)$ for all $x,y \in T$,
\end{enumerate}
as well as the following {\em positivity condition} (see \S \ref{subsec-viable}) \begin{enumerate} 
\item[(3)] the form $R(x,y) = \omega(x,\iota y)$ is symmetric and positive definite,
for some (and hence any) totally $\Phi$-positive imaginary element $iota$,
\end{enumerate}
in which case we will say that $(T,F,\B)$ is {\em a $\Phi$-positively polarized Deligne module}.

\begin{lem}\label{lem-unique-positive}
Let $(T,F)$ be a Deligne module that is simple over $\QQ$.  For any CM type $\Phi$ on $\QQ[F]$ there
exists a $\Phi$-positive polarization of $(T,F)$.  Conversely, suppose $(T,F)$ is a 
$\QQ$-simple Deligne module and suppose $\omega:T \times T \to \ZZ$ is a symplectic form 
such that $\omega(Fx,y) = \omega(x,Vy)$.  Then there
exists $\iota \in \QQ[F]$ so that the bilinear form 
\begin{equation}\label{eqn-R-positive}
R(x,y) = \omega(x,\iota y)\ \text{ 
is symmetric and positive definite.}\end{equation}
Moreover there exists a unique CM type $\Phi$ on $\QQ[F]$ so that
every such element $\iota$ 
is totally $\Phi$-positive imaginary.
\end{lem}
\begin{proof}
For the first part of the lemma, we may reduce to the case that $\QQ[F]\cong K$ is isomorphic to
a CM field of the same dimension as $T$ and we may assume that $T \subset \mathcal O_K$ is a
lattice that is preserved by $F$.  A polarization that is
positive with respect to $\Phi$ and compatible with complex conjugation is described in
\cite{ShimuraTaniyama} \S 6.2 and \cite{ShimuraCM} \S 6.2; see also
\cite{MilneToronto} Prop. 10.2 p. 335; we repeat the definition here.
Let $\alpha \in \mathcal O_K$ be totally $\Phi$-positive imaginary and set
\[ \B(x,y) = \tTr_{K/\QQ}(\alpha x.\overline{y}).\]
Then $\B:\mathcal O_K \times \mathcal O_K \to \ZZ$ is antisymmetric and the bilinear form
$R(x,y) = \B(x, \alpha y)$ is symmetric and positive definite.  We also remark that
\begin{equation}\label{eqn-omega-conjugation} \omega(\bar x, \bar y) = - \omega(x,y).
\end{equation}

Conversely, suppose $(T,F)$ is $\QQ$-simple and 
$\omega:T\times T \to \ZZ$ is alternating and nondegenerate over $\QQ$ with $\omega(Fx,y)
= \omega(x,Vy)$.
A choice of a $F$-cyclic vector gives an isomorphism of $\QQ[F]$-modules,
$T\otimes\QQ \cong \QQ[F]$.  Using this isomorphism, the symplectic form $\omega$ is given by
\[ \omega(x,y) = \tTr_{K/\QQ}(\alpha x \bar y)\]
for some element $\alpha \in K$ with $\bar \alpha = -\alpha$.  (In fact, the mapping
$x \mapsto \omega(x,1)$ is $\QQ$-linear, hence given by $\Tr_{K/\QQ}(\alpha x)$ for some
$\alpha \in K$ which is easily seen to satisfy $\bar \alpha = -\alpha$.  Therefore
$\omega(x,y) = \omega(\bar{y}x,1) = \Tr_{K/\QQ}(\alpha x \bar{y})$.)  
Let $\iota = \alpha$.  Then
\[ R(x,x) = \omega(x,\iota x) = \tTr_{K/\QQ}(\alpha \bar\alpha x \bar x) >0\]
which verifies the condition (\ref{eqn-R-positive}).
For any embedding $\phi:K \to \CC$ the image $\phi(\alpha)$ is purely imaginary.  It follows
that there is a unique choice $\phi$ from each pair of complex conjugate embeddings such
that $\phi(\alpha)$ is positive imaginary, and this defines a CM type $\Phi$ for $K$.  If
$\beta \in \QQ[F]$ is any other element such that $(x,y) \to \omega(x,\beta y)$ is
symmetric and positive definite then $\bar\beta = - \beta$ so $\phi(\beta)$ is purely imaginary
for every $\phi \in \Phi$.  Moreover,
\[ \omega(x,\beta x) = \sum_{\phi \in \Phi} \phi(\alpha \bar\beta x \bar x) + 
\overline{\phi}(\alpha \bar\beta x \bar x) = 2 \sum_{\phi \in\Phi} \phi(\alpha \bar\beta
x \bar x) > 0\]
for all $x \in \QQ[F]^{\times}$.  This implies that $\phi(\alpha) \phi(\bar\beta) =
- \phi(\alpha)\phi(\beta) >0$ for each $\phi \in \Phi$, hence $\phi(\beta)$ is 
also positive imaginary.

\end{proof}

\subsection{}\label{subsec-phi-varepsilon}
Following Serre (and Deligne \cite{Deligne} p.~242) it is possible to make a universal choice 
for the CM type associated to any Deligne module.  The embedding $\varepsilon:W(\bar k) \to 
\CC$ of (\ref{eqn-epsilon}) determines 
a $p$-adic valuation $\val_p$ on the algebraic closure $\overline{\QQ}$ of $\QQ$ in
$\CC$.  If $(T,F)$ is a Deligne module let $K = \QQ[F]$ as above and define 
$\Phi_{\varepsilon}$ to be the CM type on $K$ that consists 
of all homomorphisms $\phi:K \to \CC$ such that $\val_p(\phi(F))>0.$ 

\subsection{Howe's theorem}\label{subsec-HoweTheorem}
If $(T,F)$ is a Deligne module then the {\em dual} Deligne module
$(\widehat{T},\widehat{F})$ is defined by $\widehat{T} = \Hom(T,\ZZ)$
and $\widehat{F}(\phi)(x) = \phi(Vx)$ for all $\phi\in\widehat{T}.$
To give a $\Phi_{\varepsilon}$-positive polarization $\B$ on a Deligne module $(T,F)$ is therefore equivalent
(setting $\lambda(x)(y) = \B(x,y)$) to giving a homomorphism $\lambda:T \to \widehat{T}$ 
such that  $\lambda\otimes\QQ$ is an isomorphism, $\widehat{\lambda} = -\lambda$,
 $\lambda \circ F = \widehat{F} \circ \lambda$, and
$\widehat{\lambda\iota} = \lambda\iota$ is symmetric and negative definite (where $\iota
\in \QQ[F]$ is totally $\Phi_{\varepsilon}$ positive imaginary).

Let $A$ be an ordinary Abelian variety with associated Deligne module $(T_A,F_A)$ that is
determined by the embedding $\varepsilon$ of (\ref{eqn-epsilon}).
Let $B$ be the Abelian variety that is dual to $A.$  Then $B$ is ordinary and
(\cite{Howe} Prop. 4.5) there is a canonical ``duality" isomorphism
$(T_A,F_A) \cong (\widehat{T}_B,\widehat{F}_B)$.
Let $\omega:T_A \times T_A \to \ZZ$ be an alternating
bilinear form that satisfies conditions (1) and (2) of \S \ref{subsec-polarizations}.  It
induces an isomorphism $\lambda:(T_A\otimes \QQ,F)  \to (\widehat{T}_A\otimes\QQ, \widehat{F})$ 
and hence an isogeny $\lambda_A: A \to \widehat{A}$.  Then Howe proves (\cite{Howe}) that
$\omega$ is positive with respect to the CM type $\Phi_{\varepsilon}$ (that is, $\B$ is
a $\Phi_{\varepsilon}$-positive polarization of $(T_A,F_A)$) if and only if $\lambda_A$ 
is a polarization of the  Abelian variety $A$.

\begin{lem}\label{lem-Rosati}  Let $(T,F)$ be a Deligne module and let $\Phi$
be a CM type on $\QQ[F]$.  Let $\B$ be a $\Phi$-positive polarization.
Suppose $\beta \in \End_{\QQ}(T,F)$ is a self
isogeny that is fixed under the Rosati involution, that is, $\beta' = \beta$ where
$\B(\beta'x,y) = \B(x,\beta y)$ for all $x,y \in T\otimes\QQ$.  
Then there exists $\alpha \in \End(T,F)\otimes \overline{\QQ}$ such that
$\beta = \alpha'\alpha.$  If $\beta'=\beta$ and $\beta$ is positive definite then
the element $\alpha$ may be chosen to lie in $\End(T,F)\otimes\RR.$
If $\B_1,\B_2$ are two $\Phi$-positive polarizations of the same Deligne
module $(T,F)$ then there exists an $\RR$-isogeny $(T,F,\B_1) \to
(T,F,\B_2)$ with multiplier equal to 1  (cf.~\S\ref{sec-isogeny}).
\end{lem}

\begin{proof} (See also \cite{Kottwitz} p.~206.)
As indicated in \cite{Mumford} p.~220, the algebra $\End(T,F)\otimes\RR$ is
isomorphic to a product of matrix
algebras $M_{d\times d}(\CC)$ such that $\beta' = \tr{\bar\beta}.$  Then
$\beta'=\beta$ implies that $\beta$ is Hermitian so there exists a unitary matrix
$U \in M_{d \times d}(\CC)$ with $\beta = \tr{\bar{U}}DU$ where $D$ is a diagonal matrix
of real numbers.   Choose a square root $\sqrt{D} \in M_{d\times d}(\overline{\QQ})$
and set $\alpha = \sqrt{D} U \in \End(T,F)\otimes\overline{\QQ}$.  Then $\alpha'\alpha
= \tr{\bar U}DU = \beta$ as claimed.  Moreover, if $\beta$ is positive definite
then the entries of $D$ are positive real numbers so we may arrange that
$\sqrt{D} \in M_{d\times d}(\RR)$, so $\alpha \in \End(T,F)\times \RR$
as claimed.

For the ``moreover" part, let $\beta \in \End_{\QQ}(T,F)$ be the unique endomorphism 
so that $\B_2(x,y) = \B_1(\beta x,y)$.  Then
$\beta$ is fixed under the Rosati involution for the polarization $\B_1$ because
\[\B_1(\beta'x,y) = \B_1(x,\beta y) = - \B_1(\beta y, x) = -\B_2(y,x) =\B_2(x,y) = \B_1(\beta x,y).\]
Moreover, $\beta$ is positive definite for if $x \in T \otimes \RR$ is an eigenvector of $\beta$ with
eigenvalue $t$ then
\[tR_1(x,x) = R_1(\beta x,x) = \B_1(\beta x, \iota x) = \B_2(x,\iota x) = R_2(x,x) > 0\]
in the notation of \S \ref{subsec-polarizations}.  According to the first part of this
lemma, there exists $\alpha \in \End(T,F)\otimes \RR$ so that $\beta = \alpha'\alpha$, or
\[ \B_2(x,y) = \B_1(\alpha'\alpha x,y) = \B_1(\alpha x, \alpha y)\]
which says that $\alpha$ is an $\RR$-isogeny which takes $\B_1$ to $\B_2$ with multiplier equal to 1.
\end{proof}

\subsection{Viable elements}\label{subsec-viable}
Let $\gamma_0 \in \GSp_{2n}(\QQ)$ be a semisimple element with multiplier $c(\gamma) = q$,
whose characteristic polynomial is an ordinary Weil $q$-polynomial.  Choose a CM type $\Phi$
on the algebra $K = \QQ[\gamma_0]$.
Let us say that $\gamma_0$ is {\em viable} with respect to $\Phi$ if the pair $(\gamma_0,
\B_0)$ satisfies the positivity condition of \S \ref{subsec-polarizations}, namely, the form $R_0(x,y) = 
\B_0(x,\iota_0 y)$ is symmetric and positive definite, where $\iota_0\in K$ is a totally $\Phi$-positive
imaginary element and where $\B_0$ is the standard symplectic
form (\S \ref{subsec-preliminaries}).  Then $\gamma_0$ is $\Phi$-viable if and only if
 there exists a lattice $L$ so that $(L, \gamma_0, \B_0)$ is a $\Phi$-positively polarized Deligne module. 
Such a lattice $L$ may be taken, for example, to be the orbit  $L=\mathcal O_K \cdot x_0$ of 
an appropriate point $x_0 \in \QQ^{2n}$.
If $\gamma_0$ is also $q$-inversive then, by choosing the point $x_0$ so as to be fixed under
the standard involution $\tau_0$ of \S \ref{sec-involutions}, the lattice $L$ will be preserved by 
$\tau_0$.  Consequently, $(L, \gamma_0, \B_0, \tau_0)$ is a $\Phi$-positively polarized 
Deligne module with real structure.

\subsection{CM type of a conjugacy class}
Let $p(x)\in\QQ[x]$ be an ordinary Weil $q$-polynomial and let 
$\mathcal C(p)\subset \GSp_{2n}(\QQ)$ denote the
set of all rational, semisimple elements whose characteristic polynomial equals $p(x)$. 
It is the intersection of a $\GSp_{2n}(\overline{\QQ})$-conjugacy class with $\GSp_{2n}(\QQ)$. 
Let $\gamma, \gamma' \in \mathcal C(p)$.  Conjugation defines a unique isomorphism between
$\QQ[\gamma]$ and $\QQ[\gamma']$.  Therefore, a choice of CM type for $\QQ[\gamma]$
determines, in a canonical way, a CM type on $\QQ[\gamma']$, for every $\gamma' \in \mathcal 
C(p)$.  The distinct roots of $p(x)$ (that is, the distinct eigenvalues of $\gamma$)
occur in complex conjugate pairs, and a CM type for
$\QQ[\gamma]$ consists of a choice of one element from each such pair.  (If $\QQ[\gamma]$
is a field and $\beta\in\CC$ is a root of $p(x)$ then the corresponding embedding
$\phi:\QQ[\gamma] \to \CC$ is given by $\gamma \mapsto \beta$.)  In this way we may speak
of a choice of CM type for the conjugacy class $\mathcal C(p)$.

The set $\mathcal C(p)$ decomposes into finitely many classes 
under the equivalence relation of  $\GSp_{2n}(\RR)$-conjugacy.  By a slight abuse of terminology
we refer to these equivalence classes as $\GSp_{2n}(\RR)$-conjugacy classes in $\mathcal C(p)$.
Every $\GSp_{2n}(\RR)$-conjugacy class is a union of two $\Sp_{2n}(\RR)$-conjugacy classes,
and conjugation by $\tau_0$ exchanges them.  (If $h\in \GSp_{2n}(\RR)$ has 
multiplier $c\in\RR$ then $\tau_0h$ has multiplier $-c$.  If $c>0$ then
 conjugation by $\sqrt{c}^{-1}h \in \Sp_{2n}(\RR)$ agrees with that of $h$.) 
In the following proposition we identify
the viable elements in $\mathcal C(p)$.

\begin{prop} \label{prop-viable}  Let $p(x) \in \QQ[x]$ be an ordinary Weil $q$-polynomial.
Fix a CM type $\Phi$ on $\mathcal C(p)$. 
Among the $\Sp_{2n}(\RR)$-conjugacy classes in $\mathcal C(p)$ there is a unique one,
$\mathcal C^+(p)\subset \mathcal C(p)$ that contains  $\Phi$-viable elements, 
and for this class, every element $\gamma \in \mathcal C^+(p)$ is $\Phi$-viable.

\end{prop}
\begin{proof}  By Proposition \ref{prop-Qbar-conjugacy} there exist viable elements in
$\mathcal C(p)$.  Fix one such element $\gamma_0\in \mathcal C(p)$ and choose a lattice
$L_0 \subset\QQ^{2n}$ so that  $(L_0, \gamma_0, \B_0)$ is a polarized Deligne module.

Suppose that $\gamma \in \mathcal C(p)$ is viable and that  $(L, \gamma, \B_0)$ is a polarized Deligne module. 
We claim this implies that $\gamma, \gamma_0$ are $\Sp_{2n}(\RR)$ conjugate.   Since $\gamma, \gamma_0$
have the same characteristic polynomial, they are conjugate by some element $\phi\in\GL_{2n}(\QQ)$
where therefore induces an identification $\QQ[\gamma] \cong \QQ[\gamma_0]$.  Then
$(\phi^{-1}(L), \gamma_0, \phi^*(\B_0))$ is a polarized Deligne module.  Choose $c \in \QQ$, $c>0$
so that $c\phi^*(\B_0)$ takes integer values on $L_0$.  Then $c\phi^*(\B_0)$ is a second polarization
of the Deligne module $(L_0, \gamma_0)$.
By Lemma \ref{lem-Rosati} there is an
$\RR$-isogeny $\psi:(L_0,\gamma_0,\B_0) \to (L_0,\gamma_0, c\phi^*(\B_0))$ with multiplier equal to $1$,
which implies that $\psi^*\phi^*(c\B_0) = \B_0$.  Thus, $\phi\circ\psi \in \Sp_{2n}(\RR)$
and conjugation by $\phi\circ\psi$ takes $\gamma_0$ to $\gamma$.

Conversely we claim that every element $\gamma \in \mathcal C(p)$ that is $\Sp_{2n}(\RR)$-conjugate
to $\gamma_0$ is also viable.  Suppose $\gamma = h \gamma_0 h^{-1}$ where $h \in \Sp_{2n}(\RR)$.
Let $\iota_0 \in \QQ[\gamma_0]$ be a totally $\Phi$-positive imaginary element with associated
positive definite symmetric bilinear form $R_0(x,y) = \B_0(x, \iota_0 y)$.  It follows that
$\iota = h \iota_0 h^{-1} \in \QQ[\gamma]$ is totally $\Phi$-positive imaginary, and
\[ R(x,x) = \B_0(x, h \iota_0 h^{-1}x) =  \B_0(h^{-1}x, \iota_0 h^{-1} x) >0\]
for all $x$.
\end{proof}



\section{Involutions on the symplectic group} \label{sec-involutions} 
\subsection{}
Let $R$ be an integral domain and
let $\GSp_{2n}(R)=\GSp(R^{2n},\B_0)$ denote the symplectic group with
respect to the standard symplectic form $\B_0$, as in \S \ref{subsec-preliminaries}.
The {\em standard involution} $\tau_0:R^{2n} \to R^{2n}$ is
 $\tau_0 = \left( \begin{smallmatrix} -I_n & 0 \\ 0 & I_n \end{smallmatrix} \right).$ 
If $g \in \GSp_{2n}(R)$ let $\tilde g = \tau_0^{-1} g \tau_0$.
If $e_1,e_2,\cdots, e_{2n}$ denotes the standard basis of $R^{2n}.$
and if  $1 \le r \le n$ set
\begin{equation}\label{eqn-tau-cases}
\tau_r(e_i) = \begin{cases} e_i &\text{if}\ 1 \le i \le 2r\\
-e_i &\text{if}\ 2r+1\le i \le 2n.\end{cases}\end{equation}

\begin{prop}  \label{prop-classification}{\rm(}\cite{Hua, Dieudonne, Huppert}{\rm)}  
Let $R$ be an integral domain that contains $\frac{1}{2}$. 
Let $\tau: R^{2n}\to R^{2n}$ be an $R$-linear mapping such that $\tau^2=I,$ and suppose
that conjugation by $\tau$ preserves $\Sp_{2n}(R)\subset \GL_{2n}(R).$   
Then $\tau \in \GSp_{2n}(R)$ and its multiplier is $\pm 1.$  If it is
$-1$ then $\tau$ is $\Sp_{2n}(R)$-conjugate to $\tau_0.$
If its multiplier is $+1$ then $\tau$ is
$\Sp_{2n}(R)$-conjugate to $\tau_r$ for some $r$ with $1 \le r \le n.$
\end{prop}

\begin{proof}
First we claim that $\tau \in \GSp_{2n}(R).$
The matrix $M = \tr{\tau}J\tau$ is antisymmetric and (strongly) nondegenerate so it defines
a symplectic form and we claim that $\Sp(R^{2n},M) = \Sp(R^{2n},\B_0)$ for
\begin{align*}
g \in \Sp_{2n}(M) &\iff \tr{g} \tr{\tau} J \tau g = \tr{\tau}J\tau \\
&\iff \tau g \tau^{-1} \in \Sp_{2n}(\B_0) \iff g \in \Sp_{2n}(\B_0) \end{align*}
since conjugation by $\tau$ preserves $\Sp_{2n}(\B_0).$  Then $J^{-1}M:R^{2n}
\to R^{2n}$ is an intertwining operator, for if $g \in \Sp_{2n}(\B_0)$ then
\[
g^{-1}(J^{-1}M) = \left(g^{-1}J^{-1}\tr{g}^{-1}\right)\left(\tr{g}Mg\right)g^{-1}
= (J^{-1}M)g^{-1}.\]
By Schur's lemma there exists $c$ (in the algebraic closure of the
fraction field of $R$) such that $J^{-1}M = cI$ or $\tr{\tau} J \tau = cJ.$
Thus $\tau \in \GSp_{2n}(\B_0)$ has multiplier equal to $c$, and since $\tau$ is an
involution we have $c = \pm 1.$

Write $T=R^{2n}.$
The symplectic form $J$ is (strongly) non-degenerate so it induces an isomorphism
\begin{equation}\label{eqn-adjoint}
T \cong \Hom(T,R)\quad \text{say,}\quad x \mapsto x^{\sharp}.\end{equation}
Let $T_+, T_-$ be the $\pm 1$ eigenspaces of $\tau.$  Since $2^{-1} \in R$, any
$x\in T$ may be written
\[ x = \frac{x-\tau(x)}{2} + \frac{x+\tau(x)}{2} \in T_- + T_+\]
so $T = T_- \oplus T_+.$  Apply this splitting to equation (\ref{eqn-adjoint}) to find
\begin{equation}\label{eqn-decomposition}
 \Phi: T_- \oplus T_+ \longrightarrow \Hom(T_-,R) \oplus \Hom(T_+,R).\end{equation}

Let us consider the case $c=-1,$ that is, $\B(\tau x, \tau y) = -\B(x,y).$
It follows that $\Phi(x,y) =(y^{\sharp}, x^{\sharp}),$ hence $\dim(T_-) = \dim(T_+) = n$
and we obtain an isomorphism
$T_+ \cong \Hom(T_-,R).$  Let $t_1,t_2,\cdots,t_n$ be a basis of $T_-$ and let
$\lambda_1,\cdots,\lambda_n \in \Hom(T_-,R)$ be the dual basis.  Using $\Phi$,
the dual basis translates into a basis $t'_1,\cdots,t'_n$ of $T_+.$  With respect to
this basis $\left\{t_1,\cdots,t_n, t'_1,\cdots,t'_n\right\}$ the matrix of the symplectic
form is $J,$ and the matrix of $\tau$ is $\left( \begin{smallmatrix}
-I & 0 \\ 0 & I \end{smallmatrix}\right).$

Now suppose $c=+1$ so that $\B(\tau(x),\tau(y)) = \B(x,y).$  Let $\dim(T_+) = r.$
Then $T_-$ and $T_+$ are orthogonal under
the symplectic form, so equation (\ref{eqn-decomposition}) gives isomorphisms
$T_- \cong \Hom(T_-,R)$ and $T_+ \cong \Hom(T_+,R)$ which is to say that the restriction
of the symplectic form to each of these subspaces is non-degenerate.  In particular, $r$ is
even.
A symplectic basis $\left\{v_1,\cdots,v_{2r}\right\}$ for $T_+$ and a symplectic basis
$\left\{v_{2r+1},\cdots,v_{2n}\right\}$ for $T_-$ gives a symplectic basis
$\left\{v_1,\cdots,v_{2n}\right\}$ for $T$ for which $\tau$ is given by equation
(\ref{eqn-tau-cases}).
\end{proof}

\subsection{}
According to \cite{Hua}, the other involutions of the symplectic group arise either from
an involution of the underlying ring $R$ or from conjugation by an element
$\eta \in \GSp_{2n}(R)$ such that $\eta^2 = \lambda.I$ where $\lambda \in R^{\times}$.
If the ring $R$ contains both $2^{-1}$ and $u = 
\sqrt{\lambda}$ then the above argument shows that $\eta$ is $\Sp_{2n}(R)$-conjugate to the matrix
\[ \left(\begin{matrix} -uI_n & 0 \\ 0 & uI_n \end{matrix} \right).\]
The proposition fails if the ring $R$ does not contain $\frac{1}{2}$, in fact we have:

\begin{lem}\label{lem-Z-involutions}
Let $\B_0$ be the standard symplectic form on $\ZZ^{2n}$ and let $\tau\in \GSp_{2n}(\ZZ)$ be
an involution with multiplier equal to $-1$.  Then $\tau$ is $\Sp_{2n}(\ZZ)$ conjugate to an
element
\[ \left( \begin{matrix} I & S \\ 0 & -I \end{matrix} \right)\]
where $S$ is a symmetric matrix consisting of zeroes and ones which may be taken to be
one of the following:  if $\rank(S) = r$ is odd then $S = \left( \begin{smallmatrix}
I_r & 0 \\ 0 & 0 \end{smallmatrix} \right)=I_r \oplus 0_{n-r}$; if $r$ is even then either
$S = I_r \oplus 0_{n-r}$ or $S = H\oplus H \cdots \oplus H\oplus 0_{n-r}$ 
where $H = 
\left(\begin{smallmatrix} 0 & 1 \\ 1 & 0 \end{smallmatrix} \right)$ appears $r/2$ times
in the sum.
\end{lem}
\quash{
 which may be taken to be one of
the following symmetric bilinear forms: if $\text{rank}(S) = 2r+1$ is odd then
\begin{align*}
\tr{x}Sx &= x_1^2 + 2(x_2x_3 + \cdots + x_{2r}x_{2r+1})
\intertext{if $\text{rank}(S)=2r$ is even then there are two possibilities:  either}
\tr{x}Sx &= 2(x_1x_2 + x_3x_4 + \cdots + x_{2r-1}x_{2r})  \intertext{or }
\tr{x}Sx &= x_1^2 + x_2^2 + 2(x_3x_4 + \cdots + x_{2r-1}x_{2r}). 
\end{align*}
}

\begin{proof}
There exists a vector $v \in \ZZ^{2n}$ that is primitive and has $\tau(v) = v$.  (Choose a
rational vector $u$ so that $\tau(u) = u$, clear denominators to obtain an integral vector,
and divide by common divisors to obtain a primitive vector.)  
We claim there exists $g \in \Sp_{2n}(\ZZ)$ so that $gv = e_1 = (1,0,\cdots,0)$.  This is
a lemma of Siegel (see \cite{Freitag} Satz A5.4) but here is an outline.
 Let $v = (a_1,\cdots,a_n, b_1,\cdots, b_n)$.  Acting by permutations
we can arrange that $|a_1|<|a_j|$ for $2 \le j \le n$.  Acting by
$\left(\begin{smallmatrix} A & 0 \\ 0 & \tr{A}^{-1}\end{smallmatrix}\right)$ we can subtract 
multiplies of $a_1$ from
the other $a_i$ and continuing in this way (by the Euclidean algorithm) we can eventually arrange
that $a_j = 0$ for $2 \le j \le n$.  Similarly, acting by $\left(\begin{smallmatrix}
I & 0 \\ S & I \end{smallmatrix} \right)$ we can arrange that $|b_i|<|a_1|$ for $1 \le i \le n$.
Then using $\left( \begin{smallmatrix}0 & I \\ -I & 0\end{smallmatrix} \right)$ we can switch the
$a's$ and the $b's$.  Continuing in this way we can arrange that $b_j = 0$ for all $j$.  This
implies that $a_1$ is a unit, so we can adjust it to equal one. 

It follows that $\tau$ is $\Sp_{2n}(\ZZ)$ conjugate to a matrix 
$\left( \begin{smallmatrix} A & B \\ C & D \end{smallmatrix} \right)$ where
\[ A = \left( \begin{matrix} 1 & * \\ 0 & A_1 \end{matrix}\right),\ \
B = \left(\begin{matrix} * & * \\ * & B_1 \end{matrix} \right),\ \
C = \left(\begin{matrix} 0 & 0 \\ 0 & C_1 \end{matrix} \right),\ \
D = \left(\begin{matrix}-1 & 0 \\ * & D_1 \end{matrix} \right)\]
and where $\left( \begin{matrix} A_1 & B_1 \\ C_1 & D_1 \end{matrix}\right) \in \GSp_{2n-2}(\ZZ)$
is an involution with multiplier equal to $-1$.  By induction, the involution $\tau$ is therefore
conjugate to such an element where $A_1 = I$, $B_1$ is symmetric, $C_1 = 0$ and $D_1 = -I$.  The
condition $\tau^2 = I$ then implies that $A = I$, $D = -I$, $C = 0$ and $B$ is symmetric.
Conjugating $\tau$ by any element $\left(\begin{smallmatrix} I & T \\ 0 & I \end{smallmatrix}
\right) \in \Sp_{2n}(\ZZ)$ (where $T$ is symmetric) we see that $B$ can be modified by the addition
of an even number to any symmetric pair $(b_{ij}, b_{ji})$ of its entries.  
Therefore, we may take $B$ to consist of zeroes and ones.

The problem then reduces to describing the list of possible 
symmetric bilinear forms on a $\ZZ/(2)$ vector space $V$. 
It suffices to consider the case of maximal rank.  Certainly, $B = I$ is one such.  Let 
$\langle v,w \rangle = \tr{\!v}Bw$.  Suppose it happens that
 $\langle v,v \rangle = 0$ for all $v\in V$.  Choose $v,w$ so that $\langle v,w \rangle = 1$,
let $W_1$ be the span of $v,w$ and apply the same reasoning to $W_1^{\perp}$ to obtain
$B \cong H \oplus \cdots \oplus H$ (and $\dim(V)$ is even). On the other hand, 
if there exists $v \in V$ so that 
$\langle v,v \rangle = 1$ let $W_1$ be the span of $v$ and consider $W_1^{\perp}$.  
One checks that $W_1 \oplus H \cong I_{3}$ and more generally that
$I_{r}\oplus H \cong I_{r+1}$ if $r$ is odd.  Thus, if $\dim(V)$ is odd then
$B \cong I$.
\end{proof}

\subsection{Remark}\label{remark-UFD}  
A similar argument classifies involutions of $\Sp_{2n}(R)$ with multiplier equal to $-1$,
whenever $R$ is a Euclidean domain.

\begin{lem}\label{lem-two-splittings}
Let $K$ be a field of characteristic not equal to 2.  Let $V$ be $K$-vector space of
dimension $2n$.  Let $\B:V \times V \to R$ be a nondegenerate symplectic form.
Let $\tau\in \GSp(V,\B)$ be an involution with multiplier $-1$.
Suppose $V = V' \oplus V^{\pp}$ is a decomposition into Lagrangian subspaces that are
exchanged by $\tau$.  Then there exists a numbers $a_1,\cdots, a_n \in K^{\times}$
and there exists a symplectic basis $\{e'_1,\cdots,e'_n,
e^{\pp}_1,\cdots, e^{\pp}_n\}$ of $V$ so that $\{e'_1,\cdots,e'_n\}$ form a basis
of $V'$, so that $\{e^{\pp}_1,\cdots,e^{\pp}_n\}$ form a basis of $V^{\pp}$, and 
so that $\tau(e'_i) = a_ie^{\pp}_i$.  
\end{lem}
\begin{proof}
Any choice of basis $\{u'_1,\cdots,u'_n\}$ of $V'$ determines a dual basis
$\{u^{\pp}_1,\cdots, u^{\pp}_n\}$ of $V^{\pp}$ with respect to the nondegenerate
pairing $V' \times V^{\pp} \to K$ defined by $\B$.  The collection 
$\{u'_1,\cdots, u'_n, u^{\pp}_1,\cdots, u^{\pp}_n\}$ is thus a symplectic basis of
$V$ and it defines an isomorphism $\phi:V \to K^{2n}$ which takes the symplectic form
$\B$ to the standard symplectic form $J_0$ and it takes the Lagrangian spaces $V'$ and
$V^{\pp}$ to
the Lagrangian spaces $K^n \times \{0\}$ and $\{0\}\times K^n$ respectively.  
The matrix of $\tau$ with respect to this
basis is therefore $\left( \begin{smallmatrix} 0 & A \\ \tr{\!A}^{-1} & 0 \end{smallmatrix} \right)$.
The condition $\tau^2=I$ implies that $A$ is symmetric.  Let $g = \left(
\begin{smallmatrix} B & 0 \\ 0&\tr{B}^{-1} \end{smallmatrix} \right)$ be a symplectic change of basis
that preserves the decomposition $K^{n} \oplus K^{n}$.  This has the effect of changing
the matrix of $\tau$ by replacing $A$ with $BA\tr{B}$.  Thus, it is possible to choose the matrix
$B$ so that $B A \tr{B}$ is diagonal (and its diagonal entries are determined up to multiplication
by squares of elements in $K$). 
\end{proof}

\quash{
\subsection{}\label{sec-H}
Let $\GL^*_n$ denote the centralizer of $\tau_0$ or equivalently, the subgroup of
$\GSp_{2n}$ that is fixed under the involution defined by $\tau_0.$  It consists
of all elements $
 h = \left( \begin{smallmatrix}  X & 0 \\ 0 & Y
\end{smallmatrix} \right)$
where $X\tr{Y} = \mu I$ for some $\mu \ne 0.$
There is a short exact sequence
\[
1 \to \GL_n \to \GL^*_n \to \mathbb G_m \to 1\]
where $\GL_n$ is identified with its image under the {\em standard embedding}
$\delta:\GL_n\hookrightarrow\GSp_{2n}$ given by
$\delta(A) = \left(\begin{smallmatrix} A & 0 \\ 0 & {}^tA^{-1}\end{smallmatrix}\right).$
If
\begin{equation}\label{eqn-h}
h = \left(\begin{matrix} X & 0 \\ 0 & \mu\tr{\!X}^{-1} \end{matrix} \right) \in \GL^*_n
\end{equation} and if $\gamma = \left(
\begin{smallmatrix} A & B \\ C & D\end{smallmatrix}\right)\in \GSp_{2n}$ then conjugation by
$h$ is given by:
\begin{equation}\label{eqn-gln-action}
h \gamma h^{-1} = \left( \begin{matrix} XAX^{-1} & \frac{1}{\mu}XB\tr{\!X}\\
\mu\tr{\!X}^{-1}CX^{-1} & \tr{\!X}^{-1} D \tr{\!X} \end{matrix} \right)
\end{equation}
It follows that $\gamma$ is \qreal if and only if $h \gamma h^{-1}$ is \qreal.
We say that two elements of $\GSp_{2n}$ are $\GL_n$-conjugate if the conjugating element
lies in the image of $\delta.$  We will often identify $\GL_n$ with its image under
$\delta$ (and therefore suppress the $\delta$).
}

\section{Symplectic cohomology}\label{appendix-cohomology}
\subsection{Nonabelian cohomology}
Let $R$ be a commutative ring with $1$.  As in Appendix \ref{sec-involutions} the involution $\tau_0$ of
$R^n \times R^n$ is defined by $\tau_0(x,y) = (-x,y)$.
Let $\langle \tau_0 \rangle = \left\{ 1, \tau_0 \right\} \cong \ZZ/(2)$
denote the group generated by the involution $\tau_0$.  For $g \in \Sp(2n,R)$ 
let $\tilde g = \tau_0 g \tau_0^{-1}$.  This defines an action of the group 
$\langle\tau_0\rangle$ on $\Sp(2n,R)$.   Let $\Gamma \subset \Sp_{2n}(R)$ be a
subgroup that is preserved by this action (that is, $\widetilde{\Gamma} = \Gamma$).  
Recall that a 1-cocycle for this action
is a mapping $f:\langle\tau_0\rangle \to \Gamma$ such that $f(1) = I$ and $f(\tau_0) = g$
where $g \tilde g = I$.  We may write $f = f_g$ since the mapping $f$ is determined by the element $g$.
Then two cocycles $f_g, f_{g'}$ are cohomologous if there exists $h \in \Gamma$ such that
$g' = h^{-1} g \tilde{h}$ or equivalently, such that $g' = \tilde{h} g h^{-1}$.  The set of
cohomology classes is denoted 
\[H^1(\langle \tau_0 \rangle, \Gamma).\]

If $\tau \in \GSp_{2n}(R)$ is another involution (meaning that $\tau^2=I$) with multiplier 
equal to $-1$ then $g=\tau\tau_0$ defines a cocycle since $g \tilde{g} = 1$.  One easily checks the following.
\begin{prop} \label{prop-cohomology-involutions}
Let $\Gamma \subseteq \Sp_{2n}(R)$ be a subgroup that is normalized by $\tau_0$. The mapping $\tau \mapsto \tau\tau_0$ determines a one to one correspondence between the
set of $\Gamma$-conjugacy classes of involutions (i.e.~elements of order $2$),
$\tau \in \Gamma.\tau_0$  
and the cohomology set $H^1(\langle \tau_0 \rangle, \Gamma)$.
\end{prop}

\subsection{Lattices and level structures}\label{appendix-level} 
If $L \subset \QQ^{2n}$ is a lattice its {\em symplectic dual} is the lattice
\[ L^{\vee} = \left\{ x \in \QQ^{2n}\left| \B_0(x,y)\in \ZZ \text{ for all } y \in L \right. \right\}\]
where $\B_0$ is the standard symplectic form.
A lattice $L \subset \QQ^{2n}$ is {\em symplectic} if $L^{\vee} = L$.  
A lattice $L \subset \QQ^{2n}$ is {\em symplectic up to homothety} if there exists
$c \in \QQ^{\times}$ so that $L^{\vee} = cL$.  In this case the symplectic form $b = c\B_0$ is
integer valued and strongly nondegenerate on $L$.  
A lattice $L\subset \QQ^{2n}$ is {\em real} if it is preserved by the standard involution
$\tau_0$, in which case write $\tau_L = \tau_0|L$.

Fix $N \ge 1$ and let $\bar L = L/NL$.
 A level $N$ structure on a lattice $L$ is an isomorphism $\alpha:\bar L \to 
(\ZZ/N\ZZ)^{2n}$. A level $N$ structure $\alpha$ is compatible with an integer valued symplectic form
$b:L \times L \to \ZZ$ if $\alpha_*(b) = \bar\B_0$ is the reduction modulo $N$ of the
standard symplectic form $\B_0$.  
A level $N$ structure $\alpha:\bar L \to \bar L_0$ is {\em real} if it is
 compatible with the standard involution, that is, if
$\bar\tau_0 \alpha = \alpha \bar \tau_L:\bar L \to \bar L_0$.

\subsection{Ad\`elic lattices}\label{subsec-adelic-lattices}
Let $\AA_f=\prod'_{v<\infty} \QQ_v$ (restricted direct product) denote the finite ad\`eles of $\QQ$
and let $\widehat{\ZZ} = \prod_p \ZZ_p$.
A $\widehat{Z}$-lattice $\widehat{M}\subset \AA_f^{2n}$ is a product $\widehat{M} = \prod_{v<\infty} M_v$ of 
$\ZZ_v$-lattices $M_v \subset \QQ_v^{2n}$ with $M_v = (\ZZ_v)^{2n}$ for almost all finite places $v$.  
If $\widehat{M} = \prod_{v<\infty}M_v$ is a lattice, its symplectic dual is 
$\widehat{M}^{\vee} = \prod_{v<\infty} {M}^{\vee}_v$ where
\[ (M_v)^{\vee} = \left\{ x \in \QQ_v^{2n}|\
\B_0(x,y) \in \ZZ_v \ \text{ for all } y \in M_v \right\}.\]
The lattice $\widehat{M}$ is {\em symplectic up to homothety} if there exists $\mathfrak c
\in \AA_f^{\times}$ so that $\widehat{M}^{\vee} = \mathfrak c \widehat{M}$.  In this case, 
there exists $c\in\QQ^{\times}$ (unique, up to multiplication
by $\pm 1$) so that $\widehat{M}^{\vee}=c\widehat{M}$, and the alternating form $b=c\B_0$ takes
$\widehat{Z}$ values on $\widehat{M}$. 
A lattice $\widehat M$ is {\em real} if it is preserved by the standard involution $\tau_0$.

 A level $N$ structure on an ad\`elic lattice
$\widehat{M}$ is an isomorphism $\beta:\widehat{M}/N\widehat{M} \to (\ZZ/N\ZZ)^{2n}$.  It is
compatible with a $\widehat{Z}$-valued symplectic form $b:\widehat{M} \times \widehat{M} \to \widehat{Z}$
if $\beta_*(b) = \bar{\B}_0$ is the reduction modulo $N$ of the standard symplectic form.  It is
{\em real} if it commutes with the standard involution $\tau_0$.
The following statement is standard, see for example \cite{Platonov} Theorem 1.15:
\begin{lem}\label{lem-Platonov} Let $L\subset  \QQ^{2n}$ be a $\ZZ$-lattice and let
$L_v = L \otimes \ZZ_v$ for each finite place $v$.  Then\begin{itemize}
\item $L_v =  \ZZ_v^{2n}$ for almost all $v<\infty.$
\item $L = \bigcap_{v<\infty} (\QQ^{2n} \cap L_v).$
\item Given any collection of lattices $M_v \subset \QQ_v^{2n}$ such that $M_v = \ZZ_v^{2n}$
for almost all $v<\infty,$ there exists a unique $\ZZ$-lattice $M \subset \QQ^{2n}$ such that 
$M_v = M \otimes \ZZ_v$ for all $v<\infty.$\end{itemize}\end{lem}
This correspondence is clearly compatible with symplectic structures, real structures and level structures.

\subsection{The cohomology class of a symplectic lattice with ``real" structure}
Let $L\subset\QQ^{2n}$ be a lattice, symplectic up to homothety (say, $L^{\vee} = cL$ where $c\in\QQ$), and 
suppose that $L$ is preserved by the standard involution $\tau_0:\QQ^{2n} \to \QQ^{2n}$, in which case we
refer to $L$ as a ``real" lattice.
Let $\alpha:L/NL \to (\ZZ/N\ZZ)^{2n}$ be a level $N$ structure that is compatible with the involution (meaning that
$\alpha_*(\bar\tau) = \bar\tau_0$ is the standard involution, where $\tau = \tau_0|L$,
 and where the bar denotes reduction modulo $N$) and with the nondegenerate
symplectic form $b = c\B_0$ on $L$ (meaning that $\alpha_*(b) = \bar\omega_0$ is the standard symplectic
form on $(\ZZ/N\ZZ)^{2n}$).  By the strong approximation theorem, the mapping
\[ \Sp_{2n}(\ZZ) \to \Sp_{2n}(\ZZ/N\ZZ)\]
is surjective. Together with the symplectic basis theorem (Lemma \ref{lem-Darboux}) (and the
fact that $\alpha$ is compatible with $b=c\B_0$)
this implies that there exists $g \in \GSp_{2n}(\QQ)$ so that $(L,\alpha) = g.(L_0,\alpha_0)$, where
$L_0 = \ZZ^{2n}$ is the standard lattice with its standard level $N$ structure
$\alpha_0:L_0/NL_0 \to (\ZZ/N\ZZ)^{2n}$.  Both the lattice $L$ and the level structure
$\alpha$ are compatible with the involution which implies that 
$(L,\alpha) = g.(L_0,\alpha_0) =\tilde{g}.(L_0,\alpha_0)$ (where $\tilde g = \tau_0 g \tau_0^{-1}$).  
Therefore
\[ t = g^{-1} \tilde g \in K^0_N\subset\Sp_{2n}(\QQ)\]
is a cocycle (with multiplier equal to $1$) which lies in the principal congruence subgroup
\[ K^0_N = \ker\left(\Sp_{2n}(\ZZ) \to \Sp_{2n}(\ZZ/N\ZZ)\right).\]
Let $[(L,\alpha)] \in H^1(\langle\tau_0\rangle, K^0_N)$ denote the resulting cohomology class.

Similarly, an ad\`elic lattice $\widehat{L}$, symplectic up to homothety, and preserved by the involution
$\tau_0$, together with a level $N$ structure $\beta$, (compatible with the involution and 
with the corresponding symplectic form) determine a cohomology class $[(\widehat{L},\beta)] \in
H^1(\langle \tau_0 \rangle, \widehat{K}^0_N)$ where
\[\widehat{K}^0_N = \ker(\Sp_{2n}(\widehat{\ZZ}) \to \Sp_{2n}(\ZZ/N\ZZ)).\]
The following proposition is essentially the same as in \cite{Rohlfs}.

\begin{prop}\label{prop-cohomology-lattice}
The resulting cohomology classes $[(L,\alpha)]$ and $[(\widehat{L},\beta)]$ are well defined.
The mapping  $L \mapsto \widehat{L} = \prod_v(L\otimes\ZZ_v)$) determines a
one to one correspondence between
\begin{enumerate}
\item $GL_n^*(\QQ)$-orbits in the set of such pairs $(L,\alpha)$ that are symplectic up to homothety and
compatible with the involution (as above),
\item $\GL_n^*(\AA_f)$-orbits in the set of such pairs $(\widehat{L},\beta)$ that are symplectic
up to homothety and compatible with the involution (as above),
\item
elements of the cohomology set
\begin{equation}\label{eqn-hatH}
 H^1:= H^1(\langle\tau_0\rangle, K^0_N) \cong H^1(\langle\tau_0\rangle, \widehat{K}^0_N). 
\end{equation}
\end{enumerate}\end{prop}
\begin{proof}
 The cohomology class $[(L,\alpha)]$ is well defined for, suppose that 
$(L,\alpha) = h.(L_0,\alpha_0)$ for some $h \in \GSp_{2n}(\QQ)$.  Since $L$ is symplectic
up to homothety, the elements $g, h$ have the same multiplier hence $u = g^{-1}h 
\in K^0_N$.  Therefore the cocycle $h^{-1} \tilde h=u^{-1} (g^{-1}\tilde g) \tilde u$ is 
cohomologous to $g^{-1} \tilde g$.

Suppose $(L',\alpha')=g'.(L_0,\alpha_0)$ is another lattice with level $N$ structure, with the same
cohomology class.  Then $(g')^{-1}\tilde{g}' = u^{-1} (g^{-1} \tilde g) \tilde u$ for some
$u \in K^0_N$ which implies that the element $h = g'u^{-1} g^{-1}$ is fixed under the involution.
Hence $(L',\alpha') = h.(L,\alpha)$ is in the same $\GL_n^*(\QQ)$ orbit as $(L,\alpha)$.

Similar remarks apply to ad\`elic lattices. Finally, Lemma \ref{lem-Platonov} implies that the
cohomology sets (\ref{eqn-hatH}) may be canonically identified.
\end{proof}

\subsection{}
There is a simple relation between Propositions \ref{prop-cohomology-involutions}
and \ref{prop-cohomology-lattice} which identifies the cohomology class of
a lattice with a conjugacy class of involutions, as follows.  Suppose $(L,\alpha)$ is a 
``real" symplectic (up to homothety) lattice with a level $N$ structure.
Express $(L,\alpha) = g.(L_0,\alpha_0)$ for some $g \in \GSp_{2n}(\QQ)$.  Set
$\tau = g^{-1} \tau_0 g = h^{-1} \tau_0 h$ where $h \in \Sp_{2n}(\QQ)$.  Then
$\tau$ is an involution in $K^0_N.\tau_0$ because  $\tau\tau_0$ preserves $(L_0,\alpha_0)$,
and the cohomology class of $(L,\alpha)$ coincides with the cohomology class of $\tau$.
We remark, moreover, if the cohomology class $[(L,\alpha)]\in H^1(\langle \tau_0 \rangle, K^0_N)$ is
trivial then the lattice $L$ splits as a direct sum $L = L^{+} \oplus L^{-}$ of
$\pm 1$ eigenspaces of $\tau$ and $\alpha$ determines a principal level $N$
structure on each of the factors.

\quash{
\subsection{}\label{subsec-KN}  Fix $N \ge 1$.
Let $\widehat{Z} = \prod_p \ZZ_p$ and let $K^0_N, \widehat{K}^0_N, K_N,\widehat{K}_N$ denote
the principal level $N$ subgroups of $\Sp_{2n}(\ZZ)$,  $\Sp_{2n}(\widehat{\ZZ})$, $\GSp_{2n}(\ZZ)$
and $\GSp_{2n}(\widehat{\ZZ})$ respectively, that is,
\begin{alignat*}{2}
K^0_N &= \ker(\Sp_{2n}(\ZZ) \to \Sp_{2n}(\ZZ/N\ZZ)) \qquad& 
 \widehat{K}^0_N &= \ker(\Sp_{2n}(\widehat{Z}) \to \Sp_{2n}(\ZZ/N\ZZ))\\
{K}_N&=\ker(\GSp_{2n}({Z}) \to \GSp_{2n}(\ZZ/N\ZZ))
\qquad&
\widehat{K}_N &= \ker(\GSp_{2n}(\widehat{\ZZ}) \to \GSp_{2n}(\ZZ/N\ZZ)).\end{alignat*}
}

\begin{prop}\label{prop-trivial-cohomology} 
Let $R$ be an integral domain containing $\frac{1}{2}$.  Then $H^1(\langle \tau_0 \rangle,
\Sp_{2n}(R))$ is trivial.  If $2|N$ the mapping
$H^1(\langle\tau_0\rangle, K^0_N) \to H^1(\langle\tau_0\rangle,\Sp_{2n}(\ZZ))$ is trivial.
The cohomology sets 
\begin{equation}\label{eqn-p-adic-cohomology}
 H^1(\langle \tau_0\rangle,\Sp_{2n}(\ZZ)) 
\cong H^1(\langle\tau_0\rangle,\Sp_{2n}(\widehat{\ZZ})) \cong
H^1(\langle \tau_0 \rangle, \Sp_{2n}(\ZZ_2))\end{equation} 
are isomorphic and have order $(3n+1)/2$ if $n$ is odd, or $(3n+2)/2$ if $n$ is even.
\end{prop}
\begin{proof} 
By Proposition \ref{prop-cohomology-involutions} cohomology classes in $\Sp_{2n}(R)$
correspond to conjugacy classes of involutions with multiplier $-1$.  If 
$\frac{1}{2}\in R$ then Proposition \ref{prop-classification} says there is a unique
such, hence the cohomology is trivial.  For the second statement suppose $N \ge 2$ is even.
Suppose $\alpha \in\Sp_{2n}(\langle\tau_0\rangle, K^0_N)$ is a cocycle.  
Then $\alpha\tau_0$ is an involution which, by Lemma \ref{lem-Z-involutions} 
implies that there exists $h \in \Sp_{2n}({\ZZ})$ so that $h^{-1} \alpha \tilde h
= \left( \begin{smallmatrix} I & B \\ 0 & I \end{smallmatrix} \right)$ where $B$ is a
symmetric matrix of zeroes and ones.  It now suffices to show that $B = 0$ which 
follows from the fact that $\alpha \equiv I \mod 2$ and that $h^{-1} \tilde h \equiv I
\mod 2$, for if $h = \left( \begin{smallmatrix} a & b \\ c & d \end{smallmatrix} \right)$
then
\[h^{-1}\tilde h=
 I + 2\left(\begin{matrix} b\tr{c} & b \tr{a} \\ c \tr{d} & b \tr{c} \end{matrix}\right).\]
The cohomology set $H^1(\langle\tau_0\rangle,\Sp_{2n}(\ZZ))$ is finite because it 
may be identified with 
$\Sp_{2n}(\ZZ)$-conjugacy classes of involutions with multiplier $-1$ which, by Lemma
\ref{lem-Z-involutions} corresponds to $\GL_n(\ZZ)$-congruence classes of symmetric 
$n \times n$ matrices $B$ consisting of zeroes and ones. Summing over the possible
ranks $0 \le r \le n$ for the matrix $B$, with two possibilities when $r$ is even
and only one possibility when $r$ is odd gives $(3n+1)/2$ for $n$ odd and
$(3n+2)/2$ for $n$ even, cf.\cite{Lidl}.   Equation 
(\ref{eqn-p-adic-cohomology}) holds since $\frac{1}{2} \in \ZZ_p$ for $p$ odd.
\end{proof}

\section{Finiteness}\label{sec-finiteness}
Throughout this section, all polarizations are considered to be $\Phi_{\varepsilon}$-positive.
Recall the following result of A. Borel, \cite{Borel} (\S 9.11).
\begin{lem}\label{lem-finiteness}    
Let $G$ be a reductive algebraic group defined over $\QQ$ and let $\Gamma \subset G_{\QQ}$
be an arithmetic subgroup.  Let $G_{\QQ} \to \GL(V_{\QQ})$ be a rational representation of 
$G$ on some finite dimensional rational vector space.  Let $L \subset V_{\QQ}$ be a lattice that is
stable under $\Gamma$.  Let $v_0 \in V$ and suppose that the orbit $G_{\CC}.v_0$ is closed 
in $V_{\CC} = V_{\QQ} \otimes\CC$. Then $L \cap G_{\CC}.v_0$ consists of a finite number 
of orbits of $\Gamma$.
\end{lem}
\subsection{}
As in \S \ref{sec-ordinary}, let $\FF_q$ be a finite field of characteristic $p >0$, 
fix $N \ge 1$ not divisible by $p$ and let $n \ge 1$. We refer to \S 
\ref{appendix-level} for the definition of a level $N$ structure.  Recall the statement
of Theorem \ref{prop-finite-isomorphism}:
There are finitely many isomorphism classes of principally ($\Phi_{\varepsilon}$-positively)
polarized Deligne modules of rank $2n$ over $\FF_q$ with real structure and with 
principal level $N$ structure. 

\subsection{Proof of Theorem \ref{prop-finite-isomorphism}}
 It follows from Proposition \ref{prop-Qbar-conjugacy} that there are finitely many
$\overline{\QQ}$-isogeny classes of polarized Deligne modules with real structure.  Moreover, it is
easy to see that each isomorphism class (of principally polarized Deligne modules with real
structure) contains at most finitely many level $N$ structures.  So, for simplicity, we may omit
the level structure, and it
suffices to show that each $\overline{\QQ}$-isogeny class contains at most finitely many isomorphism
classes of principally polarized modules.  Therefore, let us fix a principally polarized
Deligne module with real structure, $(T,F,\B,\tau)$.  Using Lemma \ref{lem-Darboux} we may 
assume that $T = T_0 = \ZZ^{2n}$ is the standard lattice and that
$\B = \B_0$ is the standard symplectic form.  Using these coordinates the endomorphism
$F$ becomes an integral element $\gamma_0 \in \GSp_{2n}(\QQ) \cap M_{2n \times 2n}(\ZZ)$
and the involution $\tau$ becomes an element $\eta_0 \in \GSp_{2n}(\ZZ)$ with multiplier
equal to $-1$.  These elements have the following properties (from Lemma
\ref{subsec-real-standard-form}):
\begin{enumerate}
\item[(a)]  The eigenvalues of $\gamma_0$ are Weil $q$-numbers and
\item[(b)] $\eta_0 \gamma_0 \eta_0^{-1} = q \gamma_0^{-1}$.
\end{enumerate}

  The group $G = \Sp_{2n}$ acts on the vector space
\[ V = M_{2n \times 2n} \times M_{2n \times 2n}\]
by $g.(\gamma,\eta) = (g\gamma g^{-1}, g\eta g^{-1})$.  Let $\Gamma = \Sp_{2n}(\ZZ)$ be the arithmetic
subgroup that preserves the lattice
\[ L = M_{2n \times 2n}(\ZZ) \times M_{2n \times 2n}(\ZZ)\]
of integral elements.  It also preserves the set of pairs $(\gamma,\eta)$ that satisfy the
above conditions (a) and (b). 
Let $v_0 = (\gamma_0,\eta_0)$.   We claim \begin{enumerate}
\item the orbit $G_{\CC}.v_0$ is closed in $V_{\CC}$, and
\item there is a natural injection from \begin{enumerate}
\item the set of isomorphism classes of principally polarized
Abelian varieties with real structure within the $\overline{\QQ}$-isogeny class of
$(T_0, \gamma_0, \B_0, \eta_0)$ to
\item the set of $\Gamma$-orbits in $L \cap G_{\overline{\QQ}}.v_0$.\end{enumerate}
\quash{
\item isomorphism classes of principally polarized Abelian varieties with real structure within the 
$\overline{\QQ}$-conjugacy class of $(T_0, \gamma_0, \B_0, \eta_0)$ are in one to one correspondence 
with $\Gamma$-orbits in $L \cap G_{\overline{\QQ}}.v_0$.
}
\end{enumerate} 
Using claim (1) we may apply Borel's theorem and conclude that there are finitely 
many $\Gamma$ orbits in $L \cap
G_{\overline{\QQ}}.v_0$ which implies, by claim (2) that there are finitely many isomorphism 
classes, thus proving Theorem \ref{prop-finite-isomorphism}.

\medskip
\paragraph{\em Proof of claim (2).}
Consider a second principally polarized Deligne module with real structure within the same 
$\overline{\QQ}$-isogeny class.  As above, using Lemma \ref{lem-Darboux} may assume it to 
be of the form $(T_0, \gamma_1, \B_0, \eta_1)$ where $\gamma_1 \in \GSp_{2n}(\QQ) 
\cap M_{2n \times 2n}(\ZZ)$ and where $\eta_1 \in \GSp_{2n}(\ZZ)$ is an involution with 
multiplier equal to $-1$.  A $\overline{\QQ}$-isogeny between these two Deligne modules 
is an element $X \in \GSp_{2n}(\overline{\QQ})$ such that
$\gamma_1 = X \gamma_0 X^{-1}$ and $\eta_1 = X \eta_0 X^{-1}$.   In particular this means that
the pair $(\gamma_1,\eta_1)$ is in the orbit $\GSp_{2n}({\overline{\QQ}}).v_0$, which 
coincides with the orbit $G_{\overline{\QQ}}.v_0 = \Sp_{2n}(\overline{\QQ}).v_0$.   
Moreover, such an isogeny $X$ is an isomorphism (of principally polarized Deligne modules with real 
structure) if and only if $X$ and $X^{-1}$ preserve the lattice $T_0$ and the symplectic form $\B_0$, which is to say that $X \in \Gamma$.

We remark that the mapping from (2a) to (2b) above is not necessarily surjective for the
following reason.  The element $\gamma_0$ is {\em viable} (see \S \ref{subsec-viable}),
that is, it satisfies the ``positivity" condition of \S \ref{subsec-polarizations},
because it comes from a polarized Abelian variety.  However,  if $(\gamma,\eta) 
\in L \cap G_{\overline{\QQ}}.v_0$ is arbitrary then $\gamma$ may fail to be viable.
\medskip

\paragraph{\em Proof of claim (1).}
Since  $\gamma_0$ and $\eta_0$ are both semisimple, the conjugacy class
\[ (G_{\CC}.\gamma_0) \times (G_{\CC}.\eta_0) \subset M_{2n\times 2n}(\CC) 
\times M_{2n \times 2n}(\CC)\]
is closed (\cite{Humphreys2} \S 18.2).  We claim that the orbit $G_{\CC}.v_0$ 
coincides with the closed subset
\[S = \left\{ (\gamma,\tau) \in (G_{\CC}.\gamma_0) \times (G_{\CC}.\tau_0)|\ 
\tau \gamma \tau^{-1} = q \gamma^{-1} \right\}.\]
Clearly, $G_{\CC}.v_0\subset S$. If $(\gamma,\eta) \in (G_{\CC}.\gamma_0) \times (G_{\CC}.\eta_0)$ 
lies in the subset $S$ then by Proposition \ref{prop-classification}, conjugating by an 
element of $G_{\CC}$ if necessary, we may arrange that $\eta = \tau_0$ is the standard 
involution.  Consequently,
$\tau_0 \gamma \tau_0^{-1} = q \gamma^{-1}$, which is to say that $\gamma$ is $q$-inversive.  
By assumption it is also $G_{\CC}$-conjugate to $\gamma_0$.  According to 
Proposition \ref{prop-q-conjugacy}, over the
complex numbers there is a unique (up to reordering of the coordinates) standard form, and 
every $q$-inversive element $\gamma \in
G_{\CC}.\gamma_0$ is $\delta(\GL_n(\CC))$-conjugate to it.  Thus there exists $g \in \delta(\GL_n(\CC))$ 
so that $(g \gamma g^{-1}, g \tau_0 g^{-1}) = (\gamma_0, \tau_0)$.  In summary, the element 
$(\gamma,\eta)$ lies in the $G_{\CC}$-orbit of $(\gamma_0,\tau_0)$.  
This concludes the proof of Theorem \ref{prop-finite-isomorphism}.\qed

\subsection{The case $n=1$}
Fix $q = p^m$ and let $\FF_q$ denote the finite field with $q$ elements.  According to 
Proposition \ref{prop-Qbar-conjugacy} the set of $\overline{\QQ}$-isogeny classes of Deligne modules $(T,F)$ 
of rank 2, over $\FF_q$ is determined by a quadratic ordinary Weil $q$-number $\pi$, which we now fix. 
This means that $\pi$ satisfies an equation
\[ \pi^2 + B\pi + q = 0\]
where $p \nmid B.$ Let $D = B^2-4q$.  Then $D \equiv 0, 1 \mod 4$ and $-4q < D < 0$.
The pair $\{\pi, \bar\pi\}$ determines $D$ and vice versa.

Isomorphism classes of polarized Deligne modules with real structure fall into orbits that are identified by
certain cohomology classes as described in Proposition \ref{prop-cohomology-lattice} or equivalently
by integral conjugacy classes of involutions as described in Proposition \ref{prop-cohomology-involutions}. 
For $n=1$ there are two involutions (see Lemma \ref{lem-Z-involutions}) to consider, namely
\[ \tau_0 = \left(\begin{matrix} -1 & 0 \\ 0 & 1 \end{matrix}\right) \ \text{ and }\
\tau_1 = \left(\begin{matrix} -1 & 0 \\ 1 & 1 \end{matrix} \right).\] 

\begin{prop}\label{prop-case-n-equals-one}
Over the finite field $\FF_q$, the number of (real isomorphism classes of)
principally polarized Deligne modules $(T,F,\lambda,\eta)$ 
with real structure and rank 2, such that the eigenvalues of $F$ are $\{\pi, \bar\pi\}$, which correspond
to the cohomology class of the standard involution $\tau_0$ is:
\begin{equation*}
\begin{cases} \sigma_0(-D/4)&\text{if}\ D \equiv 0 \mod 4\\
0 &{} \text{otherwise} \end{cases}\end{equation*}
where $\sigma_0(m)$ denotes the number of positive divisors of $m>0$.
The number of isomorphism classes which correspond to the cohomology class of  $\tau_1$ is:
\begin{equation*}
\begin{cases}\sigma_0(-D)&\text{if}\ D \equiv 1 \mod 4\\
\sigma'_0(-D/4) &\text{if}\ D \equiv 0 \mod 4
 \end{cases}\end{equation*}
where $\sigma'_0(m)$ denotes the
number of ordered factorizations $m = uv$ such that $u,v>0$ have the same parity.
\end{prop}
\begin{proof}
According to Proposition \ref{prop-Qbar-conjugacy}
 the isomorphism classes of principally polarized Deligne modules with real structures correspond 
to $q$-inversive pairs $(\gamma,\eta)$ where the eigenvalues of $\gamma \in \GL_2(\ZZ)$ are 
$\pi$ and $\bar\pi$.  For the involution $\tau_0$ , the pair $(\gamma,\tau_0)$ is $q$-inversive if 
$\gamma = \left(\begin{smallmatrix}
a & b \\ c & a \end{smallmatrix}\right)$ and $\det(\gamma)=q$.  This implies that 
$a = -B/2$, so $B$ is even and $D \equiv 0 \mod 4$. Then $bc = a^2-q = D/4$ has a 
unique solution for every (signed) divisor $b$ of $D/4$.  Half of these will be viable 
(see \S \ref{subsec-viable}) so the number of solutions is equal to the number of 
positive divisors of $-D/4$.

For the involution $\tau_1$, the pair $(\gamma,\tau_1)$ is $q$-inversive if $\gamma = \left(\begin{smallmatrix}
a &b \\ c & a-b \end{smallmatrix} \right)$.  This implies that $ D = B^2-4q = b(b+4c)$.
Let us first consider the case that $b$ is odd or equivalently, that $D \equiv 1 \mod 4$.   For every divisor 
$b|D$ we can solve for an integer value of $c$ so we conclude that the number of viable solutions in this case is equal to $\sigma_0(-D)$.  
Next, suppose that $b$ is even, say, $b = 2b'$. Then $D$ is divisible by $4$, say, 
$D = 4D'$ and $D' = b'(b'+2c)$ is an ordered factorization of $D'$ with factors of the same parity.  So in this case the number of viable solutions is $\sigma'_0(-D/4)$.
\end{proof}

\subsection{}
For any totally positive imaginary integer $\alpha \in L=\QQ(\pi)$ the bilinear form
 $\omega(x,y) = \Tr_{L/\QQ}(\alpha x \bar y)$ is symplectic.  If $\Lambda \subset L$ is a lattice
then $\alpha$ may be chosen so that the form $\omega$ takes integer values on $\Lambda$.  Modifying
$\Lambda$ by a homothety if necessary, it can also be arranged that $\omega$ is a principal polarization,
hence $(\Lambda,\pi,\omega)$ is a principally polarized Deligne module.
 If complex conjugation on $L=\QQ(\pi)$ preserves $\Lambda$ then it defines a real structure 
on this Deligne module. 
\begin{prop}
The set of isomorphism classes of principally polarized Deligne modules (of rank $2$) with real structure
and with eigenvalues $\{\pi, \bar\pi\}$ may be identified with the set of homothety classes of lattices
$\Lambda \subset \QQ(\pi)$ that are preserved by complex conjugation and by multiplication by $\pi$.
\end{prop}
\begin{proof}
The most natural proof, which involves considerable checking, provides a map back 
from lattices $\Lambda$ to Deligne modules:  use the CM type of \S \ref{subsec-polarizations}
which determines an embedding $\Phi_{\varepsilon}:\QQ(\pi) \to \CC$, then realize the elliptic 
curve $\CC/\Phi_{\varepsilon}(\Lambda)$ as the complex
points of the canonical lift of an ordinary elliptic curve over $\FF_q$ whose associated Deligne module 
is $(\Lambda,\pi)$.  Then check that complex conjugation is compatible with these constructions.

A simpler but less illuminating proof is simply to count the number of homothety classes of lattices and to see that this number coincides with the number in Proposition \ref{prop-case-n-equals-one}.
Fix $\pi, B, D$ as above and suppose
first, that $B$ is even, say $B=2B'$, hence $D \equiv 0 \mod 4$, say $D = 4D'$, and $\pi = -B' + \sqrt{D'}$.
If a lattice $\Lambda \subset \QQ(\sqrt{D'})$ is preserved by complex conjugation then there are 
two possibilities up to homothety: either it has a basis consisting of $\{1, s\sqrt{D'}\}$ 
 or it has a basis consisting of $\{ 1, \frac{1}{2}+s\sqrt{D'}\}$ (for some $s \in \QQ$).  The matrix of
$\pi$ with respect to these bases is
\[
\left( \begin{matrix} -B' & sD' \\ 1/s & -B' \end{matrix} \right)\ \text{ or }\
\left( \begin{matrix} -B' -1/2s & sD' - 1/4s \\ 1/s & -B' + 1/2s \end{matrix} \right) \]
respectively, which must be integral.  In the first case this implies that $t = 1/s \in \ZZ$ divides $D'$.  
In the second case it implies that $t = 1/2s$ is integral and if we write $D' = tu$ then $t-u$ is even.

Similarly, if $B$ is odd then $D \equiv 1 \mod 4$.  In this case the matrix for $\pi$ with respect to any
lattice spanned by $\{ 1, s\sqrt{D}\}$ is never integral.  The matrix for $\pi$ with respect to a lattice
spanned by $\{ 1, \frac{1}{2} + s \sqrt{D}\}$ is
\[ \left( \begin{matrix}
\frac{1}{2}(-B-\frac{1}{s}) & \frac{1}{4} (sD-\frac{1}{s}) \\
\frac{1}{s} & \frac{1}{2}(-B+\frac{1}{s})\end{matrix}\right)\]
This implies that $t = 1/s$ is integral and odd.  Writing $D = tu$ we also require that 
$(u-t)/4$ is integral.  However, since $D \equiv 1 \mod 4$ this integrality condition holds for
any factorization $D = tu$.  So we obtain a lattice for every divisor $t |D$.
\end{proof}

\section{Proof of Proposition \ref{prop-Dieudonne-module}}\label{appendix-Dieudonne}
\subsection{}  Let $k = \FF_{q}$ with $q = p^a$ be the finite field with $p^a$ elements.
Let $\sigma:W(\bar k) \to W(\bar k)$ be the lift of the Frobenius
mapping $\sigma:\bar k \to \bar k$, $\sigma(x) = x^p.$  Recall that
a Dieudonn\'e module $M$ over $W(k)$ is a finite dimensional free $W(k)$ module with
endomorphisms $\mathcal F,\mathcal V:M \to M$ such that $\mathcal F$ is $\sigma$-linear and
$\mathcal{FV} = \mathcal{VF} = p.$  
Let $A/k$ be an ordinary Abelian variety with Deligne module $(T,F)$, with its decomposition
$T \otimes \ZZ_p = T' \oplus T^{\pp}$.  The associated finite group scheme $A[p^r] = \ker(\cdot p^r)$ 
 decomposes similarly
into a sum $A'[p^r] \oplus A^{\pp}[p^r]$ of an \'etale-local scheme and a local-\'etale scheme, with
a corresponding decomposition of the associated $p$-divisible group, $A[p^{\infty}] = A' \oplus
A^{\pp}$. 

\begin{lem}  The Dieudonn\'e module $M(A')$ of the $p$-divisible group $A'$ may be identified as follows,
\[ M(A') \cong (T' \otimes W(\bar k))^{\Gal}\]
where the Galois action is determined by $\pi.(t'\otimes w) = F(t')\otimes \sigma^a(w)$ and where
$\mathcal F(t'\otimes w) = pt'\otimes \sigma(w)$.
\end{lem}
\begin{proof}
Over $W(\bar k)$ the finite \'etale group scheme $A'[p^r]$ becomes constant and canonically isomorphic to
$p^{-r}T'_A/T'_A$ so its covariant Dieudonn\'e module over $W(\bar k)$ is the following:
\begin{equation}\label{eqn-Mtensor}
 \overline{M}(A'[p^r]) = (p^{-r}T'_A/T'_A)\otimes_{\ZZ}W(\bar k) \cong (T'_A/p^rT'_A)\otimes_{\ZZ} W(\bar k)\end{equation}
with $\mathcal F(t'\otimes w) = pt'\otimes \sigma(w)$, see \cite{Demazure} p. 68.  Then (see
\cite{Demazure} p. 71 or \cite{OortChai} \S B.3.5.9, p. 350), 
\begin{equation}\label{eqn-Mlimit}
\overline{M}(A') = \underset{\longleftarrow}{\lim}\overline{M}(A'[p^r]).\end{equation}
Tensoring the short exact sequence 
\[\begin{CD}
0 @>>> T_A' @>{\cdot p^r}>>T_A' @>>> T_A'/p^rT_A' @>>> 0 \end{CD}\]
with $W(\bar k)$ gives a sequence, which turns out to be exact, so we obtain a canonical isomorphism
\[ \left(T'_A/p^rT'_A\right)\otimes_{\ZZ}W(\bar k) \cong T'_A \otimes W(\bar k)/p^r(T'_A\otimes W(\bar k)).\]
Equation (\ref{eqn-Mlimit}) therefore gives
\begin{align*}
M(A') &=\left(\underset{\longleftarrow}{\lim}(T_A'/p^rT_A')\otimes W(\bar k)\right)^{\Gal}\\
&\cong\left(\underset{\longleftarrow}{\lim}\left(T'_A\otimes W(\bar k)/p^r (T'_A \otimes W(\bar k)\right)\right)^{\Gal}\\
&\cong \left(T'_A \otimes W(\bar k)\right)^{\Gal}.\qedhere\end{align*}
\end{proof}

\begin{lem}
The Dieudonn\'e module corresponding to the $p$-divisible group $A^{\pp}$ may be identified as follows,
\[ M(A^{\pp}) \cong (T^{\pp}\otimes W(\bar k))^{\Gal}\]
where the Galois action is given by $\pi(t^{\pp}\otimes w) = q^{-1}F(t^{\pp})\otimes \sigma^a(w)$ and where
$\mathcal F(t^{\pp}\otimes w) = t^{\pp} \otimes \sigma(w)$.
\end{lem}
\begin{proof}
 Let $B$ denote the ordinary Abelian variety that is dual to $A$ with Deligne module $(T_B,F_B)$
and corresponding $p$-divisible groups
$B'$, $B^{\pp}$.  Then $B'$ is dual to $A^{\pp}$ (and vice versa), hence it follows from the preceding Lemma 
(see also \cite{OortChai} \S B.3.5.9, \cite{Demazure} p. 72 and \cite{Howe} Prop. 4.5) that:
\renewcommand{\sc}{\scriptstyle}
\begin{align*}
\overline{M}(B') &= T'_B \otimes_{\ZZ_p} W(\bar k) &{ \begin{matrix}
\sc{\pi(t'\otimes w) =F_B(t')\otimes \sigma^a(w)}\\
\sc{\mathcal F(t'\otimes w) =pt'\otimes \sigma(w)} 
\end{matrix}}\\
\overline{M}(A^{\pp}) &= \hHom_{W(\bar k)}(\overline{M}(B'), W(\bar k) )& \begin{matrix}
\scriptstyle{\pi_A\psi(m) = \sigma^a\psi(\pi_B^{-1}(m))}\\
\scriptstyle{\mathcal F \psi(m) = \sigma \psi(\mathcal V(m))}
\end{matrix}\\
T'_B &= \hHom_{\ZZ_p}(T^{\pp}_A, \ZZ_p) &
\scriptstyle{F_B\phi(t') = \phi V_A(t')}
\end{align*}
From this, we calculate that the isomorphism 
\[\Psi: T^{\pp}_A \otimes W(\bar k) \to \hHom_{W(\bar k)}\left(\hHom_{\ZZ_p}(T^{\pp}_A,\ZZ_p)
\otimes W(\bar k), W(\bar k)\right) =\overline{M}(A^{\pp})\]
defined by 
\[ \Psi_{t^{\pp}\otimes w}(\phi \otimes u) = \phi(t^{\pp}).wu\]
(for $t^{\pp} \in T^{\pp}_A$, for $\phi \in \Hom(T^{\pp}_A, \ZZ_p)$ and for $w,u \in W(\bar k)$) satisfies:
\begin{align*}
\left(\pi.\Psi_{t^{\pp}\otimes w}\right)(\phi \otimes u) &= \sigma^a \Psi_{t^{\pp}\otimes w}(\pi_B^{-1}(\phi \otimes u))\\
&= \sigma^a \Psi_{t^{\pp}\otimes w}(F_B^{-1} \phi \otimes \sigma^{-a} u)\\
&= \sigma^a\left( (F_B^{-1}\phi)(t^{\pp}).w.\sigma^{-a} u\right)\\
&= \phi(V_A^{-1}(t^{\pp})).\sigma^a(w).u)\\
&=\left(\Psi_{V_A^{-1}t^{\pp}\otimes \sigma^a(w)}\right)(\phi \otimes u)
\end{align*}
Therefore $\pi(t^{\pp}\otimes w) = V_A^{-1}(t^{\pp})\otimes \sigma^a(w)$, which proves equation
(\ref{eqn-Galois-action}).  Similarly
\[ \left(\mathcal F.\Psi_{t^{\pp}\otimes w}\right)(\phi \otimes u) = 
\Psi_{t^{\pp}\otimes \sigma(w)}(\phi \otimes u)\]
hence $\mathcal F(t^{\pp}\otimes w) = t^{\pp}\otimes \sigma(w)$, which is equation (\ref{eqn-curlyF}).
Finally, $M(A^{\pp}) = \left(\overline{M}(A^{\pp})\right)^{\Gal}$.
\end{proof}

\begin{lem}\label{lem-justification1}  Let $(T,F)$ be a Deligne module of rank $2n$ with 
its decomposition $T\otimes \ZZ_p = T' \oplus T^{\pp}$.  Then the Dieudonn\'e module 
\[M(T) = (T'\otimes W(\bar k))^{\Gal} \oplus (T^{\pp}\otimes W(\bar k))^{\Gal}\]
is a free module over $W(k)$ of rank $2n$.
\end{lem}

\begin{proof}
  Choose a basis
\[\Phi:T\otimes \ZZ_p = T' \oplus T^{\prime\prime} \to \ZZ_p^{n}\oplus \ZZ_p^n\] 
and let $\gamma = \gamma' \oplus \gamma^{\prime\prime} = \Phi F \Phi^{-1}$.  This induces an isomorphism
\[ (T'\otimes W(\bar k))^{\Gal} \cong \mathcal J(\gamma'):= \left\{ b \in W(\bar k)^n|\ \gamma'b = \sigma^{-a}b\right\}.\]
It suffices to show that this is a free module over $W(k)$ of rank $n$.  
We must find a matrix $B' \in \GL_n(W(\bar k))$ so that
$\gamma'B' = \sigma^{-a}B'$ whose columns span $\mathcal J(\gamma').$
Let $L = (\ZZ_p)^n$. For each $m = 1, 2, \cdots$ choose an extension $k_m$  of $k$ so that
\begin{enumerate}
\item $(\gamma')^{f_m} \equiv I$ on $L/q^{m}L$.  
\item $(\sigma^{-a})^{f_m} = I$ on $k_m$ (hence, also on $W(k_m)$.)
\end{enumerate}where $f_m= [k_m:k]$ is the degree of $k_m$ over $k$.
Such an extension exists:  since $L/q^mL$ is finite and $\gamma$ is invertible, there exists an
integer  $f_m\ge 1$ so that $(\gamma')^{f_m}$ acts as the identity on $L/q^m L$.  Let $k_m$ be the unique extension of $k$ with degree $f_m$.  Then elements of $k_m$ are fixed by $\sigma^{-af_m}$.  
As usual we denote by $W(k_m)$ the ring of Witt vectors
and by $K(k_m))$ its fraction field. We identify $L \otimes W(k_m) \cong W(k_m)^n$. 
Define $S_m:W(k_m)^n \to W(k_m)^n$ by
\[ S_m(x) = x + (\gamma')^{-1} \sigma^{-a}(x) + (\gamma')^{-2} \sigma^{-2a}(x) +\cdots + 
(\gamma')^{-(f_m-1)}\sigma^{-a(f_m-1)}(x).\]
This mapping is $K(k)$-linear but not $K(k_m)$-linear.  For any $x \in W(k_m)^n$ we have
\[(\gamma')^{-1}\sigma^{-a} S_m(x) \equiv S_m(x) \mod{q^m},\] hence
\[ \gamma' S_m(x) \equiv \sigma^{-a} S_m(x) \mod{q^mW(k_m)^n}.\]
We will show that the image of $S_m$ contains $n$ vectors $x_1,x_2,\cdots,x_n \in W(k_m)^n$
that are linearly independent over $K(k_m)$ such that the matrix $B'_m = [x_1,\cdots,x_n]$ (whose
columns are the vectors $x_i$) is in $\GL(n, W(k_m))$. 

Let $\bar S_m:k_m^n \to k_m^n$ be the reduction of $S$ modulo the maximal ideal, and let
$\bar \gamma' $ denote the reduction of $\gamma'$ modulo $p$.  We claim that the image of
$\bar S_m$ contains $n$ linearly independent vectors $\bar x_1,\bar x_2,\cdots, \bar x_n \in k_m^n$.
For this, following a standard technique, it suffices to show that if $\alpha:k_m^n \to k_m$ is 
$k_m$-linear and if it vanishes on the image of $\bar S_m$ then it is zero. 

Suppose that $\alpha$ is such a linear map, that is, for any $\bar x \in k_m^n$ and any 
$c\in k_m$ 
\[
\alpha(\bar S_m(c\bar x)) = \sum_{i=0}^{f_m-1} \sigma^{-ai}(c) \alpha(\bar{\gamma}')^{-i} 
\sigma^{-ai}(\bar x))=0.
\]
The characters $\sigma^0,\sigma^{-a},\cdots,\sigma^{-a(f_m-1)}$ are linearly independent over $k_m$
and since $c$ is arbitrary, it follows that $\alpha(\bar{\gamma}'^{-i} \sigma^{-ai}(\bar x)) = 0$ for all
$i$.  But $\bar\gamma'$ and $\sigma^{-a}$ are invertible, hence $\alpha = 0$, which proves
the claim.

Choose lifts $x_1,x_2,\cdots,x_n \in W(k_m)^n$ of these vectors $\bar x_1,\cdots, \bar x_n$ and
let 
\[B'_m = [x_1,x_2,\cdots,x_n] \in M_{n\times n}(W(k_m))\] 
be the matrix whose columns are these  vectors.  The reduction $\bar B'_m$ lies in $\GL(n, k_m)$ 
hence $B'_m \in \GL(n, W(k_m))$.  Moreover,
\[ \gamma' B'_m \equiv \sigma^{-a} B'_m \mod{q^mW(k_m)^n}.\]
By replacing the sequence $\{B'_m\}$ with
 a subsequence if necessary (since $\GL(n,W(\bar k))$ is compact) we can arrange that the
elements $B'_m$ converge to some element $B' \in \GL(n,W(\bar k))$ which therefore satisfies
$\gamma' B' = \sigma^{-a} B'$.  

Finally, suppose that $b' \in \mathcal J(\gamma') \in W(\bar k)^n$ is any vector.  Then
 $\gamma' b' = \sigma^{-a} b'$.   Then
\[ \sigma^{-a}((B')^{-1}b') = (B')^{-1} (\gamma')^{-1} \gamma' b' = (B')^{-1} b'.\]
Consequently the vector $c = (B')^{-1}b'$ lies in $W(k)^n$ which implies $b' = B'c$ is a $W(k)$-linear
combination of the columns of $B'$.
Thus, the columns of the matrix $B' \in \GL(n, W(\bar k))$ 
form a $W(k)$-basis of $\mathcal J(\gamma')$. 
\end{proof}

\section{Involutions on the Witt vectors}\label{sec-Witt-involution}
Fix a finite field $\k$ of characteristic $p>0$ having $q = p^a = |\k|$
elements.  Fix an algebraic closure $\overline{\k}$ and let $W(\k)$, $W(\overline \k)$
denote the ring of (infinite) Witt vectors.  These are lattices within
the corresponding fraction fields, $K(\k)$ and $K(\overline \k).$  Let $W_0(\overline\k)$ be
the valuation ring in the maximal unramified extension $K_0(\overline\k)$ of $\QQ_p \subset K(\k)$.  
We may canonically identify $W(\overline\k)$ with the completion of $W_0(\overline\k)$.
Denote by $\pi:\overline \k \to \overline \k$ the Frobenius
$\pi(x) = x^q.$  It has a unique lift, which we also denote by
$\pi:W(\overline \k) \to W(\overline \k),$ and the cyclic group
$\langle \pi \rangle \cong \ZZ$ is dense in the Galois group
$G = \Gal(K(\overline \k)/K(\k) \cong\Gal(\overline \k/\k).$  If $L\supset \k$
is a finite extension, for simplicity we write $\Gal(L/\k)$ in place
of $\Gal(K(L)/K(\k))$ and we write $\Tr_{L/\k}$ for the trace $W(L) \to W(\k).$

\begin{prop}\label{prop-tau-Witt}
There exists a continuous $W(\k)$-linear mapping $\taub:W(\overline \k) \to W(\overline \k)$
such that
\begin{enumerate}
\item \label{item-1}$\taub^2 = I.$
\item \label{item-2}$\taub \pi = \pi^{-1} \taub.$
\item \label{item-3}
For any finite extension $E/\k,$ the mapping $\taub$ preserves
$W(E)\subset W(\overline \k).$
\item \label{item-4}
For any finite extension $L \contains E \contains \k$ the following diagrams commute
\begin{diagram}[size=2em]
W(L) & \rTo_{\taub} & W(L) &\qquad& W(L) & \rTo_{\taub} & W(L)\\
\dTo^{\Tr_{L/E}} && \dTo_{\Tr_{L/E}}\qquad && \uTo && \uTo \\
W(E) & \rTo_{\taub} & W(E) &\qquad& W(E) & \rTo_{\taub} & W(E)
\end{diagram}
\end{enumerate}
\end{prop}
Such an involution will be referred to as an {\em anti-algebraic involution of the
Witt vectors}.
\begin{proof}
Let $E\contains \k$ be a finite extension of degree $r.$
Recall that an element $\theta_E \in W(E)$ is a
{\em normal basis generator} if the collection $\theta_E, \pi\theta_E,\pi^2\theta_E,\cdots,
\pi^{r-1}\theta_E$ forms a basis of the lattice $W(E)$ over $W(\k).$  By simplifying and
extending the argument in \cite{Lenstra}, P.~Lundstr\"om showed \cite{Lundstrom} that
there exists a compatible collection $\{ \theta_E\}$ of normal basis generators of
$W(E)$ over $W(\k)$, where $E$ varies over all finite extensions of $\k,$ and where
``compatible''
means that $\Tr_{L/E}(\theta_L) = \theta_E$ for any finite extension
$L \contains E \contains \k.$
Let us fix, once and for all, such a collection of generators.  This is equivalent to fixing
a ``normal basis generator" $\theta$ of the free rank one module
\[ \underset{\underset{E}{\longleftarrow}}{\lim\ } W(E)\]
over the group ring
\[ W[[G]] = \underset{\underset{ E}\longleftarrow}{\lim\ } W(\k)[\Gal(E/\k)].\]

For each finite extension $E/\k$ define $\tau_E:W(E) \to W(E)$ by
\[ \tau_E\left(\sum_{i=0}^{r-1} a_i \pi^i\theta_E\right) :=
\sum_{i=0}^{r-1} a_i \pi^{-i} \theta_E =
\sum_{i=0}^{r-1} a_i \pi^{r-i} \theta_E\]
where $a_0,a_1,\cdots, a_{r-1} \in W(\k).$
Then $\tau_E^2=I$ and $\tau_E\pi = \pi^{-1}\tau_E.$  We refer to $\tau_E$ as an
anti-algebraic involution of $W(E).$  The mapping $\tau_E$ is an isometry (hence, continuous)
because it takes units to units.  To see this, suppose $v \in W(E)$ is a unit and set
$\tau_E(v) = p^ru$ where $u\in W(E)$ is a unit.  Then $v = \tau_E^2(v) = p^r\tau_E(u) \in p^rW(E)$
is a unit, hence $r=0$.

Next, we wish to show, for every finite extension $L \contains E \contains \k,$ that
$\tau_L|W(E) =\tau_E$ (so that $\tau_E$ is well defined) and that 
$\tau_E\circ\Tr_{L/E} = \Tr_{L/E}\circ \tau_L.$  We have an exact sequence
\begin{diagram}[size=2em]
1 &\rTo& \Gal(L/E) & \rTo & \Gal(L/\k) & \rTo & \Gal(E/\k) & \rTo & 1.
\end{diagram}
For each $h \in \Gal(E/\k)$ choose a lift $\hat h \in \Gal(L/\k)$ so that
\[\Gal(L/\k) = \{ \hat h g:\ h \in \Gal(E/\k), g \in \Gal(L/E)\}.\]
Let $ x = \sum_{h \in \Gal(E/\k)} a_h h\theta_E \in W(E)$ where
$a_h \in W(\k).$ Then
\begin{align*}
x &= \sum_{h \in \Gal(E/\k)} a_h h\sum_{g \in \Gal(L/E)} g\theta_L\\
&= \sum_{h \in \Gal(E/\k)} a_h \sum_{g \in \Gal(L/E)}\hat h g \theta_L\\
\intertext{so that}
\tau_L(x) &= \sum_{h \in \Gal(E/\k)} a_h \sum_{g \in \Gal(L/E)} \hat h^{-1} g^{-1} \theta_L\\
&= \sum_{h \in \Gal(E/\k)} a_h \hat h^{-1} \sum_{g \in \Gal(L/E)} g^{-1}\theta_L\\
&= \sum_{h \in \Gal(E/\k)}a_h h^{-1} \theta _E = \tau_E(x).
\end{align*}
To verify that $\tau_E\circ\Tr_{L/E}(x) = \Tr_{L/E}\circ\tau_L(x)$ it suffices to consider
basis vectors $x = \hat h g \theta_L$ where $g \in \Gal(L/E)$ and $h \in \Gal(E/\k).$
Then $\Tr_{L/E}(x) = h \theta_E$ and
\begin{align*}
\tTr_{L/E}(\tau_L(x)) &= \sum_{y \in \Gal(L/E)} y \hat h^{-1} g^{-1} \theta_L\\
&= \hat h^{-1} \sum_{z \in \Gal(L/E)}z\theta_L \\
&= h^{-1}\tTr(\theta_L) = \tau_E\tTr_{L/E}{(x)}. \end{align*}
It follows that the collection of involutions $\{\tau_E\}$ determines an involution 
\[ \bar\tau:W_0(\overline\k) \to W_0(\overline\k)\]
of the maximal unramified extension of $W(\k)$. It is a continuous isometry (so it takes 
units to units) and it satisfies the conditions (\ref{item-1}) to (\ref{item-4}).  
Therefore it extends uniquely and continuously to the completion $W(\bar\k)$. 
\end{proof}


Starting in Section \ref{sec-real-Dieudonne}) we make a choice, once and for all, of a $W(k)$-linear
anti-algebraic involution $\bar\tau:W(\bar k) \to W(\bar k)$, where 
$\k = k = \FF_{p^a}$ is the finite field that was  fixed in \S \ref{sec-ordinary}.

\quash{
 Then Frobenius reciprocity gives (cf.~\cite{Demazure})
\begin{align*}
M(A'') &\cong \hHom_{W(k)}\left(\hHom_{W(\bar k)}(T'_B\otimes W(\bar k), W(\bar k))^{\Gal},
W(k)\right)\\
&\cong \hHom_{W(\bar k)}\left( \hHom_{W(\bar k)}(T'_B\otimes W(\bar k), W(\bar k)), W(\bar k)
\right)^{\Gal}\\
&\cong \left(T'_B\otimes W(\bar k)\right)^{\Gal}.\end{align*}
Here, the action of $\Gal$ is given by: $\pi.(t_B'\otimes w) = F_Bt_B'\otimes \sigma^a(w),$
and $\mathcal F(t'_B\otimes w) = pt'_B\otimes\sigma(w)$ as in \cite{Demazure}.
This gives
\begin{align*}
M(A^{\prime\prime}) &= \left(\hHom_{\ZZ_p}(T^{\prime\prime}_A, \ZZ_p)\otimes W(\bar k)\right)^
{\Gal}\\
&= \hHom_{\ZZ_p}(T^{\prime\prime}_A,W(\bar k))^{\Gal} \end{align*}
where $(\pi.\psi)(t^{\prime\prime}_A) = \sigma^a\left(\psi(V_A(t^{\prime\prime}_A))\right).$
Therefore
this Dieudonn\'e module consists of $\ZZ_p$-linear homomorphism
$\psi:T^{\prime\prime}_A \to W(\bar k)$ such that
\begin{equation}\label{eqn-equivariant2}
\psi(t^{\prime\prime}_A) = \sigma^a(\psi(V_At^{\prime\prime}_A)).  \end{equation}
and $\mathcal F.\psi(t^{\prime\prime}_A) = p\sigma(\psi(t^{\prime\prime}_A).$
}


\end{document}